\documentclass[11pt,a4paper]{article}
\usepackage[english]{babel}
\usepackage[T1]{fontenc}

\usepackage[a4paper,margin=2.5cm]{geometry}
\usepackage{amsmath}
\usepackage{amssymb}
\usepackage{amsfonts}
\usepackage{enumitem}
\usepackage[all]{xy}
\usepackage{mathrsfs}
\usepackage{mathtools}
\usepackage{lmodern}
\usepackage{tikz}
\usepackage{graphicx} % Pour \includegraphics
\usepackage{float}    % Pour l'option [H]
\usepackage[normalem]{ulem}
\usepackage[pagebackref=true]{hyperref}

\newcommand{\N}{\ensuremath{\mathbb{N}}}
\newcommand{\R}{\ensuremath{\mathbb{R}}}

\newcommand{\Z}{\ensuremath{\mathbb{Z}}}

\newcommand{\w}{\ensuremath{\omega}}
\newcommand{\loc}{\ensuremath{\mathop{\rm loc}}}

\renewcommand{\leq}{\ensuremath{\leqslant}}
\renewcommand{\geq}{\ensuremath{\geqslant}}
\newcommand{\qed}{\hfill \vrule height6pt  width6pt depth0pt}

\newcommand{\dist}{\mathrm{dist}}

\newtheorem{thm}{Theorem}[section]
\newtheorem{defi}[thm]{Definition}

\newtheorem{prop}[thm]{Proposition}

\newtheorem{cor}[thm]{Corollary}
\newtheorem{lem}[thm]{Lemma}

\newtheorem{rem}[thm]{Remark}

\newenvironment{proof}[1][]{\noindent {\it Proof #1} : }{\hbox{~}\qed
\smallskip
}

\numberwithin{equation}{section}

\begin{document}
\selectlanguage{english}
\title{\bfseries{ Ornstein-Uhlenbeck semigroup maximal operator on weighted $L^p$ space}}
\date{September 2026}
\author{\bfseries{Christoph Kriegler and J\'er\'emie Moukambi}}

\maketitle

%%%%%%%%%%%%%%%%%%%%%%%%%%%%%%%%%%%%%%%%%%%%%%%%%%%%%%%%%%%%%%%%
%%%%%%%%%%%%%%%%%%%%%%%%%%%%%%%%%%%%%%%%%%%%%%%%%%%%%%%%%%%%%%%%
\begin{abstract}
Let $(\mathcal{H}_t)_{t \geq 0}$ denote the standard Ornstein-Uhlenbeck semigroup over ambient space $\R^d$ equipped with Gaussian measure $d\gamma(x) = e^{-|x|^2} dx$. Further let $\mathcal{H}^*(f)(x) = \sup_{t > 0} \mathcal{H}_t(|f|)(x)$ denote the associated maximal operator. Since $(\mathcal{H}_t)_{t \geq 0}$ is a Markovian semigroup over $(\R^d,d\gamma(x))$, it is well known that $\mathcal{H}^*$ is bounded on $L^p(\R^d,d\gamma(x))$ for all $1 < p \leq \infty$. We show that $\mathcal{H}^*$ is also bounded on the weighted space $L^p(\mathbb{R}^d,\omega(x)e^{-\frac{p}{2}|x|^2}dx)$ for $1 < p <\infty$, for weights $\omega : \R^d \to (0,\infty)$ belonging to a certain Muckenhoupt type class $A^\alpha_p$, where $\alpha \in [0,1)$. Essentially, a weight $\omega$ belongs to $A^\alpha_p$ if and only if  $\omega(Q) \omega^{-\frac{1}{p-1}}(Q)^{p-1} \leq C |Q|^p$ for all cubes $Q$ contained in any reference cube $N_\alpha(R_x) \subseteq \R^d$, where the well chosen family $(N_\alpha(R_x))_{x \in \R^d}$ covers $\R^d$, $N_\alpha(R_x)$ contains $x$ and the side length of $N_\alpha(R_x)$ is comparable to $1/\max(1,|x|)^\alpha$, with constants depending only on $d$. The class $A^\alpha_p$ contains the usual Muckenhoupt class.
We also use the weight class $A_p^{\loc}$ from \cite[Definition 2.2]{B} defined in a similar way as $A^\alpha_p$ but with dyadic subcubes contained in some $N(R)$ built over the Gaussian dyadic grid (see below), and we show that the local part $\sup_{t >0}\mathcal{H}_t(|f|\chi_{N(R_x)})(x)$ is bounded on $L^p(\mathbb{R}^d,\omega(x)e^{-\frac{p}{2}|x|^2}dx)$ if and only if $\omega \in A^{\loc}_p$.
\end{abstract}

%%%%%%%%%%%%%%%%%%%%%%%%%%%%%%%%%%%%%%%%%%%%%%%%%%%%%%%%%%%%%%%%
%%%%%%%%%%%%%%%%%%%%%%%%%%%%%%%%%%%%%%%%%%%%%%%%%%%%%%%%%%%%%%%%

\makeatletter
 \renewcommand{\@makefntext}[1]{#1}
 \makeatother
 \footnotetext{
 {\it Mathematics subject classification:}
 42A45, 42B25, 47A60.
%  42A45 Multipliers.
%  42B25 Maximal functions, Littlewood-Paley theory.
%  47A60 Functional calculus.
\\
{\it Key words}: Semigroups, maximal estimates, weighted $L^p$ spaces.}

 \tableofcontents

\section{Introduction}
This article is concerned with the Ornstein-Uhlenbeck semigroup.
It is well-known that it can be defined by means of the Mehler kernel and is given, for all $x \in \mathbb{R}^d$ and $t > 0$, by
\begin{align}\label{equ-intro-Mehler}
\mathcal{H}_t f(x) &= \int_{\mathbb{R}^d} M_t(x,y)f(y)dy\\
&= \int_{\mathbb{R}^d} \left(\frac{1}{\pi(1 -e^{-2t})}\right)^{\frac{d}{2}}\exp\left(-\frac{|e^{-t}x-y|^2}{1-e^{-2t}}\right)f(y) dy\\
&=\int_{\mathbb{R}^d}\left(\frac{1}{\pi(1-e^{-2t})}\right)^{\frac{d}{2}} \exp\left(\frac{1}{2} \frac{1}{e^t+1} |x+y|^2 - \frac{1}{2} \frac{1}{e^t-1} |x-y|^2\right) f(y) e^{-|y|^2}dy
\end{align}
We let $d\gamma(x) = \exp(-|x|^2)dx$ denote Gaussian measure normalized
throughout the article in this way.
The Ornstein-Uhlenbeck semigroup is contractive when acting on the
spaces $L^p(\R^d,d\gamma(x))$ for any $1 \leq p \leq \infty$ and any $t
\geq 0$ and self-adjoint on $L^2(\R^d,d\gamma(x))$.
Therefore, according to a result by Stein \cite[p.~73]{S2},
building on the Hopf-Dunford-Schwartz maximal theorem, the semigroup
satisfies the maximal estimate
\[ \| \sup_{t > 0} |\mathcal{H}_t f| \|_{L^p(\R^d,d\gamma(x))} \leq C_p
\|f\|_{L^p(\R^d,d\gamma(x))} \]
for $1 < p \leq \infty$.
The purpose of the article is to extend this maximal estimate to $L^p$
spaces with respect to other measures.
That is, we prove estimates in the space  $L^p(\mathbb{R}^d,\omega(x)e^{-\frac{p}{2}|x|^2}dx)$ for $1 < p < \infty$ where $\omega : \R^d \to (0,\infty)$ is a weight that obeys some sort of Muckenhoupt condition made precise below. Note that weighted $L^p$ estimates are well studied problems in harmonic analysis notably for Calder\'on-Zygmund operators \cite{GCRF,Hy,MW}. A somewhat parallel program for weighted $L^p$ estimates has been pursued for particular examples of elliptic semigroups \cite{AM1,AM2}. For maximal estimates on weighted $L^p$ space of an operator specifically related to the Ornstein-Uhlenbeck semigroup, we refer to \cite{HTV}.

When studying Gaussian harmonic analysis of the Ornstein-Uhlenbeck semigroup, one is often confronted with a decomposition of the space into regions like annuli centered at the origin resprctively cubes that have a decay of the width respectively side length comparable to $1/\max(1,|x|)$ where $|x|$ is the distance of the annulus to the origin. This principle is well-known and has been exploited in many references, for example \cite{B,MNP,GMMST1}. It is therefore no surprise that also for our particular question, cubes of side length decreasing comparably to (a power of) $1/\max(1,|x|)$ play an important r\^ole. We record the following definition.
\begin{defi}\label{AZ}
Let $1 < p < \infty$ and $\alpha \in [0,1)$.
\begin{enumerate}
\item For $m \in \Z$ and $\ell \geq 1$, let $\Delta_m = \{2^{-m}(x+[0,1)^d) :\: x \in \Z^d \}$, $L_0 = [-1,1)^d,\: L_\ell= [-2^\ell,2^\ell)^d \backslash [-2^{\ell-1},2^{\ell-1})^d$.
Then we define the Gaussian dyadic grid $\Delta^\gamma_0$ by
\[\Delta^\gamma_0 = \bigcup_{\ell \ge 0} \left\{ Q \in \Delta_\ell : \: Q \subseteq L_\ell \right\}. \]
For $R \in \Delta^\gamma_0$, we attach to $R$ a region $N(R)$, namely a cube of side length $4 l(R)$ which is a finite union of cubes of $\Delta^\gamma_0$, which admits $R$ as a dyadic subcube, and which contains every $R' \in \Delta^\gamma_0$ satisfying
\[d(R,R') := \inf \left\{|x-y|:\:x \in R \text{ and } y\in R' \right\} < l(R)\]
where $l(R)$ is the side length of $R$ and $d(R,R')$ is the Euclidean
distance between $R$ and $R'$. Such a cube exists but is in general not unique, so that one of them is fixed once and for all for every $R \in \Delta^\gamma_0$; see Definition \ref{simL}.
\item
For $n \geq 1$, we define $\delta^\alpha_n$ to be the unique nonnegative integer $k$ such that $2^{-k} \leq n^{-\alpha} < 2^{-k+1}$. Also, let us designate by $\Delta^\alpha_n$ the collection of cubes given below
\[ \Delta^\alpha_n = \left\{m + [0,2^{-\delta^\alpha_n})^d : \:
2^{\delta^\alpha_n}m \in \Z^d,\: m + [0,2^{-\delta^\alpha_n})^d
\subseteq [-(n+1),n+1)^d \backslash [-n,n)^d \right\}. \]
We also let $\delta^\alpha_0 = - 1$ and $\Delta^\alpha_0 = \{[-1,1)^d \}$. Then we define the $\alpha$ power grid $\Delta^\alpha$ by
\[ \Delta^\alpha = \bigcup_{n \geq 0} \Delta^\alpha_n. \]
Consider a cube $Q \in \Delta^\alpha$. We attach to $Q$ a region $N_\alpha(Q)$, namely a cube of side length $4 l(Q)$ which is a finite union of cubes of $\Delta^\alpha$, which admits $Q$ as a dyadic subcube, and which contains every $Q' \in \Delta^\alpha$ satisfying
\[\rho(Q,Q') = \inf\{\|x-y\|_\infty : \: x \in Q, \: y \in Q'\} < l(Q) .\]
Such a cube exists but is in general not unique, so that one of them is fixed once and for all for every $Q \in \Delta^\alpha$; see Lemma \ref{lem-N-alpha-exists} and Definition \ref{OKOKJHJDHJD}.

    We can see an example in Figure \ref{grille}, with the cube $Q$ colored in red and one admissible region $N_{\frac{1}{2}}(Q)$ in yellow.

\begin{figure}[H]
    \centering
\begin{tikzpicture}
    \draw(0,0)grid[step=2](14.2,14.2);
    \draw[fill=black!17](0,0)rectangle(2,2);
\draw[fill=yellow!100](7,7)rectangle(11,11);
\draw[fill=red!100](9,9)rectangle(10,10);
   %%%%Couleur des couches

 \draw(0,0)grid[red,step=2](14,14);
   
    \draw (0,0) node[below left]{0} node{};
     \draw (2,0) node[below]{1} node{};
     \draw (4,0) node[below]{2} node{};
      \draw (6,0) node[below]{3} node{};
       \draw (8,0) node[below]{4} node{};
        \draw (10,0) node[below]{5} node{};
     \draw (12,0) node[below]{6} node{};
     \draw (14,0) node[below]{7} node{};
      %\draw (6,0) node[below]{3} node{};
      % \draw (8,0) node[below]{4} node{};
       % \draw (10,0) node[below]{5} node{};
       %  \draw (12,0) node[below]{6} node{};
         % \draw (14,0) node[below]{7} node{};
         %  \draw (16,0) node[below]{8} node{};
         %   \draw (18,0) node[below]{9} node{};
      %\draw (0,1) node[left]{1} node{$\bullet$};
      %%%%%%Couche 1/2
 
    %\draw (0,0) node[left]{0} node{};
     \draw (0,2) node[left]{1} node{};
     \draw (0,4) node[left]{2} node{};
      \draw (0,6) node[left]{3} node{};
       \draw (0,8) node[left]{4} node{};
        \draw (0,10) node[left]{5} node{};
     \draw (0,12) node[left]{6} node{};
        \draw (0,14) node[left]{7} node{};
       % \draw (0,5) node[left]{5} node{};
       %  \draw (0,6) node[left]{6} node{};
         % \draw (0,7) node[left]{7} node{};
          % \draw (0,8) node[left]{8} node{};
          %  \draw (0,9) node[left]{9} node{};

      \draw[thick] (2,0)grid[step=2](14,14);
      \draw[thick] (0,2)grid[step=2](14,14);
      %%%%%cadrant haut gauche
      %  \draw[thick] (0,0)grid[step=0.5](-5,5);
      %    \draw[thick] (0,1)grid[step=0.5](-5,5);
        %%%\cadrant bas gauche
    %  \draw[thick] (0,0)grid[step=0.5](-5,-5);
    %  \draw[thick] (0,0)grid[step=0.5](5,5);
      %%%%%%%\cadrant bas droite
   %   \draw[thick] (1,0)grid[step=0.5](5,5);
    %  \draw[thick] (0,0)grid[step=0.5](5,5);

    %%%%%couche $1/4$

\draw[thick] (4,0)grid[step=1](14,14);
      \draw[thick] (0,4)grid[step=1](14,14);
      %%%%%cadrant haut gauche
        %\draw[thick] (3,3)grid[step=0.25](5,5);
        %  \draw[thick] (3,3)grid[step=0.25](5,5);
        %%%\cadrant bas gauche
      %\draw[thick] (3,0)grid[step=0.25](5,5);
      %\draw[thick] (0,0)grid[step=0.25](5,5);
      %%%%%%%\cadrant bas droite
     % \draw[thick] (3,0)grid[step=0.25](5,5);
      %\draw[thick] (0,0)grid[step=0.25](5,5);

    \draw[thick] (10,0)grid[step=0.5](14,14);
      \draw[thick] (0,10)grid[step=0.5](14,14);

\end{tikzpicture}

  \caption{Grid $\Delta^\alpha$ for $\alpha = \frac{1}{2}$}
    \label{grille}
\end{figure}
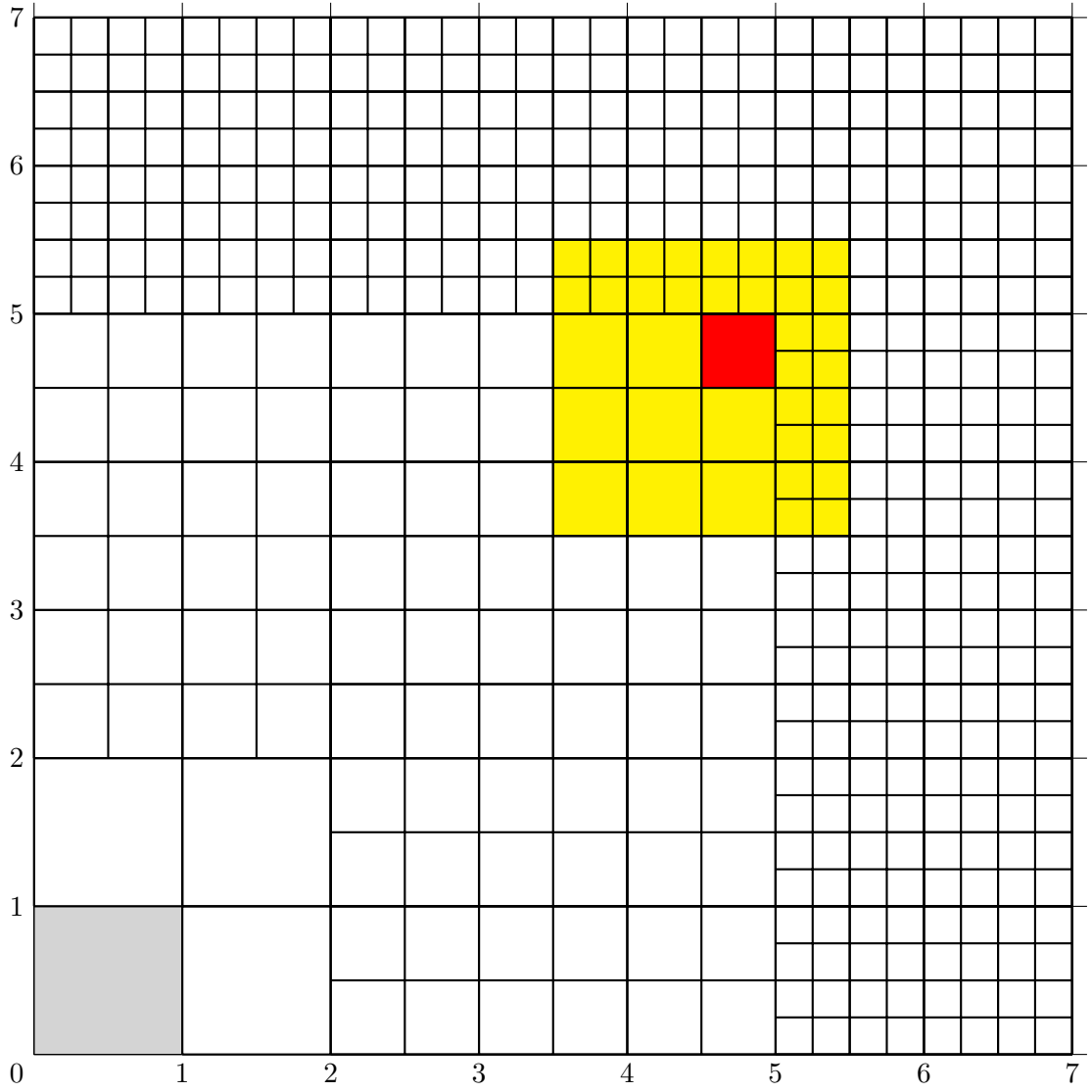

\item Let $R \subseteq \R^d$. We define the weight class $A_p(R)$ to consist of all weights $\omega : R\to (0,\infty)$ such that there exists a constant $C > 0$ such that \[ \left( \frac{1}{|Q|} \int_Q \omega(x) dx \right)^{\frac1p} \left(\frac{1}{|Q|} \int_Q \omega^{-\frac{1}{p-1}}(x)dx
\right)^{\frac{p-1}{p}} \le C \]
for all cubes $Q$ contained in $R$ with sides parallel to the axes, with $|Q| = \int_Q dx$. The smallest constant $C$ is called the characteristic of $\omega$ and will be denoted by $[\omega]_{A_p(R)}$.
If $R$ is a cube and the above condition is imposed only for those subcubes $Q$ of $R$ that arise from the repeated bisection of $R$ into $2^d$ congruent cubes -- the dyadic subcubes of $R$, see Section \ref{sec-2} -- we obtain a larger class, denoted $A_p^{dy}(R)$, with $[\w]_{A_p^{dy}(R)} \leq [\w]_{A_p(R)}.$

Then we define
\begin{equation*}
    A^{\loc}_p = \{ \omega : \R^d \to (0,\infty) : \:
\omega|_{N(R)} \in A_p^{dy}(N(R)) \: \forall \: R \in \Delta^\gamma_0 , \:
[\omega]_{A^{\loc}_p} = \sup_{R \in \Delta^\gamma_0} [\omega]_{A_p^{dy}(N(R))}
< \infty \},
\end{equation*}
and for $\alpha \in [0,1)$, we define

\begin{equation*}
    A^\alpha_p = \{ \omega : \mathbb{R}^d \to (0,\infty) : \:
\omega|_{N_\alpha(R)} \in A_p(N_\alpha(R)) \; \forall \: R \in
\Delta^\alpha, \: [\omega]_{A^\alpha_p} = \sup_{R \in \Delta^\alpha}
[\omega]_{A_p(N_\alpha(R))} <\infty \}.
\end{equation*}
\end{enumerate}
\end{defi}
Note that in particular the standard Muckenhoupt class $A_p$ is a
subclass of $A^{\loc}_p$ and $A^\alpha_p$. Moreover, the classes are nested according to $A^{\alpha_1}_p \subseteq A^{\alpha_2}_p\subseteq A^{\loc}_p$ for $0 \leq \alpha_1 \leq \alpha_2 < 1$.

When studying the Ornstein-Uhlenbeck semigroup maximal operator, one is often confronted  with a decomposition of this operator into a local and a global part. That is, when $x \in \R^d$ such that $x \in R_x \in \Delta^\gamma_0$, we put
\begin{align*}
\mathcal{H}^*(f)(x) & := \sup_{t > 0} \mathcal{H}_t(|f|)(x), \\
\mathcal{H}^*_{\loc}(f)(x) & := \sup_{t > 0} \mathcal{H}_t(|f|
\chi_{N(R_x)})(x), \\
\intertext{and}
\mathcal{H}^*_{far}(f)(x) & := \sup_{t > 0} \mathcal{H}_t(|f| \chi_{\R^d
\backslash N(R_x)})(x).
\end{align*}
It is a simple matter of fact that $\mathcal{H}^*$ is bounded on a
weighted $L^p$ space if and only if both $\mathcal{H}^*_{\loc}$ and
$\mathcal{H}^*_{far}$ are bounded on the same space. Indeed, one has

\begin{equation*}
    \max \{ \mathcal{H}^*_{\loc}(f)(x), \:
\mathcal{H}^*_{far}(f)(x) \} \leq \mathcal{H}^*(f)(x) \leq
\mathcal{H}^*_{\loc}(f)(x) + \mathcal{H}^*_{far}(f)(x)
\end{equation*}
 for any $f$ and any $x$. Then our first main result which characterizes for which weights exactly the local part of the maximal operator is bounded, reads as follows.
\begin{thm}\label{thm-intro-1}
Let $1 < p < \infty$ and $\omega : \R^d \to (0,\infty)$ a weight.
The following are equivalent.
\begin{enumerate}
\item $\omega$ belongs to $A^{\loc}_p$.
\item $\mathcal{H}^*_{\loc}$ is bounded on  $L^p(\mathbb{R}^d,\omega(x)dx)$, that is, there is a constant $C > 0$ such that
\[\|\mathcal{H}^*_{\loc}(f)\|_{L^p(\R^d,\omega(x)dx)} \leq
C \|f\|_{L^p(\R^d,\omega(x)dx)} . \]
\end{enumerate}
\end{thm}

We refer to Theorem \ref{result1} below for further equivalent conditions.
Next in order to obtain a similar control for the global part
$\mathcal{H}^*_{far}$ on weighted $L^p$ space,
the class $A^{\loc}_p$ is too large. Indeed, it turns out that any weight $\omega_\beta(x) = e^{\beta |x|^2}$ with $\beta \in \R$ belongs to $A^{\loc}_p$, albeit a bound of the full maximal operator $\mathcal{H}^*$ on $L^p(\R^d,v(x)dx)$ implies that the full mass $\int_{\R^d}v(x)dx$ is finite. Indeed, let $f = \chi_{B(0,1)}$ and fix $x \in \R^d$. For every $y \in \R^d$ one has $M_t(x,y) \to \pi^{-\frac{d}{2}}e^{-|y|^2}$ as $t \to \infty$, and $M_t(x,y) \leq \left(\pi(1-e^{-2})\right)^{-\frac{d}{2}}$ for $t \geq 1$, so that dominated convergence in the integral $\mathcal{H}_tf(x) = \int_{\R^d}M_t(x,y)f(y)dy$ over the set $B(0,1)$ of finite measure yields
\[ \mathcal{H}_tf(x) \xrightarrow[\;t \to \infty\;]{} \pi^{-\frac{d}{2}}\int_{\R^d} f(y)e^{-|y|^2}dy = \pi^{-\frac{d}{2}}\int_{\R^d} f \, d\gamma =: c_f > 0 . \]
Hence $\mathcal{H}^*f \geq c_f$ on all of $\R^d$ and therefore $\|\mathcal{H}^*f\|^p_{L^p(v)} \geq c_f^p\int_{\R^d}v(x)dx$, while $\|f\|_{L^p(v)} < \infty$ since by Theorem \ref{thm-intro-1}, $v$ belongs to $A^{\loc}_p$ and hence is locally integrable. Therefore for $\beta \geq p/2$, the weight $\omega_\beta$ belongs to $A^{\loc}_p$ yet $\mathcal{H}^*$ (and thus also $\mathcal{H}^*_{far}$) is unbounded on $L^p(\R^d,\omega_\beta(x)e^{-\frac{p}{2}|x|^2}dx)$. Then to be able to still obtain a positive result, we use the class $A^\alpha_p$ with $\alpha \in [0,1)$, where the decay of the size of cubes on which the Muckenhoupt condition is imposed to hold for a given weight is slightly less rapid than for weights of the class $A^{\loc}_p$.

Our second main result, for the full Ornstein-Uhlenbeck semigroup
maximal operator, reads as follows.

\begin{thm} \label{thm-intro-2}
Let $1 < p < \infty$ and $\alpha \in [0,1)$.
Let $\omega : \R^d \to (0,\infty)$ be a weight belonging to $A^\alpha_p$.
Then the maximal operator $\mathcal{H}^*$ is bounded on $L^p(\mathbb{R}^d,\omega(x)e^{-\frac{p}{2}|x|^2}dx)$, that is, there is a constant $C > 0$ such that
\[\|\mathcal{H}^*(f)\|_{L^p(\mathbb{R}^d,\omega(x)e^{-\frac{p}{2}|x|^2}dx)} \leq C
\|f\|_{L^p(\mathbb{R}^d,\omega(x)e^{-\frac{p}{2}|x|^2}dx)} . \]
\end{thm}
 The proof uses the following ideas. First, one decomposes $\mathcal{H}^*$ again into a local and a global part, but this time according to the modified grid $\Delta^\alpha$. That is, one introduces
\[\mathcal{H}^*_{\loc,\alpha}(f)(x) = \sup_{t > 0} \mathcal{H}_t(|f|
\chi_{N_\alpha(R_x)})(x)\]
and
\[\mathcal{H}^*_{far,\alpha}(f)(x) = \sup_{t > 0} \mathcal{H}_t(|f|
\chi_{\R^d \backslash N_\alpha(R_x)})(x).\]
The local part can be controlled on weighted $L^p$ space thanks to an
upper bound against the heat maximal operator
\[ \mathcal{H}^*_{\loc,\alpha}(f)(x) \lesssim \sup_{0<s<1}
T^\Delta_s(|f|\chi_{N_\alpha(R_x)}e^{-\frac12 |\cdot|^2}) e^{+\frac12
|x|^2},\]
where
\[ T^\Delta_s f(x) = \frac{1}{(4\pi s)^{d/2}} \int_{\R^d}
e^{-|x-y|^2/(4s)}f(y)dy \]
is the standard heat semigroup; this normalization is used throughout the article.
For the heat maximal operator in turn, bounds on weighted $L^p$ with
Muckenhoupt weights are well-known in the literature, and we use them by
extending a Muckenhoupt weight defined on the local region
$N_\alpha(R_x)$ to full space $\R^d$ without altering the characteristic.

In order to estimate the global part $\mathcal{H}^*_{far,\alpha}$, we
split the Euclidean space $\R^d = G_+(x) \cup G_-(x)$ over which the
integration \eqref{equ-intro-Mehler} in the variable $y$ is performed,
into two regions $G_+(x)$ where $y$ and $x$ in \eqref{equ-intro-Mehler}
point into the same half space (i.e. have positive scalar product), and
$G_-(x)$ where $y$ and $x$ point into opposite half spaces (negative
scalar product).
As a matter of fact, once we have needed for the estimate of the local
part $\mathcal{H}^*_{\loc,\alpha}$ a weight that belongs to the
Muckenhoupt class over a cube $R$ whose sidelength decays more slowly
than $1/\max(1,d(R,0))$ (class $A^\alpha_p$), the same weight gives a bounded global
part by means of quite crude estimates.
Namely, first, the supremum over time $t > 0$ in the definition of
$\mathcal{H}^*_{far,\alpha}$ which is outside the integration
\eqref{equ-intro-Mehler} can be bounded above by the integral over the
supremum. Then, second, the point of $t > 0$ where this supremum is localized can be spotted approximately by a calculation of the derivative of the integral kernel. Here we use the splitting $\R^d = G_+(x) \cup G_-(x)$ in order to approximate this supremum according to the estimates $|x+y|^2 \geq |x|^2 + |y|^2$ for $y \in G_+(x)$ and $|x-y|^2 \geq |x|^2 + |y|^2$ for $y \in G_-(x)$, where $|x+y|^2$ and $|x-y|^2$ appear in the exponent of the
Mehler kernel \eqref{equ-intro-Mehler}. Finally, third, the integral over the supremum estimated in this way is then estimated by H\"older's inequality together with a control of the weight $\omega$ integrated over two distant cubes $R$ and $Q$ of the grid $\Delta^\alpha$, namely $\int_Q \omega(x)dx \leq C^{r(Q,R)}\int_R\omega(x)dx$ with $C$ depending only on $[\omega]_{A^\alpha_p}$, $p$ and $d$, and with $r(Q,R)$ the number of ``steps'' between $Q$ and $R$, a notion that we shall introduce in Definition \ref{defi-steps}. This is essentially a consequence of the fact that a Muckenhoupt weight is in particular doubling, applied once for each step of a shortest chain from $Q$ to $R$; the resulting exponential growth in $r(Q,R)$ is then absorbed by the decay of the Mehler kernel.

In contrast to the first Theorem \ref{thm-intro-1} where we get an
\emph{equivalent} characterization of the weighted $L^p$ bound, in
Theorem \ref{thm-intro-2}, the condition that $\omega$ belongs to
$A^\alpha_p$ is only \emph{sufficient} for the weighted $L^p$ bound.
Indeed, we check in Lemma \ref{exep} and Proposition \ref{Proposorijfnfjj} that for $\gamma > 0$ the function $\omega_\gamma(x) =
\exp(\|x\|_\infty^\gamma)$ belongs to $A^\alpha_p$ if and only if $\gamma \leq
\alpha + 1$.
Thus, in particular, for two different values $0 \leq \alpha <
\alpha_2 < 1$, the inclusion $A^{\alpha}_p \subseteq A^{\alpha_2}_p$
is strict and thus, a weight $\omega \in A^{\alpha_2}_p \backslash
A^{\alpha}_p$ gives a counterexample for the necessity in Theorem
\ref{thm-intro-2}.

We end this introduction with an overview of the sections.
In Section \ref{sec-2}, we introduce and study the local weight class $A^{\loc}_p$ that already appears in \cite[Definition 2.2]{B}.
After defining the Gaussian dyadic grid in Definition \ref{DBKDJHDKJBDJKBDKJBDJKBDJKBDKBDKBKBDKJDBKJDBKD} and $A^{\loc}_p$ in Definition \ref{BANACHALOUGLU} itself (as a recall of Definition \ref{AZ}), we show some permanence and doubling properties in Subsection \ref{subsec-2-2} that will be useful in the sequel.
Then in Subsections \ref{subsec-2-3} resp. \ref{subsec-2-4}, we estimate the local part of the maximal operator restricted to times $t \geq 1$ resp. $0 < t < 1$.
While the former is based on a brute estimate of the Mehler kernel, the latter builds upon the estimate of the harmonic oscillator semigroup from \cite{B}.
In Subsection \ref{subsec-2-5}, we resume these estimates and extend hereby in Theorem \ref{result1} the result \cite[Theorem B]{B} on the weighted $L^p$ control of the harmonic oscillator maximal operator (over full time $t \geq 0$) by allowing a supplementary exponential factor $e^{td}$, where $d$ denotes the dimension of $\R^d$  (cf. Remark \ref{POBLOOAZDD}).
In Section \ref{sec-3}, we consider weights belonging to the standard Muckenhoupt class $A_p$.
In Subsection \ref{subsec-3-1}, we introduce a parameter change in the Mehler kernel of the Ornstein-Uhlenbeck operator that yields a pointwise estimate $\mathcal{H}^*f(x) \lesssim e^{\frac{|x|^2}{2}} \sup_{0 < s < 1} T^{\Delta}_s\left(|f|e^{-\frac{|\cdot|^2}{2}}\right)(x)$.
Here $T^{\Delta}_s$ stands for the usual heat semigroup operator.
By well-known weighted $L^p$ estimates of the maximal operator of the heat semigroup, we use this estimate in Subsection \ref{subsec-3-2} to prove a control of the Ornstein-Uhlenbeck maximal operator.
In the subsequent Section \ref{sec-4} we introduce the weight class $A^\alpha_p$ designed for the $L^p$ estimate of the Ornstein-Uhlenbeck maximal operator, recalling hereby Definition \ref{AZ}.
Its definition is based on a grid $\Delta^\alpha$ that resembles the Gaussian dyadic grid $\Delta^\gamma_0$ but with cubes whose sidelength decays more slowly, that is, comparably to $1/\max(1,|x|)^\alpha$, where $x$ is any point of the cube.
The first Subsection \ref{subsec-4-1} contains in particular Lemma \ref{oh} on the distance of two cubes of $\Delta^\alpha$ measured in the (integer) length of the shortest chain of $\Delta^\alpha$ cubes that are of distance $0$ one after another, what we call the number of \emph{steps}.
This Lemma is crucial for the proof of the main result Theorem \ref{result2} of our article.
Subsection \ref{subsec-4-2} contains, in part in parallel to Subsection \ref{subsec-2-2}, doubling and other basic properties of $A^\alpha_p$ weights that we shall need later.
The final Section \ref{sec-5} is entirely devoted to a detailed proof of the second main result Theorem \ref{result2}.
It uses once again the parameter change of the Mehler kernel already encountered in Subsection \ref{subsec-3-1}.
Moreover, the space of the integration variable $y$ in the Mehler kernel is decomposed $\R^d = G_-(x) \cup G_+(x)$ into two disjoint half-spaces.
Then Subsection \ref{subsec-5-1} contains the proof of the $G_-(x)$ part of this second main result.
It consists of a major part where $x$ and $y$ are distant, and a residual part where $x$ and $y$ are close to the origin.
For the former, the observation of Lemma \ref{oh} on the steps together with the doubling property of $A^\alpha_p$ weights comes into play.
For the latter, we rather use an extension lemma of weights defined on a cube in $\R^d$ to the whole of $\R^d$ together with the pointwise control of the Ornstein-Uhlenbeck semigroup against the heat semigroup from Subsection \ref{subsec-3-1}.
Subsection \ref{subsec-5-2} is concerned with the proof of the $G_+(x)$ part.
Here we use once again the parameter change of the Mehler kernel, together with a localization of the time $s_{\max}$ at which the new Mehler kernel is maximal, in dependence on $x$ and $y$.
Once again, we distinguish the case where $x$ and $y$ are close to each other (extension of weights and pointwise control against heat semigroup), and the contrary (number of steps and doubling property).

\section{The $A_p^{\loc}$ class and local part of the maximal operator}
\label{sec-2}
Let
\begin{equation*}
    \mathcal{L}=-\Delta+|x|^2
\end{equation*}
be the harmonic oscillator operator acting on $L^2(\R^d,dx)$, and let, for every $x\in \R^d$,
\begin{equation*}
    \mathcal{T}^*f(x):=  \sup_{t>0}e^{-t\mathcal{L}}(|f|)(x)
\end{equation*}
be its associated maximal operator. We further denote by $L$ the generator of the Ornstein--Uhlenbeck semigroup, so that
\begin{equation}\label{equ-OU-generator}
\mathcal{H}_t = e^{-tL} \qquad (t \geq 0);
\end{equation}
the operators $L$ and $\mathcal{L}$ should not be confused. For a fixed cube $R$ in the grid $\Delta_0^\gamma$, we denote the restriction of the maximal operator of $e^{-t\mathcal{L}}$ to the region $N(R)$ by
\begin{equation*}
    \mathcal{T}_{\loc}^*f(x):=\sup_{t>0}e^{-t\mathcal{L}}(|f|\chi_{N(R_x)})(x)
\end{equation*}

The weight class $A_p^{\loc}$ was introduced in \cite{B}. It is strictly larger than the Muckenhoupt class $A_p$ \cite[Proposition C]{B} and is necessary and sufficient for the weighted $L^p$ boundedness of the maximal operator $  \mathcal{T}_{\loc}^*$ that is,

\begin{equation*}
     || \mathcal{T}_{\loc}^*f||_{L^p(\w(x)dx)}\leq C_\w ||f||_{L^p(\w(x)dx)}\;\;\; \text{if and only if }\;\;\;\w\in A_p^{\loc},
\end{equation*}
when $1<p<\infty$, see Theorem $B$ in \cite{B}.\

 In this section, we study the weighted estimates, with weights in $A_p^{\loc}$, of the maximal operator $\mathcal{H}^*$ associated with the Ornstein-Uhlenbeck semigroup in the region $N(R)$.

 \subsection{Local region $N(R)$}
%Soit un cube de $\mathbb{R}^d$, $Q_0:=[a_1,a_1+l(Q_0))\times...\times [a_d,a_d+l(Q_0))$, avec $\{a_1,...,a_d\}\subset \mathbb{R}^d$ et $l(Q_0)$ désignant la longueur de coté du cube $Q_0$.  Ce cube peut encore être divisé en $2^d$ cubes disjoints, tous de longueur de coté $\frac{1}{2}l(Q_0)$. Qui eux-mêmes peuvent être divisés en $2^d$ cubes disjoints de longueur $\frac{1}{2^2}l(Q_0)$, et ainsi de suite à l'infini.

Let $Q_0:=[a_1,a_1+l(Q_0))\times...\times [a_d,a_d+l(Q_0))\subset \R^d$, with $a_1,\ldots,a_d\in \mathbb{R}$ and $l(Q_0)$ denoting the side length of the cube $Q_0$. This cube can be further divided into $2^d$ disjoint cubes, all of side length $\frac{1}{2}l(Q_0)$, which themselves can be divided into $2^d$ disjoint cubes of side length $\frac{1}{2^2}l(Q_0)$, and so on ad infinitum.

\begin{defi}\textbf{(Dyadic subcube)}
    %Un cube $Q$ de $\mathbb{R}^d$ obtenu de la décomposition précédente de $Q_0$ est dit sous-cube dyadique de $Q_0$.
 A cube $Q$ of $\mathbb{R}^d$ obtained from the previous decomposition of $Q_0$ is called a dyadic subcube of $Q_0$.   
   
\end{defi}
%\vspace{0.2cm}
 %Notons que $Q_0$ peut être aussi vu comme un sous-cube dyadique de lui-même.

Note that $Q_0$ can also be seen as a dyadic subcube of itself.

   \begin{defi}\textbf{(Standard dyadic grid)} For $m\in \mathbb{Z}$, let $\Delta_m$ be the collection of cubes   
       \begin{equation*}
           \Delta_m:=\{2^{-m}(x+[0,1)^d)\;:\;x\in \mathbb{Z}^d\}.
       \end{equation*}
      Then the standard dyadic grid is the union \begin{equation*}
           \Delta:=\bigcup_{m\in\mathbb{Z}}\Delta_m.
       \end{equation*}%, que nous désignons ici par $\Delta$.
   \end{defi}

      % Dans notre étude, nous  allons nous intéresser à une version extraite de cette grille.
      For our study, we will focus on the case where $m\geq 0$.
      We recall from Definition \ref{AZ} in the introduction the following.
\begin{defi}\textbf{(Layers)}
 Let $\ell\in \mathbb{N}$. We define 
       \begin{equation*}
           L_0:=[-1,1)^d,\;\;\;\;\;L_\ell:=[-2^\ell,2^\ell)^d\setminus[-2^{\ell-1},2^{\ell-1})^d.
       \end{equation*}
     For any $\ell\geq 0$, $L_\ell$ will be called a layer.   
\end{defi}

\begin{defi}\label{DBKDJHDKJBDJKBDKJBDJKBDJKBDKBDKBKBDKJDBKJDBKD}
   The grid $\Delta^\gamma_0$ is given by
    \begin{equation*}
        \Delta_0^\gamma:=\bigcup_{\ell\geq 0}\{Q\in \Delta_\ell\;:\; Q\subseteq L_\ell\}
    \end{equation*}
\end{defi}

Figure \ref{02}, taken from \cite{MNP}, shows a two-dimensional representation of the grid $\Delta^\gamma_0$.

\begin{figure}[H]
     \centering
     \includegraphics[width = 14cm]{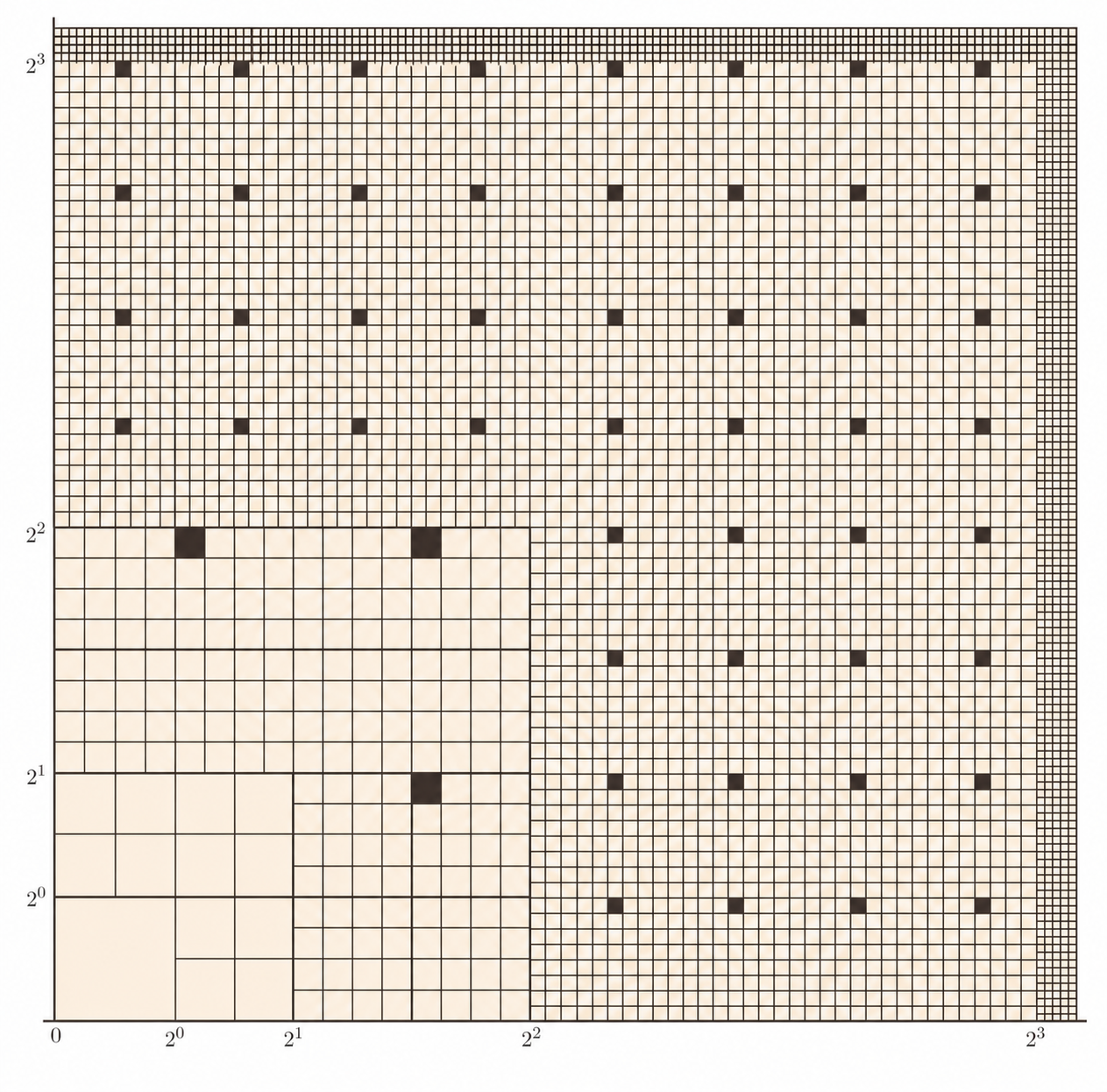} 
    \caption{Subdivision of the layers $L_0, L_1, L_2,$ and $L_3$ into cubes of $\Delta^\gamma_0$} 
    \label{02} 
\end{figure}

\begin{defi}\label{simL}\textbf{(Region $N(R)$)}   Let $R$ be a cube of $\Delta^\gamma_0$. For every such $R$ we fix, once and for all, a cube $N(R)$ of side length $4\,l(R)$ with the following three properties:
\begin{enumerate}
\item $N(R)$ is a finite union of cubes of $\Delta^\gamma_0$;
\item $R' \subseteq N(R)$ for every $R' \in \Delta^\gamma_0$ with
\begin{equation}
d(R,R'):=\inf\{|x-y|\;:\;x\in R\; \text{and} \; y\in R'\} <l(R),
\end{equation}
where $l(R)$ is the side length of $R$ and $d(R,R')$ is the Euclidean distance between $R$ and $R'$;
\item $R$ is a dyadic subcube of $N(R)$.
\end{enumerate}
Such a cube exists --- this is the content of \cite[Definition 2.2]{B}%, and it can also be obtained by the argument of Lemma \ref{lem-N-alpha-exists} below, which is simpler here since every cube of $\Delta^\gamma_0$ has side length at most $1$
--- but it need not be unique, which is why the choice is fixed beforehand. We denote by $F(R):=\mathbb{R}^d\setminus N(R)$ the complement of the region $N(R)$.
\end{defi}
Note that $l(R)=2^{-j(R)}$, where $j(R)$ denotes the unique integer representing the index of the layer in which the cube $R$ is located.

%In the remainder of the document, we will designate by $F(R):=\mathbb{R}^d\setminus N(R)$ the complement of $N(R)$.

       \begin{figure}[H]
     \centering
     \includegraphics[width = 14cm]{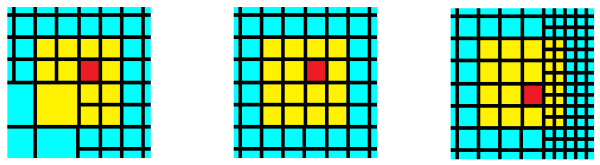} 
    \caption{In all three cases, the cube $R$ is colored red, the region $N(R)$ in yellow and the complement $F(R)$ in blue \cite{B}.} 
    \label{01} 
\end{figure}

\subsection{Definition of the weight class $A_p^{\loc}$}
\label{subsec-2-2}

We recall from the introduction the following definition.

   \begin{defi}\label{BANACHALOUGLU}
Let $1<p<\infty$ and let $\omega$ be a weight on $\mathbb{R}^d$. We say that $\omega$ belongs to $A_p^{\loc}$ if there exists a constant $C>0$ such that  
\begin{equation}\label{alo1}
    [\omega]_{A_p^{dy}(N(R))}:=\sup_{Q\subseteq N(R)} \Big(\frac{1}{|Q|}\int_{Q}\omega(x)dx\Big)^\frac{1}{p} \Big(\frac{1}{|Q|}\int_{Q}\omega^{-\frac{1}{p-1}}(x)dx\Big)^\frac{p-1}{p}\leq C
\end{equation}
for all $R\in \Delta_0^\gamma$, where the supremum is taken over all dyadic subcubes $Q\subseteq N(R)$ with sides parallel to the coordinate axes. The smallest constant $C$ satisfying equation $(\ref{alo1})$ is given by 
\begin{equation*}
    [\omega]_{A^{\loc}_p}:=\sup_{R\in\Delta_0^\gamma}[\omega]_{A_p^{dy}(N(R))}.
\end{equation*}
\end{defi}

%where $Q$ is a subcube (with sides parallel to the axes) of some larger cube $R$.

%\begin{definition}
%    Let $1<p<\infty$.  Let $w$ be a weight on $\R^d$, we will say that $w$ belongs to $A_p^{\loc}$ if there exists a constant $C>0$ such that
%    \begin{equation}\label{lm}
%      w(R')^{\frac{1}{p}}w^{-\frac{p'}{p}}(R')^{\frac{1}{p'}}\leq C|R'|,
%    \end{equation}
%   for any dyadic subcube $R'\subseteq N(R)$ with $R\in\Delta^\gamma_0$.   
%\end{definition}

%We will designate by $[w]_{A_p^{\loc}}$ the smallest constant verifying the equation (\ref{lm}).

\begin{rem}\label{Jabu} Let $1<p,p'<\infty$ satisfy $\frac{1}{p}+\frac{1}{p'}=1$. If $\w\in A_p^{\loc}$, then $\w^{-\frac{p'}{p}}\in A_{p'}^{\loc}$. Indeed, by setting $\sigma=\w^{-\frac{p'}{p}}$, for $Q$ a dyadic subcube contained in $N(R)$, we observe that  
\begin{align*}
    \sigma(Q)^\frac{1}{p'}\sigma^{-\frac{p}{p'}}(Q)^\frac{1}{p}&=\w^{-\frac{p'}{p}}(Q)^\frac{1}{p'}\w(Q)^\frac{1}{p}\\
    &\leq [\w]_{A_p^{\loc}}|Q|.
\end{align*}
From this fact, we also deduce that $[\w]_{A_p^{\loc}}=[\w^{-\frac{p'}{p}}]_{A_{p'}^{\loc}}$.
\end{rem}

On the local region, we observe the following elementary fact important for our considerations.

\begin{prop}\label{ruji}
  Let $R\in\Delta^\gamma_0$ and let $x,y \in N(R)$. Then there exists a constant $c>0$, depending only on $d$, such that 
\begin{equation}\label{rqd}
    |x|^2-c\leq |y|^2\leq |x|^2+c.
\end{equation}
We then deduce for any real $k$ that
\begin{equation}\label{(OLLLM}
    e^{k|x|^2}\lesssim e^{k|y|^2}\lesssim e^{k|x|^2}.
\end{equation}
\end{prop}

%\begin{propo}\label{20}
%Pour $x,y \in N(R)$, alors il %existe une constante $c>0$ telle %que     
%\begin{equation*}
%    |x|^2-c\leq |y|^2\leq |x|^2+c
%\end{equation*}
%et alors $e^{k|x|^2}\simeq %e^{k|y|^2}$ pour toute constante  %$k$ dans $ \mathbb{R}$.
%\end{propo}

\begin{proof}
  Write $j:=j(R)$, so that $l(R)=2^{-j}$ and $N(R)$ is a cube of side length $4l(R)=2^{2-j}$. Since $x,y\in N(R)$, we have
\begin{equation}\label{equ-ruji-1}
    |x-y|\leq 4l(R)\sqrt{d}=2^{2-j}\sqrt{d}\leq 2^2\sqrt{d}.
\end{equation}
Moreover $R\subseteq L_j$ gives $\|z\|_\infty<2^{j}$ for every $z\in R$, and every point of $N(R)$ is at $\ell^\infty$-distance at most $4l(R)$ from $R$, whence
\begin{equation}\label{equ-ruji-2}
    |y|\leq \left(2^{j}+2^{2-j}\right)\sqrt{d}\leq \left(2^{j}+2^2\right)\sqrt{d}.
\end{equation}
We then have by Cauchy-Schwarz that
    \begin{align*}
        |x|^2&=|x-y+y|^2\\
        &= |x-y|^2+2\langle x-y,y\rangle+|y|^2\\
        &\leq |x-y|^2+2|x-y||y|+|y|^2\\
        &\leq 2^4d+ 2\cdot 2^{2-j}\left(2^{j}+2^2\right)d+|y|^2\\
        &= 2^4 d + \left(2^3+2^{5-j}\right)d + |y|^2\\
        &\leq c+|y|^2,
    \end{align*}
 with $c=2^6d$, where we used \eqref{equ-ruji-1}, \eqref{equ-ruji-2} and $j\geq 0$. This implies that $|x|^2-c\leq |y|^2$. Exchanging the roles of $x$ and $y$, which is possible since both lie in $N(R)$, we obtain in the same way
\begin{align*}
    |y|^2&=|y-x+x|^2\\
    &\leq c+|x|^2.
\end{align*}
Therefore, equation (\ref{rqd}) is verified. Moreover, for any fixed real $k$, we have respectively that

\begin{equation*}
    e^{k|x|^2}\lesssim e^{k|y|^2},\;\;\;\;\text{and}\;\;\;\;e^{k|y|^2}\lesssim e^{k|x|^2}.
\end{equation*}
\end{proof}

Proposition \ref{ruji} leads to the following two consequences:
\begin{cor}\label{125}
   Let $\w$ be a weight on $\mathbb{R}^d$ and let $c$ be a real constant. Then
    \begin{equation*}
        \w\in A^{\loc}_p \iff  e^{c|.|^2} \w \in A^{\loc}_p.
    \end{equation*}
\end{cor}

%\begin{propo}\label{125}
%    Soient  $w$ un poids sur $\mathbb{R}^d$ et $c$ une constante réel, alors $w$ est un poids de $A^{\loc}_p$  si et seulement si $we^{c|.|^2}$ est un poids de $A^{\loc}_p$. 
%\end{propo}
\begin{proof}  \,   For $R\in \Delta_0^\gamma$, let $Q$ be a cube contained in $N(R)$ and $x_0$ the center of $N(R)$. Then if $\w\in A_p^{\loc}$, by Proposition \ref{ruji}, we observe that
\begin{align*}
\int_{Q}\w(x)e^{c|x|^2}dx\Big(\int_{Q}\w^{-\frac{p'}{p}}(x)e^{-\frac{p'}{p}c|x|^2}dx\Big)^\frac{p}{p'}&\lesssim   e^{c|x_0|^2}e^{-c|x_0|^2} \w(Q)\w^{-\frac{p'}{p}}(Q)^{\frac{p}{p'}}\\
        &\leq [\w]^p_{A^p_{\loc}}|Q|^p,  
\end{align*}
Thus, $e^{c|.|^2}\w\in A_p^{\loc}$.\

  On the other hand, if $e^{c|.|^2}\w\in A_p^{\loc}$, by setting $v=e^{c|.|^2}\w$, we have by Proposition \ref{ruji} that
\begin{align*}
    \int_{Q}\w(x)dx\Big(\int_{Q}\w^{-\frac{p'}{p}}(x)dx\Big)^\frac{p}{p'}&=\int_{Q}v(x)e^{-c|x|^2}dx\Big(\int_{Q}v^{-\frac{p'}{p}}(x)e^{\frac{p'}{p}c|x|^2}dx\Big)^\frac{p}{p'}\\
    &\lesssim \int_{Q}v(x)dx\Big(\int_{Q}v^{-\frac{p'}{p}}(x)dx\Big)^\frac{p}{p'}\\
    &\leq [v]^p_{A_p^{\loc}}|Q|^p.
\end{align*}
Consequently, $\w\in A_p^{\loc}$.

\end{proof}
\begin{cor}\label{126}
  Let $\w\in A_p^{\loc}$  and $c$ a real constant, then
   \begin{equation*}
       [\w e^{c|.|^2}]^p_{A_p^{\loc}}\simeq_c [\w]^p_{A_p^{\loc}}.
   \end{equation*}
\end{cor}
\begin{proof} \,  Consequence of Proposition \ref{ruji}, Corollary \ref{125} and the fact that
\begin{equation*}
   \w(Q)\w^{-\frac{p'}{p}}(Q)^{\frac{p}{p'}} \lesssim\int_{Q}\w(x)e^{c|x|^2}dx\Big(\int_{Q}\w^{-\frac{p'}{p}}(x)e^{-\frac{p'}{p}c|x|^2}dx\Big)^\frac{p}{p'} \leq [\w e^{c|x|^2}]^p_{A_p^{\loc}} |Q|^p.
\end{equation*}

\end{proof}
\begin{rem}\label{dgjdqururhfn} Given that $1 \in A_p^{\loc}$, by Proposition \ref{ruji} and Corollary $\ref{125}$, the function $x \mapsto e^{c|x|^2}$, with $c \in \R$, is a weight in $A_p^{\loc}$.
\end{rem}

\begin{rem}\label{rem-sup-norm-loc} Proposition \ref{ruji} and Corollary \ref{125} remain valid, with the same proofs, when the Euclidean norm $|\cdot|$ is replaced throughout by the norm $\|\cdot\|_\infty$. Indeed, writing $j:=j(R)$, for $x,y\in N(R)$ we have $\|x-y\|_\infty\leq 4l(R)=2^{2-j}$ and $\|y\|_\infty\leq 2^{j}+2^{2-j}$, so that
\begin{equation*}
    \|x\|_\infty^2\leq\|y\|_\infty^2+2\|y\|_\infty\|x-y\|_\infty+\|x-y\|_\infty^2\leq\|y\|_\infty^2+2\left(2^2+2^{4-2j}\right)+2^4\leq\|y\|_\infty^2+2^6,
\end{equation*}
and symmetrically. In particular, taking $\w=1$, the function $x\mapsto e^{c\|x\|^2_\infty}$ is a weight in $A_p^{\loc}$ for every $c\in\R$.
\end{rem}

%Soit $x_0$ le centre de $N(R)$, toujours par la proposition \ref{ruji}, nous avons 
%   \begin{align*}
%      [we^{c|x|^2}]^p_{A_p^{\loc}}&=\sup_{x\in Q\subseteq N(R)}\frac{1}{|Q|^p} \int_{Q}w(x)e^{c|x|^2}dx\Big(\int_{Q}w^{-\frac{p'}{p}}(x)e^{-\frac{p'}{p}c|x|^2}dx\Big)^\frac{p}{p'}\\
%      &\simeq\sup_{x\in Q\subseteq N(R)}e^{c|x_0|^2}e^{-c|x_0|^2}\frac{1}{|Q|^p} \int_{Q}w(x)dx\Big(\int_{Q}w^{-\frac{p'}{p}}(x)dx\Big)^\frac{p}{p'}\\
%      &=[w]^p_{A_p^{\loc}}
%   \end{align*}

%%%We have also observed the  results which will be useful for what follows.

\begin{prop}\label{lol}
  Let $\w\in A_p^{\loc}$. Then for any dyadic subcube $Q\subseteq N(R)$ we have  
    \begin{equation*}
        \w(N(R))\leq [\w]^p_{A_p^{\loc}}\Big(\frac{|N(R)|}{|Q|}\Big)^p\w(Q)
    \end{equation*}
\end{prop}

\begin{proof}\,
Let $f:=\chi_Q$. By Hölder's inequality,
\begin{align*}\label{tantio}
    \frac{1}{|N(R)|}\int_{N(R)}f(x)d(x)&= \frac{1}{|N(R)|}\int_{N(R)}f(x)\w^{\frac{1}{p}}(x)\w^{-\frac{1}{p}}(x)dx\\
    &\leq \frac{1}{|N(R)|}\Big(\int_{N(R)}|f(x)|^p\w(x)dx\Big)^{\frac{1}{p}} \Big(\int_{N(R)}\w^{-\frac{1}{p-1}}(x)dx\Big)^{\frac{p-1}{p}}\\
    &= \frac{1}{|N(R)|}\Big(\int_{Q}\w(x)dx\Big)^{\frac{1}{p}} \Big(\int_{N(R)}\w^{-\frac{1}{p-1}}(x)dx\Big)^{\frac{p-1}{p}}\\
    &\leq \frac{1}{|N(R)|}\Big(\int_{Q}\w(x)dx\Big)^{\frac{1}{p}} \w(N(R))^{-\frac{1}{p}}|N(R)|[\w]_{A_p^{\loc}},
\end{align*}
since
\begin{equation*}
     \Big(\int_{N(R)}\w^{-\frac{1}{p-1}}(x)dx\Big)^{\frac{p-1}{p}}\leq \w(N(R))^{-\frac{1}{p}}|N(R)|[\w]_{A_p^{\loc}}.
\end{equation*}

Thus, we deduce that 
\begin{equation*}
    \w(N(R))^\frac{1}{p}\leq [\w]_{A_p^{\loc}}\frac{|N(R)|}{|Q|}\w(Q)^\frac{1}{p}
\end{equation*}

\end{proof}
\begin{rem}\label{0002} Replacing the cube $Q$ with the cube $R$, we observe that
\begin{enumerate}
    \item  \begin{equation*}
         \w(N(R))\leq (4^d[\w]_{A_p^{\loc}})^p\w(R).
    \end{equation*}
    \item By Remark \ref{Jabu}, since for $\omega \in A_p^{\loc}$, $\omega^{-\frac{p'}{p}} \in A_{p'}^{\loc}$ and $[\omega]_{A_p^{\loc}} = [\omega^{-\frac{p'}{p}}]_{A_{p'}^{\loc}}$, it follows from the previous proposition that:
    \begin{equation}
      \w^{-\frac{p'}{p}}(N(R))^{\frac{p}{p'}}\leq (4^d[\w]_{A_p^{\loc}})^p\w^{-\frac{p'}{p}}(R)^{\frac{p}{p'}},
\end{equation}
    
\end{enumerate}

\end{rem}

 \begin{lem}\label{cardinalité}
  Let $Q\in \Delta^\gamma_0$. Then the cardinality
  \begin{equation*}
        |\{R\in \Delta^\gamma_0\;:\;N(R)\cap Q\ne \emptyset\}|\leq C(d):=2^{12d}.
        \end{equation*}
    \end{lem}
\begin{proof} Write $\ell:=j(Q)$ and $\ell':=j(R)$, so that $l(Q)=2^{-\ell}$ and $l(R)=2^{-\ell'}$. Assume $N(R)\cap Q\ne\emptyset$. Since $R\subseteq N(R)$ and $N(R)$ is a cube of side length $4l(R)\leq 4$, every point of $N(R)$ is at $\ell^\infty$-distance at most $4l(R)$ from $R$, so that
\begin{equation}\label{equ-card-1}
    \rho(Q,R)\leq 4l(R)\leq 4 .
\end{equation}
Every $u\in Q$ satisfies $\|u\|_\infty \leq 2^{\ell}$, and $\|u\|_\infty\geq 2^{\ell-1}$ if $\ell\geq 1$; likewise for $R$ and $\ell'$. Hence \eqref{equ-card-1} forces $2^{\ell'-1}\leq 2^{\ell}+4$ when $\ell'\geq 1$, and $2^{\ell-1}\leq 2^{\ell'}+4$ when $\ell\geq 1$. If $\min(\ell,\ell')\geq 3$, the first of these gives $2^{\ell'-1}\leq 2^{\ell+1}$ and the second $2^{\ell-1}\leq 2^{\ell'+1}$, that is $|\ell-\ell'|\leq 2$. If $\min(\ell,\ell')\leq 2$, say $\ell\leq 2$, then $2^{\ell'-1}\leq 2^2+4\leq 2^3$ and therefore $\ell'\leq 4$; the case $\ell'\leq 2$ is symmetric. In all cases
\begin{equation}\label{equ-card-2}
    |\ell-\ell'|\leq 4, \qquad\text{hence}\qquad 2^{-4}l(Q)\leq l(R)\leq 2^{4}l(Q).
\end{equation}
Consequently $R$ has side length at most $2^4 l(Q)$ and lies, by \eqref{equ-card-1} and \eqref{equ-card-2}, within $\ell^\infty$-distance $4l(R)\leq 2^6l(Q)$ of $Q$; therefore $R$ is contained in the cube concentric with $Q$ of side length $\left(1+2\cdot 2^6+2\cdot 2^4\right)l(Q)\leq 2^{8}l(Q)$. The cubes of $\Delta^\gamma_0$ being pairwise disjoint with $|R|\geq \left(2^{-4}l(Q)\right)^d$, their number is at most $\left(2^{8}/2^{-4}\right)^d=2^{12d}$.
% en notant par $2N(Q)$ un cube contenant $N(R)$ et de longueur de coté deux fois celle de $N(R)$.
\end{proof}

Now, before presenting our study of the weighted $L^p$ estimates for the operator $\mathcal{H}_{\loc}^*$, let us first recall the following elements:
\begin{defi}\label{local muck}
For a cube $Q_0$ in $\mathbb{R}^d$, we denote by $A_p(Q_0)$ the class of weights satisfying the Muckenhoupt condition (in the sense of equation \eqref{alo1}) for every cube $Q \subseteq Q_0$ with sides parallel to the coordinate axes, and by $A_p^{dy}(Q_0)$ the class obtained by imposing that condition only for the dyadic subcubes $Q_0$.
Thus $A_p(Q_0) \subseteq A_p^{dy}(Q_0)$ and $[\w]_{A_p^{dy}(Q_0)} \leq [\w]_{A_p(Q_0)}$.
\end{defi}

The following lemma is proved in \cite[Theorem 2.2.1]{Ooi} (see also \cite[Lemma 1.1]{Ryc}) for weights $\w$ that are defined on all of $\R^d$ and that obey a Muckenhoupt condition on all cubes of sidelength $\leq \rho$.
The same proof of first reflecting the weight along hyperplanes and then periodic continuation over all of $\R^d$ works in our case that the weight $\w$ is only defined on some cube $Q_0 \subseteq \R^d$.
Indeed, an inspection of the proof shows that only values of $\w$ on $Q_0$ are used in order to check the Muckenhoupt condition of $\overline{\w}$ (the second case there, $l(Q) < \rho$ is a bit misleadingly spelled out when passing from integrals with $\overline{\w}$ to those with $\w$).
A variant of Lemma \ref{extension} with a weight $\w$ defined only on some cube $Q_0$ and a Muckenhoupt condition on dyadic sub/supercubes of $Q_0$ in both the hypotheses and the conclusion is proved in \cite[Lemma 2.1]{B}.

\begin{lem}(\cite[Theorem 2.2.1]{Ooi})\label{extension}
Let $1 < p < \infty$ and let $Q_0 \subseteq \R^d$ be a cube of sidelength $\rho := l(Q_0)$.
Let $\w$ be a weight on $Q_0$ with $[\w]_{A_p(Q_0)} < \infty$.
Then there exists a weight $\overline{\w} \in A_p(\R^d)$ that coincides with $\w$ on $Q_0$ and satisfies 
\[[\overline{\w}]_{A_p} \leq \max(2^{dp},(2\rho+1)^{dp}) [\w]_{A_p(Q_0)}.\]
\end{lem}

\subsection{Estimate of the local part, case $t\geq 1$}
\label{subsec-2-3}

Recall that, for $f \in L^p(d\gamma)$, $p\geq 1$, $t>0$, and $x\in \mathbb{R}^d$, the Ornstein-Uhlenbeck semigroup is given by:

\begin{equation}\label{74COMBO}
\mathcal{H}_tf(x)= \int_{\mathbb{R}^d}M_t(x,y)f(y)dy
\end{equation}

with $M_t$ the Mehler kernel given, for all $t>0$ and all $x,y \in \mathbb{R}^d$, by 
 $$
   M_t(x,y)=\Big(\frac{1}{\pi(1-e^{-2t})}\Big)^\frac{d}{2}\exp\Big(-\frac{|e^{-t}x-y|^2}{(1-e^{-2t})}\Big).
$$

In this section, we show that for $1<p<\infty$ and  a weight $\w \in A^{\loc}_p$, we have the maximal estimate
\begin{equation*}
    \| \sup_{t\geq 1}   \mathcal{H}_t(|f|.\chi_{N(R_x)})\|_{L^p(\w(x)e^{-p\frac{|x|^2}{2}}dx)} \lesssim [\w]_{A^{\loc}_p} \| f \|_{L^p(\w(x)e^{-p\frac{|x|^2}{2}}dx)}
\end{equation*}
(see Prop \ref{OOll}). Our proof is based on the integral representation of the Ornstein-Uhlenbeck semigroup $\mathcal{H}_t$ via the Mehler kernel (Equation (\ref{74COMBO})), which is then majorized pointwise against a simpler term, using here that $t \geq 1$. After that step, the rest of the proof consists essentially of writing $
f(y) = \left(f(y) \w(y)^{\frac1p}\right) \w(y)^{-\frac1p}$, using Hölder's inequality and the control $\w(Q) \w^{-p'/p}(Q)^{p/p'} \leq [\w]^p_{A^{\loc}_p} |Q|^p$, due to the Muckenhoupt condition.

%\subsubsection{}
%\subsubsection{}
%Dans tout le document la mesure gaussienne est donnée par $d\gamma(x)=\pi^{-\frac{d}{2}}e^{-|x|^2}dx$. Ici, nous allons donner une estimation $L^p$ pondérée, avec $1<p<\infty$, de l'opérateur maximal local Ornstein-Uhlenbeck pour des temps plus grand que $1$. 

%Throughout the paper the Gaussian measure is given by $d\gamma(x)=\pi^{-\frac{d}{2}}e^{-|x|^2}dx$. Here we will give a weighted $L^p$ estimate, with $1<p<\infty$, of the local maximal operator of the Ornstein-Uhlenbeck semigroup for times larger than $1$.\\

Let us begin by presenting the following result:

\begin{lem}\label{Minamoto1}
 Let $x\in \mathbb{R}^d$, let $R_x$ be the unique cube of $\Delta^\gamma_0$ containing $x$, and let $N(R_x)$ be the region given in Definition \ref{simL}. For $t\geq 1$, the Ornstein--Uhlenbeck semigroup restricted to $N(R_x)$ satisfies the following equation:
    \begin{equation}\label{Minamoto2}
        \mathcal{H}_t(|f|.\chi_{N(R_x)})(x)\leq K\int_{N(R_x)}\exp\Big(-c_1|x|^2\Big)|f(y)|dy,
    \end{equation}
    where $K=\Big(\frac{1}{\pi 
 (1-e^{-2})}\Big)^\frac{d}{2}\exp\Big(\frac{2^3 d}{(1-e^{-2})}\Big)$ and $c_1=(1-e^{-1})^2$.
\end{lem}

\begin{proof} For every $x\in \R^d$, let $R_x$ denote the cube of $\Delta^\gamma_0$ containing $x$. We then restrict the Ornstein--Uhlenbeck semigroup to $N(R_x)$ as follows:
%Pour $R_x$ l'unique cube de $\Delta^\gamma_0$ contenant $x$ et $N(R_x)$ la region donné à la definition \ref{simL}, on peut observer que le semigroupe d'Ornstein-Uhlenbeck restrient à $N(R_x)$ est donné par Thus, for times $t\geq 1$, we define the local Ornstein-Uhlenbeck semigroup in the following manner
\begin{align*}
 \mathcal{H}_t(f.\chi_{N(R_x)})(x)&=\int_{N(R_x)}M_t(x,y)f(y)dy\\
    &= \frac{1}{\pi^\frac{d}{2}}\int_{N(R_x)}\frac{1}{s(t)^\frac{d}{2}}\exp\Big(-\frac{|e^{-t}x-y|^2}{s(t)}\Big)f(y)dy,
\end{align*}
with $s(t)=1-e^{-2t}$. For $t\geq 1$, we observe that $s(1)\leq s(t)$. Therefore, the last display yields
\begin{equation}\label{snn}
     \mathcal{H}_t(|f|.\chi_{N(R_x)})(x)\leq \Big(\frac{1}{\pi 
 s(1)}\Big)^\frac{d}{2}\int_{N(R_x)}\exp\Big(-\frac{|e^{-t}x-y|^2}{s(t)}\Big)|f(y)|dy.
\end{equation}

%Transformons maintenant l'exponentielle sous l'intégrale de l'équation (\ref{snn}). Nous remarquons  que  
Let us now estimate the term $\exp\Big(-\frac{|e^{-t}x-y|^2}{s(t)}\Big)$ in equation (\ref{snn}). We first observe that
\begin{align*}
    |e^{-t}x-y|^2&=|(e^{-t}-1)x+x-y|^2\\
    &=|(e^{-t}-1)x|^2+2(e^{-t}-1)\langle x,x-y\rangle+|x-y|^2\\
    &\geq (1-e^{-t})^2|x|^2-2(1-e^{-t})|x||x-y|,
\end{align*}
%et comme $x\in R$ et $y\in N(R)$ alors,  $|x|\leq 2^{j(R)}\sqrt{d}$ et $|x-y|\leq 2^22^{-j(R)}\sqrt{d}$, ainsi nous avons :
However, $|x|\leq 2^{j(R)}\sqrt{d}$ and, for $y\in N(R_x)$, $|x-y|\leq 2^22^{-j(R)}\sqrt{d}$. Therefore,
\begin{align*}
    -|e^{-t}x-y|^2&\leq -(1-e^{-t})^2|x|^2+2(1-e^{-t})\times 2^2 d.\\
    &\leq-(1-e^{-1})^2|x|^2+2 \cdot 2^2 d\\
    &\leq -(1-e^{-1})^2|x|^2+2^3 d.
\end{align*}
%De  plus, dans le cas $t\geq 1$, $s(1)\leq s(t)\leq s(\infty)=1$, donc à l'aide de ce qui précède  que, nous aboutissons à  
Furthermore, $s(1)\leq s(t)\leq 1$ since $t\geq 1$, so that, using the above, we arrive at
\begin{align*}
    \frac{-|e^{-t}x-y|^2}{s(t)}&\leq -\frac{(1-e^{-1})^2}{s(t)}|x|^2+\frac{2^3 d}{s(t)}\\
    &\leq -(1-e^{-1})^2|x|^2+\frac{2^3d}{s(1)}\\
    &=-(1-e^{-1})^2|x|^2+\frac{2^3 d}{s(1)}\\
    &=-c_1|x|^2+c_2 d,
\end{align*}
with $c_1=(1-e^{-1})^2$ and $c_2=\frac{2^3}{s(1)}$. We then deduce that equation (\ref{snn}) becomes
\begin{align}
    \mathcal{H}_t(|f|.\chi_{N(R_x)})(x)&\leq \Big(\frac{1}{\pi 
 s(1)}\Big)^\frac{d}{2}\int_{N(R_x)}\exp\Big(-\frac{|e^{-t}x-y|^2}{s(t)}\Big)|f(y)|dy\\
   &\leq\Big(\frac{1}{\pi 
 s(1)}\Big)^\frac{d}{2}  \int_{N(R_x)}\exp\Big(-c_1|x|^2+c_2 d\Big)|f(y)|dy\\
   &=K\int_{N(R_x)}\exp\Big(-c_1|x|^2\Big)|f(y)|dy\label{noio}
    \end{align}
with $K=\Big(\frac{1}{\pi 
 s(1)}\Big)^\frac{d}{2}\exp\Big(\frac{2^3 d}{s(1)}\Big)$.

\end{proof}

\begin{prop}\label{OOll}
 Let $1<p<\infty$ and let $\w \in A_p^{\loc}$. Then:
  \begin{equation*}
      ||\sup_{t\geq 1}  \mathcal{H}_t(f.\chi_{N(R_x)})||^p_{L^p(\w(x)e^{-p\frac{|x|^2}{2}}dx)}\lesssim [\w]^p_{A^p_{\loc}}||f||^p_{ L^p(\w(x)e^{-p\frac{|x|^2}{2}}dx)}
  \end{equation*}

\end{prop}
\begin{proof}\,
    %Par l'inégalité de Hölder, et l'équation (\ref{noio}), nous avons  que :
We obtain, by Lemma \ref{Minamoto1} and Hölder's inequality, that:
\begin{align*}
\begin{split}
    ||\sup_{t\geq 1}  \mathcal{H}_t(|f|.\chi_{N(R_x)})||^p_{ L^p(\w(x)e^{-p\frac{|x|^2}{2}}dx)}&= \int_{\mathbb{R}^d} \sup_{t\geq 1}   \mathcal{H}_t(|f|.\chi_{N(R_x)})(x)^p\w(x)e^{-p\frac{|x|^2}{2}}dx\\
    &=\sum_{R\in \Delta_0^\gamma}\int_{R} \sup_{t\geq 1}   \mathcal{H}_t(|f|.\chi_{N(R_x)})(x)^p\w(x)e^{-p\frac{|x|^2}{2}}dx\\
    &\leq K^p\sum_{R\in \Delta_0^\gamma}\int_{R}\Big(\int_{N(R)}e^{-c_1|x|^2}|f(y)|dy\Big)^p\w(x)e^{-p\frac{|x|^2}{2}}dx\\
    %&=\sum_{R\in \Delta_0^\gamma}\int_{R}e^{-c_1p|x|^2}\Big(\int_{N(R)}|f(y)|v^\frac{1}{p}(y)v^{-\frac{1}{p}}(y)e^{-\frac{|y|^2}{2}}e^{\frac{|y|^2}{2}}dy\Big)^p\\
    % &\quad \times v(x)e^{-p\frac{|x|^2}{2}}dx\\
    &\leq \sum_{R\in \Delta_0^\gamma}\int_{R}e^{-p(c_1+\frac{1}{2})|x|^2}\w(x)dx\Big(\int_{N(R)}\w^{-\frac{p'}{p}}(y)e^{\frac{p'}{2}|y|^2}dy\Big)^\frac{p}{p'}\\
    & \quad\times\int_{N(R)}|f(y)|^p\w(y)e^{-p\frac{|y|^2}{2}}dy\\
    &\leq \sum_{R\in \Delta_0^\gamma}\int_{N(R)}e^{-p\frac{|x|^2}{2}}\w(x)dx\Big(\int_{N(R)}\w^{-\frac{p'}{p}}(y)e^{\frac{p'}{2}|y|^2}dy\Big)^\frac{p}{p'}\\
    & \quad\times\int_{N(R)}|f(y)|^p\w(y)e^{-p\frac{|y|^2}{2}}dy\\
    &\leq \sum_{R\in \Delta_0^\gamma}[\w e^{-p\frac{|\;.\;|^2}{2}}]^p_{A_p^{\loc}}|N(R)|^p\int_{N(R)}|f(y)|^p\w(y)e^{-p\frac{|y|^2}{2}}dy\\
    &\lesssim 2^{2pd}[\w]^p_{A^{\loc}_p}\sum_{R\in \Delta_0^\gamma}\int_{N(R)}|f(y)|^p\w(y)e^{-p\frac{|y|^2}{2}}dy,
\end{split}
\end{align*}
 where the last line uses $|N(R)|=\left(4l(R)\right)^d\leq 2^{2d}$, which holds since $j(R)\geq 0$ and hence $l(R)\leq 1$, together with $[\w e^{-p\frac{|\;.\;|^2}{2}}]^p_{A_p^{\loc}}\simeq[\w]^p_{A_p^{\loc}}$ from Corollary \ref{126}. Moreover,
\begin{align*}
    \sum_{R\in \Delta_0^\gamma}\int_{N(R)}|f(y)|^p\w(y)e^{-p\frac{|y|^2}{2}}dy&=\sum_{R\in \Delta_0^\gamma}\int_{N(R)\cap \mathbb{R}^d}|f(y)|^p\w(y)e^{-p\frac{|y|^2}{2}}dy\\
    &=\sum_{R'\in \Delta_0^\gamma}\sum_{R\in \Delta_0^\gamma}\int_{N(R)\cap R'}|f(y)|^p\w(y)e^{-p\frac{|y|^2}{2}}dy\\
    &=  \sum_{R'\in \Delta_0^\gamma}|\{R\in \Delta^\gamma_0\;:\;N(R)\cap R'\ne \emptyset\}|\int_{R'}|f(y)|^p\w(y)e^{-p\frac{|y|^2}{2}}dy\\
&\leq C(d)||f||^p_{L^p(\w(x)e^{-p\frac{|x|^2}{2}}dx)},
    \end{align*}
by Lemma \ref{cardinalité}. Hence, 
 \begin{equation*}
      || \sup_{t\geq 1}   \mathcal{H}_t(|f|.\chi_{N(R_x)})||^p_{L^p(\w(x)e^{-p\frac{|x|^2}{2}}dx)}\lesssim [\w]^p_{A_p^{\loc}}||f||^p_{ L^p(\w(x)e^{-p\frac{|x|^2}{2}}dx)}.
  \end{equation*}
\end{proof}

Despite the simple nature of the proof, it will allow us in Remark \ref{POBLOOAZDD} to strengthen Bailey's maximal estimate of the local part of the harmonic oscillator semigroup \cite{B},

\begin{equation*}
    \| \sup_{t > 0} e^{-t \mathcal{L}}(f \chi_{N(R_x)}) \|_{L^p(\w(x)dx)} \leq C \|f\|_{L^p(\w(x)dx)}
\end{equation*}

into the estimate
\begin{equation*}
\| \sup_{t > 0} e^{td} e^{-t \mathcal{L}}(f \chi_{N(R_x)}) \|_{L^p(\w(x)dx)} \leq C\|f\|_{L^p(\w(x)dx)}.
\end{equation*}

\subsection{Estimate of the local part, case $t\leq 1$}
\label{subsec-2-4}

In the previous section, we obtained an estimate of the local maximal operator of the Ornstein-Uhlenbeck semigroup for times $t \geq 1$ with weights in the class $A_p^{\loc}$. We will now study the case $0 < t < 1$. Our approach will rely on the relationship between the Ornstein–Uhlenbeck semigroup $\mathcal{H}_t$ and the harmonic oscillator semigroup $e^{-t\mathcal{L}}$ (see Proposition $3.3$ of \cite{AT}).\\

%We begin our study by giving a version of $L^p(\omega(x)e^{-p\frac{|x|^2}{2}}dx)$ in $L^p(\omega(x)dx)$, with weights in $A_p^{\loc}$, of the isometry $U$ given in equation (\ref{YUKIUJDD}).

\begin{prop}\label{GDHDUEEHH}Let $1<p<\infty$ and let $\omega$ be a weight. Then the mapping
  $$\begin{array}{ccccc}
U_p & : & L^p(\omega(x)e^{-p\frac{|x|^2}{2}}dx) & \to & L^p(\omega(x)dx) \\
 & & f & \mapsto & e^{-\frac{|x|^2}{2}}f(x) \\
\end{array}$$
is a linear isometry.
\end{prop}
\begin{proof} We observe that $U_p$ is linear by construction and that
    \begin{align*}
||U_pf||^p_{L^p(\w(x)dx)}&=\int_{\mathbb{R}^d}|Uf(x)|^p\w(x)dx\\
        &=\int_{\mathbb{R}^d}|f(x)|^pe^{-p\frac{|x|^2}{2}}\w(x)dx\\
        &=||f||^p_{L^p(\w(x)e^{-p\frac{|x|^2}{2}}dx)}.
    \end{align*}

\end{proof}

 \begin{lem}\label{Abu}Let $1<p<\infty$ and let $f \in L^2(e^{-|x|^2}dx) \cap L^p(\omega(x)e^{-p\frac{|x|^2}{2}}dx)$. Then 

 \begin{equation}\label{oyesator}
          e^{-t(\mathcal{L}-d)}U_pf(x)=U_pe^{-tL}f(x).
     \end{equation}
% \begin{enumerate}
 %    \item \begin{equation}\label{oyesato}
  %       (\mathcal{L}-d)U_pf(x)=U_pLf(x)
   %  \end{equation}
   %  \item \begin{equation}\label{oyesator}
  %        e^{-t(\mathcal{L}-d)}U_pf=U_pe^{-tL}f
   %  \end{equation}
 %\end{enumerate}
\end{lem}
\begin{proof} Recall that $\mathcal{L} h_\alpha = (2|\alpha| + d) h_\alpha$, with  
$$
h_\alpha(x)=h_{\alpha_1}(x_1)\cdots h_{\alpha_d}(x_d),\;\;\;\; x=(x_1,\dots,x_d),\;\;\; \alpha=(\alpha_1,\dots,\alpha_d),\; \alpha_i\in\{0,1,\dots\},
$$
the multidimensional Hermite functions, given by $h_k(s)=(\pi^{\frac{1}{2}}2^k k!)^{-1/2} H_k(s)e^{-s^2/2}$ for all $s\in \mathbb{R}$ (see \cite{AT}). Now for $f=\sum_{\alpha\in \mathbf{N}^d}c_\alpha H_\alpha$, we have, by the spectral definition of the semigroup $e^{-t(\mathcal{L}-d)}$, that
\begin{align}
   e^{-t(\mathcal{L}-d)}U_p f(x)&=e^{td}\sum_{\alpha\in \mathbb{N}^d}c_\alpha e^{-t(2|\alpha|+d)} e^{-\frac{|x|^2}{2}}H_\alpha(x)\\
   &=U_p e^{-tL} f(x),\label{MOsern}
\end{align}
for all $x \in \mathbb{R}^d$, because  $e^{-tL}f=\sum_{\alpha\in\N^d}e^{-2t|\alpha|}c_\alpha H_\alpha$. Moreover, the semigroups $e^{-tL}$ and $e^{-t\mathcal{L}}$ are bounded on $L^2(e^{-|x|^2}dx)$ and on $L^2(dx)$ respectively, so that the operators $e^{-t\mathcal{L}} U_p$ and $U_p e^{-t d} e^{-tL}$ are both bounded from $L^2(e^{-|x|^2}dx)$ into $L^2(dx)$. Thus, by density of the set of Hermite polynomials in $L^2(e^{-|x|^2}dx)$, we obtain that
\begin{equation*}
     e^{-t(\mathcal{L}-d)}U_p f(x)=U_p e^{-tL} f(x),
\end{equation*}
for all $f \in L^2(e^{-|x|^2}dx)$. This concludes the proof.
   
\end{proof}

%As stated in the introductory text of this chapter, the local maximal operator associated with the harmonic oscillator, $\mathcal{T}^*_{\loc}$, is bounded on weighted $L^p$ spaces for weights in $A_p^{\loc}$. This result will be used in this part. Its complete statement is as follows:

Recall the following result:

\begin{thm}\label{oyesatoru}\textbf{(\cite[Theorem B]{B})} Let $\omega$ be a weight on $\mathbb{R}^d$ and let $1<p<\infty$. Then the following assertions are equivalent

 \begin{enumerate} 
 \item $||M_{\loc}||_{L^p(\w)\to L^p(\w)}<\infty$;\\
            
            \item $\w\in A_p^{\loc}$;\\
            
            \item $||T^*_{\loc}||_{L^p(\w)\to L^p(\w)}<\infty$;\\
            
            \item$||\mathcal{T}^*_{\loc}||_{L^p(\w)\to L^p(\w)}<\infty$;
        \end{enumerate}
where $M_{\loc}$ is the Hardy--Littlewood maximal operator restricted to the region $N(R)$; in other words, it is given, for all $x\in \mathbb{R}^d$ and $f \in L^1_{\loc}$, by
        \begin{equation*}
            M_{\loc}(f)(x)= \sup_{\substack{ Q \ni x \\ Q \subseteq N(R_x)}} \frac{1}{|Q|} \int_{Q} |f(y)| \, dy,
        \end{equation*} 
and $T_{\loc}^*$ is the heat maximal operator restricted to $N(R)$; that is, it is given, for all $x\in \mathbb{R}^d$ and $f \in L^1_{\loc}$, by
\begin{equation*}
    T^*_{\loc} f(x) := \sup_{t>0}T^\Delta_t(|f|.\chi_{N(R_x)})(x).
\end{equation*}

\end{thm}

\begin{rem}\label{qp} For $\omega$ a weight in $A^{\loc}_2$, we have:
    
   \begin{equation}
        ||\sup_{0<t<1}\mathcal{H}_t(|f|.\chi_{N(R_x)})||^2_{L^2(\w(x)e^{-|x|^2}dx)}\leq C([\w]_{A^{\loc}_2})||f||^2_{L^2(\w(x)e^{-|x|^2}dx)}.
   \end{equation}   

Indeed, by equation (\ref{oyesator}) and Theorem \ref{oyesatoru}, in the case $p=2$, we observe that:
    \begin{align*}
||\sup_{0<t<1} \mathcal{H}_t(|f|.\chi_{N(R_x)})||^2_{L^2(\w(x)e^{-|x|^2}dx)}&=||\sup_{0<t<1}U_2e^{-tL}(|f|.\chi_{N(R_x)})||^2_{L^2(\w(x)dx)}\\
        &=||\sup_{0<t<1}e^{td}e^{-t\mathcal{L}}U_2(|f|.\chi_{N(R_x)})||^2_{L^2(\w(x)dx)}\\
        %&\leq e^{d}||\sup_{0<t< 1}e^{-t\mathcal{L}}U_2(|f|.\chi_{N(R_x)})||^2_{L^2(\w(x)dx)}\\
        &\leq e^{2d}||\sup_{t>0}e^{-t\mathcal{L}}U_2(|f|.\chi_{N(R_x)})||^2_{L^2(\w(x)dx)}\\
        &\leq C([\w]_{A^{\loc}_2})||U_2(f)||^2_{L^2(\w(x)dx)}\\
&=C([\w]_{A^{\loc}_2})||f||^2_{L^2(\w(x)e^{-|x|^2}dx)}.
    \end{align*}

\end{rem}

Before generalizing Remark \ref{qp} to weights in $\w$ in $A_p^{\loc}$ for $1<p<\infty$, we need the following lemma:
\begin{lem}\label{OKOSSSK}\textbf{(\cite[Lemma 2.17]{DKK})} Let $(\Omega,\mu)$ be a measure space and let $1 \le p < \infty$. Let $D$ be a dense subspace of $L^p(\Omega,\mu)$ and let $T : D \to L^p(\Omega,\mu)$ 
be a sublinear operator, meaning that for all $f,g \in D$ and all $c \in \mathbb{C}$,
\[
|T(f+g)| \le |Tf| + |Tg|,
\]

\[
|T(cf)| = |c|\,|Tf| 
\]
and
\[
|T(f)-T(g)| \leq |T(f-g)|.
\]

Assume that there exists a constant $C > 0$ such that
\[
\|Tf\|_{L^p(\Omega,\mu)} \le C \|f\|_{L^p(\Omega,\mu)}, \quad \forall f \in D.
\]

Then $T$ admits a unique extension to a bounded operator on $L^p(\Omega,\mu)$, still denoted by $T$, and this extension satisfies
\[
\|Tf\|_{L^p(\Omega,\mu)} \le C \|f\|_{L^p(\Omega,\mu)}, \quad \forall f \in L^p(\Omega,\mu).
\]
\end{lem}

\begin{prop}\label{OOl} Let $1<p<\infty$ and  let $\w\in A_p^{\loc}$, then : 
\begin{equation*}
    ||\sup_{0<t<1}\mathcal{H}_t(|f|.\chi_{N(R_x)})||^p_{L^p(\w(x)e^{-p\frac{|x|^2}{2}}dx)}\lesssim [\w]^p_{A^{\loc}_p}||f||^p_{L^p(\w(x)e^{-p\frac{|x|^2}{2}}dx)}
\end{equation*} 
\end{prop}

\begin{proof}
Let $f \in L^2(e^{-|x|^2}dx) \cap L^p(\omega(x)e^{-p\frac{|x|^2}{2}}dx)$. We observe, by equation (\ref{oyesator}) and Theorem \ref{oyesatoru}, that
\begin{align*}
        ||\sup_{0<t< 1}\mathcal{H}_t(|f|.\chi_{N(R_x)})||^p_{L^p(\w(x)e^{-\frac{p|x|^2}{2}}dx)}&=||\sup_{0<t< 1}U_pe^{-tL}(|f|.\chi_{N(R_x)})||^p_{L^p(\w(x)dx)}\\
        &=||\sup_{0<t< 1}e^{td}e^{-t\mathcal{L}}U_p(|f|.\chi_{N(R_x)})||^p_{L^p(\w(x)dx)}\\
        &\leq e^{pd}||\sup_{t>0}e^{-t\mathcal{L}}U_p(|f|.\chi_{N(R_x)})||^p_{L^p(\w(x)dx)}\\
        &\leq C([\w]_{A_p^{\loc}})||U_p(f)||^p_{L^p(\w(x)dx)}.
    \end{align*}
Now, for the case $f\in  L^p(\w(x)e^{-p\frac{|x|^2}{2}}dx)$, we first note that
\begin{equation}\label{equ-dense-1}
D:=L^2(e^{-|x|^2}dx)\cap L^p(\w(x)e^{-p\frac{|x|^2}{2}}dx) \quad\text{is dense in}\quad L^p(\w(x)e^{-p\frac{|x|^2}{2}}dx).
\end{equation}
Indeed, let $f\in L^p(\w(x)e^{-p\frac{|x|^2}{2}}dx)$ and put $f_n:=f\chi_{\{|f|\leq n\}}\chi_{B(0,n)}$ for $n\in\mathbb{N}$. Each $f_n$ is bounded and supported in a bounded set, hence lies in $L^2(e^{-|x|^2}dx)$; moreover $|f_n|\leq |f|$, so that $f_n\in L^p(\w(x)e^{-p\frac{|x|^2}{2}}dx)$ as well and thus $f_n\in D$. Since $f_n\to f$ pointwise and $|f-f_n|^p\leq 2^p|f|^p$, dominated convergence gives $f_n\to f$ in $L^p(\w(x)e^{-p\frac{|x|^2}{2}}dx)$, which proves \eqref{equ-dense-1}. As the maximal operator $
\sup_{0<t< 1}\mathcal{H}_t(|\;\cdot\;|.\chi_{N(R_x)})$ is sublinear, the result now follows from Lemma \ref{OKOSSSK}.
\end{proof}

By Propositions \ref{OOll} and \ref{OOl}, we have the following corollary:
\begin{cor}\label{import}
    Let $1<p<\infty$ and $\w\in A_p^{\loc}$, then:
\begin{equation*}
    ||\mathcal{H}^*_{\loc}(f)||^p_{L^p(\w(x)e^{-p\frac{|x|^2}{2}}dx)}\lesssim [\w]^p_{A^{\loc}_p}||f||^p_{L^p(\w(x)e^{-p\frac{|x|^2}{2}}dx)}
\end{equation*} 
\end{cor}
\begin{proof}
  The proof follows from the fact that for all $x\in \R^d$  
\begin{equation*}
   \mathcal{H}^*_{\loc}(f)(x)=\sup_{t>0}\mathcal{H}_t(|f|.\chi_{N(R_x)})(x)\leq \sup_{0<t<1}\mathcal{H}_t(|f|.\chi_{N(R_x)})(x)+\sup_{t\geq 1}\mathcal{H}_t(|f|.\chi_{N(R_x)})(x), 
\end{equation*}
followed by Minkowski's inequality.
\end{proof}

\subsection{Theorem on the local maximal operator}
\label{subsec-2-5}

This section aims to conclude our study of the local part of the Ornstein--Uhlenbeck semigroup maximal operator. Before stating the main result of this part, we begin by presenting the following lemma:

\begin{lem}\label{stukimo}
  Let $1<p<\infty$, let $\w$ be a weight on $\mathbb{R}^d$, and let $c\in \mathbb{R}$. Then $\mathcal{H}^*_{\loc}$ is bounded on $L^p(\w(x)dx)$ if and only if it is bounded on $L^p(\w(x)e^{c|x|^2}dx)$.
\end{lem}

\begin{proof}Let $x\in \mathbb{R}^d$ and let $R_x$ be the cube of the grid $\Delta_0^\gamma$ containing $x$. We recall that, by Proposition \ref{ruji}, $e^{\frac{c}{p}|x|^2}\simeq e^{\frac{c}{p}|y|^2}$ for $x,y\in N(R_x)$ and any real constant $c$, with implicit constants depending only on $c$ and $d$. Thus, by this fact, we first observe that if $\mathcal{H}^*_{\loc}$ is bounded from $L^p(\w(x)dx)$ into itself, then
\begin{align*}
    ||\mathcal{H}^*_{\loc}f||^{p}_{L^p(\w(x)e^{c|x|^2}dx)}&=\int_{\R^d}\sup_{t>0}\mathcal{H}_t(|f|\chi_{N(R_x)})(x)^p\w(x)e^{c|x|^2} dx\\
    &=\int_{\R^d}\sup_{t>0}\Big|e^{\frac{c}{p}|x|^2}\int_{N(R_x)}M_t(x,y)|f(y)|dy\Big|^p \w(x)dx\\
    &\simeq \int_{\R^d}\sup_{t>0}\Big|\int_{N(R_x)}e^{\frac{c}{p}|y|^2}M_t(x,y)|f(y)|dy\Big|^p \w(x)dx\\
    &=||\sup_{t>0}\mathcal{H}_t(|f|e^{\frac{c}{p}|.|^2}\chi_{N(R_x)})||^{p}_{L^p(\w(x)dx)}\\
    &\lesssim |||f|e^{\frac{c}{p}|.|^2}||^{p}_{L^p(\w(x)dx)}\\
    &=||f||^{p}_{L^p(\w(x)e^{c|x|^2}dx)}.
   \end{align*}

Secondly, if $\mathcal{H}^*_{\loc}$ is bounded from $L^p(\w(x)e^{c|x|^2}dx)$ into itself, we have that
\begin{align*}
||\mathcal{H}^*_{\loc}f||^{p}_{L^p(\w(x)dx)}&\simeq||\sup_{t>0}\mathcal{H}_t(|f|e^{-\frac{c}{p}|.|^2}\chi_{N(R_x)})||^{p}_{L^p(\w(x)e^{c|x|^2}dx)}\\
    &\lesssim |||f|e^{-\frac{c}{p}|.|^2}||^{p}_{L^p(\w(x)e^{c|x|^2}dx)}\\
    &=||f||^{p}_{L^p(\w(x)dx)}.
\end{align*}    
\end{proof}

We can now present the main result of this subsection.
\begin{thm}\label{result1}(\textbf{Theorem \ref{thm-intro-1}})
 Let $\w$ be a weight on $\mathbb{R}^d$. For $1<p<\infty$, the following assertions are equivalent:
    \begin{enumerate}[label=$(\arabic*)$]
        \item $\w\in A_{p}^{\loc}$.\\
        \item $\mathcal{H}^*_{\loc} : L^p(\w(x)dx)\to L^p(\w(x)dx)$ is bounded.\\
        \item $\mathcal{T}^*_{\loc} : L^p(\w(x)dx)\to L^p(\w(x)dx)$ is bounded.\\
        \item The operator $\sup_{t>0}e^{td}e^{-t\mathcal{L}}(|f|.\chi_{N(R_x)}):L^p(\w(x)dx)\to L^p(\w(x)dx)$ is bounded.
    \end{enumerate}
\end{thm}

\begin{proof} We will prove the following chain of implications: $(1)\Rightarrow (2)\Rightarrow (4)\Rightarrow (3)\Rightarrow (1)$.\\

 $(1)\Rightarrow (2)$ : Corollary $\ref{import}$.\\
 
 $(2)\Rightarrow (4)$: Let us first note that by Lemma $\ref{stukimo}$, since the operator $\mathcal{H}^*_{\loc}$ is bounded on $L^p(\w(x)dx)$, it is then bounded on $L^p(\w(x)e^{-\frac{p|x|^2}{2}}dx)$. Now, for $f$ a function in $ L^2(dx)\cap L^p(\w(x)dx)$, we observe that $fe^{\frac{|x|^2}{2}}\in L^2(e^{-|x|^2}dx)\cap L^p(\w(x)e^{-p\frac{|x|^2}{2}}dx)$. Thus, by Lemma \ref{Abu} applied to the function $fe^{\frac{|x|^2}{2}}$,
\begin{equation}\label{AP}
          e^{td}e^{-t\mathcal{L}}f(x)=U_pe^{-tL}(U_p^{-1}f)(x), \;\text{for all $x\in \R^d$}.
     \end{equation}
By equation $(\ref{AP})$, since the operator $\mathcal{H}^*_{\loc}$ is bounded on $L^p(\w(x)e^{-\frac{p|x|^2}{2}}dx)$,
    \begin{align*}
        ||\sup_{t>0}e^{td}e^{-t\mathcal{L}}(|f|.\chi_{N(R_x)})||^p_{L^p(\w(x)dx)}&=||\sup_{t>0}U_pe^{-tL}U_p^{-1}(|f|.\chi_{N(R_x)})||^p_{L^p(\w(x)dx)}\\
        &=||\sup_{t>0}e^{-tL}U_p^{-1}(|f|.\chi_{N(R_x)})||^p_{L^p(\w(x)e^{-\frac{p|x|^2}{2}}dx)}\\
        &\lesssim [\w]^p_{A^{\loc}_p} ||U_p^{-1}|f|||^p_{L^p(\w(x)e^{-\frac{p|x|^2}{2}}dx)}\\
    &=[\w]^p_{A^{\loc}_p}||f||^p_{L^p(\w(x)dx)}.
    \end{align*}
 
Therefore, the operator $\sup_{t>0}e^{td}e^{-t\mathcal{L}}(|f|.\chi_{N(R_x)})$ is bounded on $L^2(dx)\cap L^p(\w(x)dx)$ equip\-ped with $L^p(\w(x)dx)$ norm. However, the result extends to $L^p(\w(x)dx)$ by Lemma \ref{OKOSSSK}, since the maximal operator $\sup_{t>0}e^{td}e^{-t\mathcal{L}}(|\;\cdot\;|.\chi_{N(R_x)})$ is sublinear and since $L^2(dx)\cap L^p(\w(x)dx)$ is dense in $L^p(\w(x)dx)$. The latter is seen as in \eqref{equ-dense-1}: for $f\in L^p(\w(x)dx)$, the truncations $f_n:=f\chi_{\{|f|\leq n\}}\chi_{B(0,n)}$ are bounded with bounded support, hence belong to $L^2(dx)$, they satisfy $|f_n|\leq|f|$ and $f_n\to f$ pointwise, so that $f_n\to f$ in $L^p(\w(x)dx)$ by dominated convergence.\\
 
  $(4)\Rightarrow (3)$: For all $x\in \R^d$,
    \begin{equation}\label{NBDDHKDJDKVKD}
        \sup_{t>0}e^{-t\mathcal{L}}(|f|.\chi_{N(R_x)})(x)\leq \sup_{t>0}e^{td}e^{-t\mathcal{L}}(|f|.\chi_{N(R_x)})(x).
    \end{equation}
   Hence the result follows.\\
   
  $(3)\Longleftrightarrow (1)$: Theorem \ref{oyesatoru}.

\end{proof}

\begin{rem}\label{POBLOOAZDD}
  This theorem improves Theorem \ref{oyesatoru} from \cite[Theorem B]{B}, since the operator
  \begin{equation*}
      \sup_{t>0}e^{td}e^{-t\mathcal{L}}(|f|.\chi_{N(R_x)})
  \end{equation*}
   dominates the operator 
   \begin{equation*}
        \sup_{t>0}e^{-t\mathcal{L}}(|f|.\chi_{N(R_x)}) 
   \end{equation*}
  in the sense of equation (\ref{NBDDHKDJDKVKD}).
   \end{rem}

%......-----.....-----......--\textcolor{blue}{Je dois montrer l'insuffisance de la classe $A_p^{\loc}$ pour le controle de l'operateur maximal total et donner un phrase de transition}......-----.....-----......--

\section{ The $A_p$ weight class and the Ornstein-Uhlenbeck semigroup maximal operator  }
%Opérateur maximal du semi-groupe Ornstein-Uhlenbeck et classe de Muckenhoupt  $A_p$
\label{sec-3}

In this section, we shall mainly consider weights $\omega$ that belong to the standard Muckenhoupt class $A_p$, that is, the weight and its inverse are controlled over cubes of arbitrary size. We will see among others that such $\omega$ produce weighted spaces $L^p(\R^d,\omega(x) e^{-\frac{p|x|^2}{2}}dx)$ on which the Ornstein-Uhlenbeck semigroup maximal operator is bounded. This is a main result which complements Theorem \ref{result1} from the preceding section concerning weights from the $A_p^{\loc}$ class.

\subsection{Parameter change in Mehler kernel}
\label{subsec-3-1}

In order to obtain a complete control of the maximal operator of the Ornstein--Uhlenbeck semigroup $\mathcal{H}^*$, we will need to re-express its Mehler kernel in a way that allows us to compare it with the standard heat kernel on $\R^d$. This approach will also allow us to simplify the passage from $(1)$ to $(2)$ in the proof of Theorem \ref{result1} (see Remark \ref{import121}); in other words, we will have an alternative proof of Corollary \ref{import}.\\

Recall that the Ornstein-Uhlenbeck semigroup is given, for every $f \in L^p(d\gamma)$, $p\geq 1$, $t>0$, and $x\in \mathbb{R}^d$, by:
$$
\mathcal{H}_tf(x)=\int_{\mathbb{R}^d}M_t(x,y)f(y)dy
$$

with $M_t$ the Mehler kernel given, for all $t>0$ and all $x,y \in \mathbb{R}^d$, by 
 $$
   M_t(x,y)=\Big(\frac{1}{\pi(1-e^{-2t})}\Big)^\frac{d}{2}\exp\Big(-\frac{|e^{-t}x-y|^2}{(1-e^{-2t})}\Big).
$$

Note that,  for every fixed $t>0$, the kernel $M_t$ admits several equivalent forms. An example of these forms is given in the following lemma:

\begin{lem}\label{pipipi}
  For all $t>0$ and $x$ and $y$ in $\mathbb{R}^d$ 
   $$
  M_t(x,y)= S_t(x,y)\exp(-|y|^2)\pi^{-\frac{d}{2}},
   $$
with

$$
S_t(x,y)=\Big(\frac{1}{1-e^{-2t}}\Big)^\frac{d}{2}\exp\Big(\frac{1}{2}\frac{1}{e^t+1}|x+y|^2-\frac{1}{2}\frac{1}{e^t-1}|x-y|^2\Big)
$$
\end{lem}
\begin{proof} By direct computation, we have
    \begin{align*}
        S_t(x,y)&=\Big(\frac{1}{1-e^{-2t}}\Big)^\frac{d}{2}\exp\Big(\frac{1}{2}\frac{1}{e^t+1}|x+y|^2-\frac{1}{2}\frac{1}{e^t-1}|x-y|^2\Big)\\
        &= \Big(\frac{1}{1-e^{-2t}}\Big)^\frac{d}{2}\exp\Big(-\frac{1}{2}(\frac{1}{e^t-1}|x-y|^2-\frac{1}{e^t+1}|x+y|^2)\Big)\\
        &=\Big(\frac{1}{1-e^{-2t}}\Big)^\frac{d}{2}\exp\Big(-\frac{1}{2e^{2t}}\frac{2|x|^2-4\langle x,y\rangle e^{t}+2|y|^2}{(1-e^{-2t})}\Big)\\
        &=\Big(\frac{1}{1-e^{-2t}}\Big)^\frac{d}{2}\exp\Big(-\frac{1}{2}\frac{2(|x|^2e^{-2t}-2\langle x,y\rangle e^{-t}+|y|^2)+2(e^{-2t}-1)|y|^2}{(1-e^{-2t})}\Big)\\
        &=\Big(\frac{1}{1-e^{-2t}}\Big)^\frac{d}{2}\exp\Big(-\frac{|e^{-t}x-y|^2}{(1-e^{-2t})}\Big)\exp(|y|^2)\\
        &=M_t(x,y)\exp(|y|^2)\pi^{\frac{d}{2}},
    \end{align*}
which is the assertion.
\end{proof}

By Lemma \ref{pipipi}, we retain that for every $f \in L^p(d\gamma)$, $p\geq 1$, $t>0$, and $x\in \mathbb{R}^d$, the Ornstein-Uhlenbeck semigroup also admits the following integral representation:
\begin{equation}\label{KJLDKDKDL}
    \mathcal{H}_tf(x)=\pi^{-\frac{d}{2}}\int_{\mathbb{R}^d}S_t(x,y)f(y)d\gamma(y).
\end{equation}

Now, in equation (\ref{KJLDKDKDL}), using the same parameter change as the one used for the kernel of the harmonic oscillator in \cite[page $209$]{AT}, we have, by setting

\begin{equation*}
    t=t(s)=\log\frac{1+s}{1-s},\;\;\; 0<s<1,\;\;\; 0<t<\infty,
\end{equation*}
that
\begin{equation}\label{eqs}
    S_t(x,y)=N_s(x,y)=\frac{(1+s)^d}{(4s)^\frac{d}{2}}\exp\Big(\frac{|x|^2+|y|^2}{2}-\frac{s}{4}|x+y|^2-\frac{1}{4s}|x-y|^2\Big).
\end{equation}
 Indeed, we can see in the few lines that follow that
\begin{align*}
    S_t(x,y)=S_{t(s)}(x,y)&=\Big(\frac{(1+s)^2}{(1+s)^2-(1-s)^2}\Big)^\frac{d}{2}\exp\Big(\frac{1}{2}\times\frac{1-s}{2}|x+y|^2-\frac{1}{2}\times\frac{1-s}{2s}|x-y|^2\Big)\\
    &=\frac{(1+s)^d}{(4s)^\frac{d}{2}}\exp\Big[\frac{1}{4}\Big(|x+y|^2+|x-y|^2-s|x+y|^2-\frac{1}{s}|x-y|^2\Big)\Big]\\
    &=\frac{(1+s)^d}{(4s)^\frac{d}{2}}\exp\Big(\frac{|x|^2+|y|^2}{2}-\frac{s}{4}|x+y|^2-\frac{1}{4s}|x-y|^2\Big)\\
    &=:N_s(x,y).
\end{align*}

\begin{lem}\label{chal}
Let $f\in L^1_{\loc}$, then
    \begin{equation*}
\mathcal{H}^*f(x)\lesssim e^\frac{|x|^2}{2}\sup_{0<s<1}T_s^\Delta(|f|e^{-\frac{|\;.\;|^2}{2}})(x),
\end{equation*}
for all $x\in \mathbb{R}^d$, with $T_s^\Delta$ the heat semigroup.
\end{lem}
\begin{proof}\;
 For $x\in \mathbb{R}^d$, we have by the equation (\ref{eqs}) that
    \begin{align*}
\mathcal{H}^*f(x)&=\pi^{-\frac{d}{2}}\sup_{t>0}\int_{\mathbb{R}^d}S_t(x,y)|f(y)|d\gamma(y)\\
&=\pi^{-\frac{d}{2}}\sup_{0<s<1}\int_{\mathbb{R}^d}N_s(x,y)|f(y)|d\gamma(y)\\
&=\sup_{0<s<1} \frac{(1+s)^d}{(4\pi s)^\frac{d}{2}}\int_{\mathbb{R}^d}e^{\frac{|x|^2+|y|^2}{2}-\frac{s}{4}|x+y|^2-\frac{1}{4s}|x-y|^2}|f(y)|d\gamma(y)\\
&\lesssim \sup_{0<s<1} \frac{1}{(4\pi s)^\frac{d}{2}}\int_{\mathbb{R}^d}e^{\frac{|x|^2}{2}}e^{-\frac{s}{4}|x+y|^2-\frac{1}{4s}|x-y|^2}|f(y)|e^{\frac{-|y|^2}{2}}dy\\
&\leq e^\frac{|x|^2}{2}\sup_{0<s<1} \frac{1}{(4\pi s)^\frac{d}{2}}\int_{\mathbb{R}^d}e^{-\frac{1}{4s}|x-y|^2}|f(y)|e^{-\frac{|y|^2}{2}}dy\\
%&=e^\frac{|x|^2}{2}\sup_{0<s<1}T_{2s}^\Delta(|f|e^{-\frac{|\;.\;|^2}{2}})(x)\\
&=e^\frac{|x|^2}{2}\sup_{0<s<1}T_{s}^\Delta(|f|e^{-\frac{|\;.\;|^2}{2}})(x)\label{chaleur}.
 \end{align*}

\end{proof}

By Lemma \ref{chal}, we reduce the proof of Corollary \ref{import} where we showed that that if $\w \in A_p^{\loc}$, then $\mathcal{H}^*_{\loc}$ is bounded on $L^p(\w(x)dx)$, to the implication (2) $\Longrightarrow$ (3) of Theorem \ref{oyesatoru}.
\begin{rem}\label{import121}Let $\w \in A^{\loc}_p$. Then by Theorem \ref{oyesatoru} and Lemma \ref{chal},
\begin{align*}
    ||\sup_{t>0}\mathcal{H}_t(|f|\chi_{N(R_x)})||_{L^p(\w(x)dx)}&=||\sup_{t>0}\mathcal{H}_t(|f|\chi_{N(R_x)})||_{L^p(v(x)e^{-\frac{p|x|^2}{2}}dx)}\\&\lesssim||\sup_{0<s<1}T_s^\Delta(|f|\chi_{N(R_x)}e^{-\frac{|\;.\;|^2}{2}})||_{L^p(v(x)dx)}\\
    &\leq ||\sup_{s>0}T_s^\Delta(|f|\chi_{N(R_x)}e^{-\frac{|\;.\;|^2}{2}})||_{L^p(v(x)dx)}\\
    &\overset{Thm \,\ref{oyesatoru}}{\lesssim_v} ||f||_{L^p(v(x)e^{-\frac{p|x|^2}{2}}dx)}\\
    &=||f||_{L^p(\w(x)dx)},
\end{align*}
where we put $v=\w e^{\frac{p|.|^2}{2}}\in A_p^{\loc}$.
Hence $\mathcal{H}^*_{\loc}$ is bounded on $L^p(\w(x)dx)$.
\end{rem}

\subsection{Consequence on the Ornstein-Uhlenbeck  semigroup maximal operator}
\label{subsec-3-2}

In this subsection, we will exploit the pointwise control of the Ornstein-Uhlenbeck maximal operator in terms of the heat maximal operator further, in order to obtain a weighted $L^p$ estimate of the full Ornstein-Uhlenbeck maximal operator $\mathcal{H}^*$. The key to our results will rely on the following theorem.

\begin{thm}\label{Stein}\textbf{(\cite{S3})}
   For $1<p<\infty$,   \begin{equation*}
           \w\in A_p\iff ||T^*||_{L^p(\w(x)dx)\to L^p(\w(x)dx)}<\infty.
       \end{equation*}
   \end{thm}

\begin{thm}\label{resultt2}
     Let $1<p<\infty$. Let $\w_0$ be an $A_p$ weight on $\R^d$, or more generally, a weight such that $\sup_{0<s\leq 1} T_s^\Delta(|f|)$ is bounded on $L^p(\R^d, \w_0(x)dx)$. Then the Ornstein-Uhlenbeck maximal operator is bounded with respect to the weight 
      \begin{equation*}
          \w(x)=\w_0(x)e^{-\frac{p|x|^2}{2}}.
      \end{equation*}
 In other words     
       \begin{equation*}
||\mathcal{H}^*f||^p_{L^p(\w(x)dx)}\leq C\left(\left[\w_0\right]_{A_p}\right)||f||^p_{ L^p(\w(x)dx)},
  \end{equation*}
with $C\left(\left[\w_0\right]_{A_p}\right)$ a constant depending on the characteristic of $\w_0$ and on $d$ and $p$.
    \end{thm} 
\begin{proof}\, In case that $\w_0\in A_p$, we have, by Lemma \ref{chal} and Theorem \ref{Stein},
    \begin{align*}
||\mathcal{H}^*f||^p_{L^p(\w_0(x)e^{-\frac{p|x|^2}{2}}dx)}&\lesssim\int_{\mathbb{R}^d}\sup_{0<s\leq 1}T_{s}^\Delta(|f|e^{-\frac{|\;.\;|^2}{2}})(x)^p\w_0(x)dx \\ 
&\leq\int_{\mathbb{R}^d}\sup_{s>0}T_{s}^\Delta(|f|e^{-\frac{|\;.\;|^2}{2}})(x)^p\w_0(x)dx\\
&=\int_{\mathbb{R}^d}T^*(fe^{-\frac{|.|^2}{2}})(x)^p\w_0(x)dx \\ 
&\leq C\left(\left[\w_0\right]_{A_p}\right) ||fe^{-\frac{|.|^2}{2}}||^p_{L^p(\w_0(x)dx)}\\
&=C\left(\left[\w_0\right]_{A_p}\right)||f||^p_{L^p(\w_0(x)e^{-\frac{p|x|^2}{2}}dx)}.
 \end{align*}  
In case $\sup_{0<s\leq 1}T_s^\Delta(|f|)$ is bounded on $L^p(\w_0(x)dx)$, we argue similarly and in the above calculation, we estimate the first line directly against the fourth line.
\end{proof}

Following the previous result, and since the class $A_p^{\loc}$ does not allow the boundedness of the Ornstein--Uhlenbeck semigroup maximal operator $\mathcal{H}^*$ on $L^p(\w(x)e^{-\frac{p|x|^2}{2}}dx)$,  we wondered whether it is possible to obtain a class of weights lying between $A_p$ and $A_p^{\loc}$ that allows control of the maximal operator on $L^p(\w(x)e^{-\frac{p|x|^2}{2}}dx)$.
This issue is addressed in the rest of this article.

\section{The $A_p^\alpha$ weight class }
\label{sec-4}

  In Theorem \ref{resultt2} we have seen that the Ornstein-Uhlenbeck maximal operator $\mathcal{H}^*$ is bounded on $L^p(\w(x)e^{-\frac{p|x|^2}{2}}dx)$ if the heat maximal operator $\sup_{0<s<1}T_s^{\Delta}$, is bounded on $L^p(\w(x)dx)$, so in particular if $\w$ is an $A_p$ weight. The goal of this section is to define a larger class of weights, called $A^\alpha_p$ depending on a parameter $\alpha \in [0,1)$, such that the Ornstein-Uhlenbeck maximal operator $\mathcal{H}^*$ is still bounded on  $L^p(\w(x)e^{-p\frac{|x|^2}{2}}dx)$ for $\w \in A^\alpha_p$.\

A weight $\w$ belongs to $A^\alpha_p$ if and only if it satisfies the Muckenhoupt condition
\begin{equation*}
  \Big(\frac{1}{|Q|}\int_{Q}\omega(x)dx\Big)^\frac{1}{p} \Big(\frac{1}{|Q|}\int_{Q}\omega^{-\frac{1}{p-1}}(x)dx\Big)^\frac{p-1}{p}\leq C
\end{equation*}
for $Q$ a cube (with sides parallel to the axes) sitting inside some larger cube $N_\alpha(R)$ whose size shrinks with the distance to the origin (Definition \ref{OKOKJHJDHJD}). Roughly speaking, the side length of $N_\alpha(R)$ is comparable to $\max(1,|x|)^{-\alpha}$, where $x$ is any point of $N_\alpha(R)$.

\subsection{The $\alpha$-dyadic grid on $\mathbb{R}^d$}
\label{subsec-4-1}
%Dans cette partie nous allons construire une grille en apportant  une nouvelle décomposition à l'espace $\R^d$. 

In this subsection, we will present all the necessary elements that will allow us to define the class of weights $A_p^\alpha$.

%Puis, nous donnerons quelques propriété observer sur cette grille.

\begin{defi}\textbf{(Distance between subsets of $\mathbb{R}^d$)} Let $Q$ and $R$ be two non-empty subsets of $\mathbb{R}^d$. Then the $\ell^\infty$ distance between $Q$ and $R$ is defined as
    \begin{equation*}
        \rho(Q,R):=\inf\{||x-y||_\infty\;:\;x\in Q\; \text{and} \; y\in R\}.
    \end{equation*}
    with $||x-y||_\infty=\max\{|x_1-y_1|,...,|x_d-y_d|\}$.
We also let $\rho(x,R) = \rho(\{x\},R)$ for $x \in \R^d$.
   % Soient $Q$, $R\subseteq\mathbb{R}^d$, deux sous-parties non-vides. Alors la distance entre $Q$ et $R$ est définie comme 
  
\end{defi}
\begin{defi}\textbf{(Neighbors)}\
   Let $Q$ and $R$ be two non-empty subsets of $\mathbb{R}^d$. We will say that they are neighbors if $\rho(Q,R)=0$.
\end{defi}

We recall from Definition \ref{AZ} in the introduction the following.

\begin{defi}\label{DJDDJLZUEYYEYEYEYYE}\textbf{(Grid $\Delta^\alpha$)} Let $\alpha \in [0,1)$.
\begin{enumerate}
    \item We define the layers $L^\alpha_0 := [-1,1)^d$ and for $n \geq 1$, $L^\alpha_n := [-(n+1),n+1)^d \backslash [-n,n)^d$.
\item  For $n\geq 1$, we define $\delta_n^\alpha$ to be the unique nonnegative integer $k$ such that $2^{-k}\leq n^{-\alpha}<2^{- k +1}$. Also, for $n \geq 1$ let us designate by $\Delta_n^\alpha$ the collection of cubes given below
     \begin{equation*}
        \Delta^\alpha_n=\{m+[0,2^{-\delta^\alpha_n})^d\;:\;2^{\delta_n^\alpha}m\in \mathbb{Z}^d,\; m+[0,2^{-\delta^\alpha_n})^d\subset [-(n+1),n+1)^d\setminus[-n,n)^d\}.
    \end{equation*}
   We also let $\delta_0^\alpha = -1$ and $\Delta^\alpha_0 = \{[-1,1)^d\}$.
     % Pour $n\geq 1$, nous désignons par $\delta_n^\alpha$ l'unique entier non-négatif $k$ tel que $2^{-k}\leq n^{-\alpha}<2^{-k+1}$.\
    \item We define the grid $\Delta^\alpha$ as
\begin{equation*}
    \Delta^\alpha:=\bigcup_{n\geq 0}\Delta_n^\alpha.
\end{equation*}
%with $\Delta^\alpha_0=[-1,1)^d$.
\end{enumerate}
\end{defi}

%%%%%%%%%%%%%%%%%%%%%%%%%%%%%%%%%%%%%%%%%%%%%%%%%%%%%%%%%%%%%%%%%%%%%%%%%%%%%%
%% Replacement block for Subsection \ref{subsec-4-1}.
%%
%% Place Remark \ref{rem-grid-basics} and Lemma \ref{lem-small-cube} directly
%% after Definition \ref{DJDDJLZUEYYEYEYEYYE} (the grid $\Delta^\alpha$), and
%% replace the current proof of Lemma \ref{oh} by the one below.
%%
%% Prerequisite: in Definition \ref{DJDDJLZUEYYEYEYEYYE}, add the convention
%% $\delta^\alpha_0 := -1$ and name the layers, e.g.
%%
%%   L^\alpha_0 := [-1,1)^d, \qquad
%%   L^\alpha_n := [-(n+1),n+1)^d \setminus [-n,n)^d \quad (n \geq 1),
%%
%% so that $\Delta^\alpha_n$ consists of the cubes of side length
%% $2^{-\delta^\alpha_n}$ contained in $L^\alpha_n$, for every $n \geq 0$.
%% (Note: the symbol $L_\ell$ is already used in Section \ref{sec-2} for the
%% layers of the Gaussian grid, hence the superscript $\alpha$.)
%%%%%%%%%%%%%%%%%%%%%%%%%%%%%%%%%%%%%%%%%%%%%%%%%%%%%%%%%%%%%%%%%%%%%%%%%%%%%%

\begin{rem}\label{rem-grid-basics}
Let $\alpha \in [0,1)$. We record two elementary properties of the grid $\Delta^\alpha$ that will be used repeatedly.
\begin{enumerate}
\item The sequence $(\delta^\alpha_n)_{n \geq 0}$ is nondecreasing. Consequently,
\begin{equation}\label{equ-delta-monotone}
\delta^\alpha_{\max(m,n)} = \max\left(\delta^\alpha_m,\delta^\alpha_n\right)
\qquad (m,n \geq 0).
\end{equation}
\item $\Delta^\alpha$ is a partition of $\R^d$. For $x \in \R^d$ we denote by $Q_x$ the unique cube of $\Delta^\alpha$ containing $x$.
\end{enumerate}
\end{rem}

\begin{proof}
(1) For $n \geq 1$, the defining property $2^{-\delta^\alpha_n} \leq n^{-\alpha} < 2^{-\delta^\alpha_n+1}$ gives $\delta^\alpha_n = \lceil \alpha \log_2 n \rceil$, which is nondecreasing in $n$, and $\delta^\alpha_0 = -1 \leq 0 = \delta^\alpha_1$. Identity \eqref{equ-delta-monotone} is immediate from the monotonicity.

(2) We have $\R^d = \bigsqcup_{n \geq 0} L^\alpha_n$. It therefore suffices to see that, for each $n \geq 0$, the cubes of $\Delta^\alpha_n$ partition $L^\alpha_n$. For $n = 0$ this holds since $\Delta^\alpha_0 = \{[-1,1)^d\} = \{L^\alpha_0\}$. For $n \geq 1$, the side length $2^{-\delta^\alpha_n}$ satisfies $2^{-\delta^\alpha_n} \leq 1$ and divides $1$, while the boundary of $L^\alpha_n$ is contained in the union of the hyperplanes $\{x_i = k\}$ with $k \in \Z$. Hence every cube $m + [0,2^{-\delta^\alpha_n})^d$ with $2^{\delta^\alpha_n} m \in \Z^d$ is either contained in $L^\alpha_n$ or disjoint from it, and these cubes partition $\R^d$.
\end{proof}

\begin{lem}\label{lem-small-cube}
Let $J \geq 0$ be an integer and let $\delta \geq \max(0,\delta^\alpha_J)$ be an integer. Let $n \in 2^{-\delta}\Z^d$ with $n \in [-(J+1),J+1)^d$. Then
\begin{equation}\label{equ-small-cube}
n + [0,2^{-\delta})^d \subseteq Q_n,
\end{equation}
and $Q_n \in \Delta^\alpha_j$ for some $j \leq J$.
\end{lem}

\begin{proof}
Let $j \geq 0$ be the index of the unique layer with $n \in L^\alpha_j$, so that $Q_n \in \Delta^\alpha_j$ by Remark \ref{rem-grid-basics}. For every $k \geq J+1$ we have $[-(J+1),J+1)^d \subseteq [-k,k)^d$, while $L^\alpha_k$ is disjoint from $[-k,k)^d$; hence $L^\alpha_k \cap [-(J+1),J+1)^d = \varnothing$. Since $n \in [-(J+1),J+1)^d$, this forces $j \leq J$.

If $j = 0$, then $Q_n = [-1,1)^d$ and $n \in 2^{-\delta}\Z^d \cap [-1,1)^d$ with $\delta \geq 0$, so that $n_i + 2^{-\delta} \leq 1$ for every $i$, which gives \eqref{equ-small-cube}.

Assume now $j \geq 1$. By Remark \ref{rem-grid-basics} and $j \leq J$ we have $\delta^\alpha_j \leq \delta^\alpha_J \leq \delta$, so that $2^{-\delta}$ divides $2^{-\delta^\alpha_j}$. Put $l_i := \lfloor 2^{\delta^\alpha_j} n_i \rfloor \in \Z$ for $i = 1,\ldots,d$.
Here, $\lfloor \zeta \rfloor$ denotes the integer part of $\zeta$.
Since $n_i \in 2^{-\delta}\Z$, we obtain
\[
[n_i,n_i+2^{-\delta}) \subseteq 2^{-\delta^\alpha_j}\,[l_i,l_i+1)
\qquad (i = 1,\ldots,d),
\]
whence, writing $l = (l_i)_{i=1}^d$,
\[
n + [0,2^{-\delta})^d \subseteq 2^{-\delta^\alpha_j}\left(l + [0,1)^d\right).
\]
The cube on the right hand side is the unique cube of side length $2^{-\delta^\alpha_j}$, aligned with the lattice $2^{-\delta^\alpha_j}\Z^d$, that contains $n$; by Remark \ref{rem-grid-basics} it therefore equals $Q_n$. This proves \eqref{equ-small-cube}.
\end{proof}

\begin{defi}\label{defi-steps}\textbf{(Steps between $Q$ and $R$)} %Soient $Q$ et $R\in \Delta^\alpha$, le nombre de pas entre $Q$ et $R$ sera  défini comme 
For $Q,R\in \Delta^\alpha$, the number of steps between $Q$ and $R$ is defined as
    \begin{align*}
    \begin{split}
        r(Q,R)
        :=\min\{n\in \mathbb{N}&:\; \exists \;Q_0,...Q_n\in \Delta^\alpha\;, \text{with}\; Q=Q_0,\; R=Q_n,\; \text{and} \\&\quad \;Q_k,\; Q_{k+1}\;\text{neighbors}\; \forall k\in\{0,...,n-1\}\}.
        \end{split}
    \end{align*}
\end{defi}
%Comme sur la grille gaussienne $\Delta^\gamma_0$, nous allons construire  une région qui nous permettra de définir une nouvelle classe de poids.
As on the Gaussian grid $\Delta^\gamma_0$, we will construct a region that will allow us to define a new class of weights.   
Throughout, for a cube $P = m_P + [0,l(P))^d$ and $1 \leq i \leq d$ we write
$P_i := [m_{P,i}, m_{P,i}+l(P))$ for its $i$-th edge, so that
$P = \prod_{i=1}^d P_i$ and $\rho(P,P') = \max_i \dist(P_i,P'_i)$.

\begin{lem}\label{lem-neighbour-sizes}
Let $\alpha \in [0,1)$.
\begin{enumerate}
\item For every $n \geq 1$ one has $\delta^\alpha_{n} - \delta^\alpha_{n-1} \leq 1$.
\item Let $Q \in \Delta^\alpha_n$, put $h := l(Q)$, and let $Q' \in \Delta^\alpha_{n'}$ satisfy $\rho(Q,Q') < h$. Then $n' \leq 2$ if $n = 0$, and $|n'-n| \leq 1$ if $n \geq 1$. In both cases
\begin{equation}\label{equ-size-ratio}
\tfrac14\,h \;\leq\; l(Q') \;\leq\; 2h ,
\end{equation}
and $\tfrac12 h \leq l(Q') \leq 2h$ if $n \geq 1$.
\item In the situation of {\rm (2)} with $n \geq 1$, put $H := l(Q')$. Then for every $i \in \{1,\ldots,d\}$,
\begin{equation}\label{equ-edge-reach}
Q'_i \subseteq
\begin{cases}
[\,m_{Q,i}-H,\; m_{Q,i}+h+H\,) & \text{if } H \geq h,\\[4pt]
[\,m_{Q,i}-h,\; m_{Q,i}+2h\,) & \text{if } H < h.
\end{cases}
\end{equation}
\end{enumerate}
\end{lem}

\begin{proof}
(1) For $n = 1$ this reads $0 - (-1) \leq 1$. For $n \geq 2$ we have $\delta^\alpha_n = \lceil \alpha \log_2 n \rceil$, and by subadditivity of the ceiling function,
\[
\delta^\alpha_n \;\leq\; \left\lceil \alpha \log_2 (n-1) \right\rceil + \left\lceil \alpha \log_2 \tfrac{n}{n-1} \right\rceil \;\leq\; \delta^\alpha_{n-1} + 1,
\]
since $0 \leq \alpha \log_2 \frac{n}{n-1} \leq \alpha < 1$.

(2) Every $x \in Q$ satisfies $\|x\|_\infty \in [n,n+1]$ and every $y \in Q'$ satisfies $\|y\|_\infty \in [n',n'+1]$. As $\left| \|x\|_\infty - \|y\|_\infty \right| \leq \|x-y\|_\infty$, taking the infimum over $x \in Q$ and $y \in Q'$ gives
\[
\rho(Q,Q') \;\geq\; \max\left(0,\,|n-n'|-1\right).
\]
If $n \geq 1$, then $h \leq 1$ and $\rho(Q,Q') < h \leq 1$ forces $|n-n'| \leq 1$; if $n = 0$, then $h = 2$ and we obtain $n' \leq 2$. Estimate \eqref{equ-size-ratio} now follows from part (1) together with the monotonicity of $(\delta^\alpha_k)_{k \geq 0}$: for $n \geq 1$ and $|n'-n| \leq 1$ we get $\left|\delta^\alpha_{n'} - \delta^\alpha_n\right| \leq 1$, whence $\frac12 h \leq l(Q') \leq 2h$; for $n = 0$ and $n' \leq 2$ we get $\delta^\alpha_{n'} \leq \delta^\alpha_0 + 2 = 1$, whence $l(Q') \geq \frac12 = \frac14 h$, while $l(Q') \leq 2 = h$.

(3) Since $n \geq 1$ we have $m_Q \in h\Z^d$. Assume first $H \geq h$. If $n' \geq 1$, then $m_{Q'} \in H\Z^d \subseteq h\Z^d$, because $h$ divides $H$ by \eqref{equ-size-ratio}; if $n' = 0$, then $Q'_i = [-1,1)$ and its endpoints lie in $\Z \subseteq h^{-1}\cdot\Z$, hence in $h\Z$, because $h \leq 1$ is a power of $2$. In either case both endpoints of $Q_i$ and of $Q'_i$ lie in $h\Z$, so that $\dist(Q_i,Q'_i) \in h\N$. As $\dist(Q_i,Q'_i) \leq \rho(Q,Q') < h$, this distance vanishes, i.e.\ $Q_i$ and $Q'_i$ are adjacent or overlap, which is the first line of \eqref{equ-edge-reach}.

If $H < h$, then $m_{Q'} \in H\Z^d$ and $m_Q \in h\Z^d \subseteq H\Z^d$, so that $\dist(Q_i,Q'_i) \in H\N$ and therefore $\dist(Q_i,Q'_i) \leq h - H$. Consequently $Q'_i$ is contained in the interval obtained by enlarging $Q_i$ by $(h-H)+H = h$ on both sides, which is the second line of \eqref{equ-edge-reach}.
\end{proof}

\begin{lem}\label{lem-N-alpha-exists}
Let $\alpha \in [0,1)$ and let $Q \in \Delta^\alpha$. Then there exists a cube $C \subseteq \R^d$ of side length $4l(Q)$ such that
\begin{enumerate}
\item $C$ is a finite union of cubes of $\Delta^\alpha$, and
\item $Q' \subseteq C$ for every $Q' \in \Delta^\alpha$ with $\rho(Q,Q') < l(Q)$,
\item $Q$ is a dyadic subcube of $C$, and
\item $\left\{x \in \R^d \;:\; \rho(x,Q) < l(Q)\right\} \subseteq C$.
\end{enumerate}
\end{lem}

\begin{proof}
Write $Q = m_Q + [0,h)^d$ with $h := l(Q)$, let $n \geq 0$ be the index of the layer containing $Q$, and abbreviate $m_i := m_{Q,i}$. We distinguish three cases according to the value of $h$. Note that $h = 2$ occurs exactly for $n = 0$, and that $h \leq \frac12$ forces $\delta^\alpha_n \geq 1$, hence $n^{-\alpha} < 1$ and therefore $n \geq 2$.

\medskip
\emph{Case 1: $h = 2$, i.e.\ $Q = [-1,1)^d$.} Put $C := [-5,3)^d$, a cube of side length $8 = 4h$. By Lemma \ref{lem-neighbour-sizes}(2), every $Q'$ as in (2) lies in a layer $L^\alpha_{n'}$ with $n' \leq 2$, hence $Q' \subseteq [-3,3)^d \subseteq C$, which is (2). For (1), let $P \in \Delta^\alpha$ meet $C$. Every cube of $\Delta^\alpha$ other than $[-1,1)^d$ has side length at most $1$ and is aligned with a lattice refining $\Z^d$; since the endpoints of the edges of $C$ are integers, such a $P$ satisfies $P \subseteq C$. The remaining cube $[-1,1)^d$ is contained in $C$ anyway. As $C$ is bounded, only finitely many cubes are involved. Finally, the four dyadic subintervals of $[-5,3)$ of length $2$ are $[-5,-3)$, $[-3,-1)$, $[-1,1)$ and $[1,3)$, so that $Q = [-1,1)^d$ is a dyadic subcube of $C$, which is (3). Finally $\left\{x : \rho(x,Q)<2\right\}=(-3,3)^d\subseteq C$, which is (4).

Note that the seemingly more natural choice $C = [-4,4)^d$ satisfies (1) and (2) but not (3): the dyadic subcubes of $[-4,4)^d$ of side length $2$ are the $[2k,2k+2)^d$ with $-2 \leq k \leq 1$, and $[-1,1)^d$ is not among them.

\medskip
\emph{Case 2: $h \leq \frac12$.} Then $n \geq 2$ and $2h \leq 1$. We define
\begin{equation}\label{equ-C-case2}
C_i :=
\begin{cases}
[\,m_i - 2h,\; m_i + 2h\,) & \text{if } m_i \in 2h\Z,\\[4pt]
[\,m_i - h,\; m_i + 3h\,) & \text{if } m_i \notin 2h\Z,
\end{cases}
\qquad C := \prod_{i=1}^d C_i .
\end{equation}
Since $m_i \in h\Z$, in the second case $m_i - h \in 2h\Z$; hence in both cases $C_i$ is an interval of length $4h$ whose endpoints lie in $2h\Z$, and $Q_i \subseteq C_i$. The four dyadic subintervals of $C_i$ of length $h$ are obtained by cutting $C_i$ at the points of $h\Z$ it contains, and $Q_i = [m_i,m_i+h)$ is one of them; thus $Q$ is a dyadic subcube of $C$, which is (3). Property (4) also holds, since $\left(m_i-h,m_i+2h\right)\subseteq C_i$ in both lines of \eqref{equ-C-case2}.

We first verify (2). Let $Q' \in \Delta^\alpha$ with $\rho(Q,Q') < h$ and put $H := l(Q')$; by Lemma \ref{lem-neighbour-sizes}(2) we have $H \leq 2h$. If $H \leq h$, then \eqref{equ-edge-reach} gives $Q'_i \subseteq [m_i-h,m_i+2h) \subseteq C_i$ in both lines of \eqref{equ-C-case2}. If $H = 2h$, then $Q'$ lies in the layer $L^\alpha_{n-1}$ with $n-1 \geq 1$, so that $m_{Q'} \in 2h\Z^d$, and by the first line of \eqref{equ-edge-reach} the intervals $Q_i$ and $Q'_i$ are adjacent or overlap, which leaves
\[
m_{Q',i} \in \{m_i-2h,\,m_i\} \quad\text{if } m_i \in 2h\Z,
\qquad
m_{Q',i} \in \{m_i-h,\,m_i+h\} \quad\text{if } m_i \notin 2h\Z .
\]
In each of the four situations $Q'_i = [m_{Q',i},m_{Q',i}+2h) \subseteq C_i$. This proves (2).

For (1) we show that every cube of $\Delta^\alpha$ meeting $C$ is contained in $C$. Since $Q \subseteq L^\alpha_n$, there is an index $j$ with $m_j \geq n$ or $m_j + h \leq -n$. Let $x \in C$. In the first situation $x_j \geq m_j - 2h \geq n-1$, and in the second $x_j < m_j + 3h \leq -n + 2h \leq -(n-1)$; in both cases, $x$ cannot belong to a point of $[-1,1)^d$.
Hence $C \cap [-1,1)^d = \varnothing$, and every cube of $\Delta^\alpha$ meeting $C$ lies in a layer $L^\alpha_{n'}$ with $n' \geq 1$ and $n'+1 > n-1$, i.e.\ $n' \geq n-1$. By Lemma \ref{lem-neighbour-sizes}(1) such a cube has side length $2^{-\delta^\alpha_{n'}} \leq 2^{-\delta^\alpha_{n-1}} \leq 2h$ and is aligned with the lattice $2^{-\delta^\alpha_{n'}}\Z^d$, which refines $2h\Z^d$. As the endpoints of the intervals $C_i$ lie in $2h\Z$, any such cube meeting $C$ is contained in $C$. Since $C$ is bounded, only finitely many cubes are involved.

\medskip
\emph{Case 3: $h = 1$.} Then $m_Q \in \Z^d$, and by Lemma \ref{lem-neighbour-sizes}(2) every $Q'$ as in (2) satisfies $l(Q') \leq 2$, with $l(Q') = 2$ only for $Q' = [-1,1)^d$. We define
\begin{equation}\label{equ-C-case3}
C_i :=
\begin{cases}
[\,m_i - 1,\; m_i + 3\,) & \text{if } m_i \geq 2 \ \text{ or }\ m_i = -2,\\[4pt]
[\,m_i - 2,\; m_i + 2\,) & \text{otherwise,}
\end{cases}
\qquad C := \prod_{i=1}^d C_i ,
\end{equation}
an interval of length $4 = 4h$ with endpoints in $\Z$ and with $Q_i \subseteq C_i$. As above, $Q_i = [m_i,m_i+1)$ is one of the four dyadic subintervals of $C_i$ of length $1$, so that (3) holds; and $\left(m_i-1,m_i+2\right)\subseteq C_i$ in both lines of \eqref{equ-C-case3}, so that (4) holds as well. A direct inspection of \eqref{equ-C-case3} shows that
\begin{equation}\label{equ-C-dichotomy}
\text{either } [-1,1) \subseteq C_i \text{ or } [-1,1) \cap C_i = \varnothing
\qquad (1 \leq i \leq d):
\end{equation}
for $m_i \geq 2$ one has $C_i \subseteq [1,\infty)$; for $-1 \leq m_i \leq 1$ one has $C_i \supseteq [-1,1)$; for $m_i = -2$ one has $C_i = [-3,1) \supseteq [-1,1)$; and for $m_i \leq -3$ one has $C_i \subseteq (-\infty,-1)$.

We verify (2). Let $Q' \in \Delta^\alpha$ with $\rho(Q,Q') < 1$ and $H := l(Q')$. If $H \leq 1$, then \eqref{equ-edge-reach} gives $Q'_i \subseteq [m_i-1,m_i+2) \subseteq C_i$ for every $i$, by \eqref{equ-C-case3}. If $H = 2$, then $Q' = [-1,1)^d$, and adjacency of $Q_i$ and $Q'_i$ forces $-2 \leq m_i \leq 1$ for every $i$; in that range $C_i \supseteq [-1,1) = Q'_i$ by \eqref{equ-C-dichotomy}.

For (1), let $P \in \Delta^\alpha$ meet $C$. If $l(P) \leq 1$, then $P$ is aligned with a lattice refining $\Z^d$, and since the endpoints of the $C_i$ are integers, $P \subseteq C$. The only remaining cube is $P = [-1,1)^d$; if it meets $C$, then $[-1,1) \cap C_i \neq \varnothing$ for every $i$, hence $[-1,1) \subseteq C_i$ for every $i$ by \eqref{equ-C-dichotomy}, i.e.\ $P \subseteq C$. Again only finitely many cubes are involved.
\end{proof}

\begin{defi}\label{OKOKJHJDHJD}\textbf{(Region $N_\alpha(Q)$)}
For every $Q \in \Delta^\alpha$ we fix, once and for all, a cube $N_\alpha(Q)$ with the properties of Lemma \ref{lem-N-alpha-exists}; that is, $N_\alpha(Q)$ is a cube of side length $4l(Q)$ which is a finite union of cubes of $\Delta^\alpha$, which admits $Q$ as a dyadic subcube, which contains the set $\{x : \rho(x,Q)<l(Q)\}$, and which in particular contains every $Q' \in \Delta^\alpha$ with $\rho(Q,Q') < l(Q)$. Such a cube is in general not unique --- for $\alpha = \frac12$ and $Q = [4.5,5)$ both $[3.5,5.5)$ and $[4,6)$ qualify --- which is why the choice is fixed beforehand; otherwise the supremum in Definition \ref{defi-Ap-alpha} below would range over an unspecified family. An example is shown in Figure \ref{grille} of the introduction, with the cube $R$ colored in red and one admissible region $N_{\frac12}(R)$ in yellow.

Two consequences will be used repeatedly. First, every $Q' \in \Delta^\alpha$ is either contained in $N_\alpha(Q)$ or disjoint from it. Second, if $Q' \cap N_\alpha(Q) = \varnothing$, then $\rho(Q,Q') \geq l(Q)$.
\end{defi}

\begin{lem}\label{oh}
Let $Q$ and $R$ be two cubes of $\Delta^\alpha$. Then
\[
r(Q,R) \leq \frac{\rho(Q,R)+\max(l(Q),l(R))}{\min(l(Q),l(R))}.
\]
\end{lem}

\begin{proof}
If $Q = R$, then $r(Q,R) = 0$ and there is nothing to prove. We may therefore assume $Q \neq R$; in particular $Q$ and $R$ are not both equal to $[-1,1)^d$.

\emph{Notation.} Let $j_Q$ resp. $j_R$ be the index of the layer containing $Q$ resp. $R$, and set
\[
J := \max(j_Q,j_R), \qquad \delta := \delta^\alpha_J .
\]
Then $l(Q) = 2^{-\delta^\alpha_{j_Q}}$ and $l(R) = 2^{-\delta^\alpha_{j_R}}$, so that by \eqref{equ-delta-monotone},
\begin{equation}\label{equ-delta-min}
2^{-\delta} = \min\left(2^{-\delta^\alpha_{j_Q}},2^{-\delta^\alpha_{j_R}}\right) = \min(l(Q),l(R)).
\end{equation}
Moreover $\delta \geq 0$: otherwise $J = 0$, hence $j_Q = j_R = 0$ and $Q = R = [-1,1)^d$, which we have excluded. We write $Q = m_Q + [0,l(Q))^d$ and $R = m_R + [0,l(R))^d$ and denote by $(m_{Q,i})_{i=1}^d$ resp. $(m_{R,i})_{i=1}^d$ the coordinates of $m_Q$ resp. $m_R$. By the definition of $\Delta^\alpha$ we have $2^{\delta^\alpha_{j_Q}} m_Q \in \Z^d$ if $j_Q \geq 1$, and $m_Q = (-1,\ldots,-1) \in \Z^d$ if $j_Q = 0$; since $\delta \geq \max(0,\delta^\alpha_{j_Q})$, in both cases
\begin{equation}\label{equ-lattice}
m_Q \in 2^{-\delta}\Z^d, \qquad \text{and in the same way} \qquad m_R \in 2^{-\delta}\Z^d .
\end{equation}

\emph{The path.} We define $n^0 := m_Q$ and, recursively for $k \geq 0$ and $i = 1,\ldots,d$,
\[
n_i^{k+1} :=
\begin{cases}
n^k_i + 2^{-\delta} & \text{if } m_{R,i} > n^{k}_i,\\[4pt]
n^k_i & \text{if } m_{R,i} = n^{k}_i,\\[4pt]
n^k_i - 2^{-\delta} & \text{if } m_{R,i} < n^{k}_i,
\end{cases}
\]
and we set
\[
M := \max\left\{2^{\delta}\left|m_{Q,i}-m_{R,i}\right| \;:\; i \in \{1,\ldots,d\}\right\},
\]
which is a nonnegative integer by \eqref{equ-lattice}. We claim that
\begin{equation}\label{equ-path}
n^M = m_R
\qquad \text{and} \qquad
\min(m_{Q,i},m_{R,i}) \leq n^k_i \leq \max(m_{Q,i},m_{R,i}) \quad (k \geq 0,\ 1 \leq i \leq d).
\end{equation}
Indeed, fix $i$. By \eqref{equ-lattice} we have $n^k_i - m_{R,i} \in 2^{-\delta}\Z$ for every $k$, so the $i$-th coordinate moves towards $m_{R,i}$ by steps of length $2^{-\delta}$ without overshooting it, and remains equal to $m_{R,i}$ once it has been reached; this happens after exactly $2^{\delta}|m_{Q,i}-m_{R,i}| \leq M$ steps.

\emph{The chain.} Since $j_Q \leq J$ we have $m_Q \in Q \subseteq L^\alpha_{j_Q} \subseteq [-(j_Q+1),j_Q+1)^d \subseteq [-(J+1),J+1)^d$, and likewise $m_R \in [-(J+1),J+1)^d$. As $[-(J+1),J+1)$ is an interval, \eqref{equ-path} gives
\[
n^k \in [-(J+1),J+1)^d \qquad (k \geq 0).
\]
% Note that the half-open box cannot be replaced by $\|n^k\|_\infty < J+1$ here:
% if $Q$ sits at an outer corner of its layer, then $m_{Q,i} = -(j_Q+1)$ for
% some $i$, so that already $\|n^0\|_\infty = j_Q+1$ may occur.
Together with \eqref{equ-lattice} and $\delta = \delta^\alpha_J \geq 0$, Lemma \ref{lem-small-cube} applies to each $n^k$ and yields cubes
\[
Q^k := Q_{n^k} \in \Delta^\alpha
\qquad \text{with} \qquad
n^k + [0,2^{-\delta})^d \subseteq Q^k .
\]
In particular $Q^0 = Q_{m_Q} = Q$ and, by \eqref{equ-path}, $Q^M = Q_{m_R} = R$.

We next check that $Q^k$ and $Q^{k+1}$ are neighbors or coincide. Let $\epsilon \in (0,2^{-\delta})$ and define $x = (x_i)_{i=1}^d$ and $y = (y_i)_{i=1}^d$ by
\[
(x_i,y_i) :=
\begin{cases}
\left(n^k_i + 2^{-\delta}-\epsilon,\; n^{k+1}_i\right) & \text{if } n^{k+1}_i = n^k_i + 2^{-\delta},\\[4pt]
\left(n^k_i,\; n^{k+1}_i\right) & \text{if } n^{k+1}_i = n^k_i,\\[4pt]
\left(n^k_i,\; n^{k+1}_i + 2^{-\delta}-\epsilon\right) & \text{if } n^{k+1}_i = n^k_i - 2^{-\delta}.
\end{cases}
\]
Then $x \in n^k + [0,2^{-\delta})^d \subseteq Q^k$ and $y \in n^{k+1} + [0,2^{-\delta})^d \subseteq Q^{k+1}$, and $|x_i - y_i| \leq \epsilon$ for every $i$, so that $\rho(Q^k,Q^{k+1}) \leq \epsilon$. As $\epsilon \in (0,2^{-\delta})$ was arbitrary, $\rho(Q^k,Q^{k+1}) = 0$.

Thus $Q = Q^0,\,Q^1,\ldots,Q^M = R$ is a chain in $\Delta^\alpha$ whose consecutive members are neighbors, possibly after deleting repetitions, which only shortens the chain. Consequently, by \eqref{equ-delta-min},
\begin{equation}\label{equ-r-bound}
r(Q,R) \leq M = \max\left\{\frac{|m_{Q,i}-m_{R,i}|}{\min(l(Q),l(R))} \;:\; i \in \{1,\ldots,d\}\right\}.
\end{equation}

\emph{Conclusion.} For $x \in Q$ and $y \in R$ write $x = m_Q + s$ and $y = m_R + t$ with $s \in [0,l(Q))^d$ and $t \in [0,l(R))^d$. Then $|s_i - t_i| \leq \max(l(Q),l(R))$ for every $i$, whence
\[
\rho(Q,R) = \inf\left\{\max_{1 \leq i \leq d}|x_i-y_i| \;:\; x \in Q,\ y \in R\right\}
\geq \max_{1 \leq i \leq d}\left|m_{Q,i}-m_{R,i}\right| - \max(l(Q),l(R)).
\]
Combining this with \eqref{equ-r-bound} gives
\[
r(Q,R) \leq \frac{\rho(Q,R)+\max(l(Q),l(R))}{\min(l(Q),l(R))},
\]
which is the assertion.
\end{proof}

\begin{rem}\label{HDHDKDKDKLDJJDLK}
Let $Q$ and $R$ be two cubes of $\Delta^\alpha$. Then
\begin{equation}\label{OJDHDKHDKJHDKJDHKJDLHKDLJDKLDK}
r(Q,R) \;\leq\; \frac{2\rho(Q,R)}{\min(l(Q),l(R))} + 8 .
\end{equation}
Indeed, if $\rho(Q,R) \geq \max(l(Q),l(R))$, then Lemma \ref{oh} gives
\[
r(Q,R) \leq \frac{\rho(Q,R)+\max(l(Q),l(R))}{\min(l(Q),l(R))} \leq \frac{2\rho(Q,R)}{\min(l(Q),l(R))} .
\]
If $\rho(Q,R) < \max(l(Q),l(R))$, assume without loss of generality that $l(Q) \geq l(R)$, so that $\rho(Q,R) < l(Q)$. Then \eqref{equ-size-ratio} yields $l(R) \geq \frac14 l(Q)$, and Lemma \ref{oh} gives
\[
r(Q,R) \leq \frac{\rho(Q,R)+\max(l(Q),l(R))}{\min(l(Q),l(R))} < \frac{2\max(l(Q),l(R))}{\min(l(Q),l(R))} \leq 8 .
\]
\end{rem}

\begin{lem}\label{cardinalité2}
  Let $Q\in \Delta^\alpha$. Then the cardinality
  \begin{equation*}
        |\{R\in \Delta^\alpha\;:\;N_\alpha(R)\cap Q\ne \emptyset\}|\leq 2^{20d}.
        \end{equation*}
    \end{lem}
\begin{proof}
Let $R \in \Delta^\alpha_{n'}$ with $N_\alpha(R) \cap Q \neq \varnothing$ and let $Q \in \Delta^\alpha_n$. Since $N_\alpha(R)$ is a cube of side length $4l(R)$ containing $R$, and since $l(R) \leq 2$, we get $\rho(Q,R) < 4l(R) \leq 8$. As in the proof of Lemma \ref{lem-neighbour-sizes}(2) we have $\rho(Q,R) \geq \max(0,|n-n'|-1)$, whence $|n-n'| \leq 8$. By Lemma \ref{lem-neighbour-sizes}(1) and the monotonicity of $(\delta^\alpha_k)_{k \geq 0}$, this gives
\[ 2^{-8}\,l(Q) \leq l(R) \leq 2^{8}\,l(Q) . \]
Consequently $R$ lies within $\ell^\infty$-distance $4l(R) \leq 2^{10}l(Q)$ of $Q$ and has side length at most $2^8 l(Q)$, so that $R$ is contained in the cube concentric with $Q$ of side length $\left(1+2\cdot 2^{10}+2\cdot 2^8\right)l(Q) \leq 2^{12}l(Q)$. The cubes $R$ being pairwise disjoint with $|R| \geq \left(2^{-8}l(Q)\right)^d$, their number is at most $\left(2^{12}/2^{-8}\right)^d = 2^{20d}$.
\end{proof}

\begin{lem}\label{cardi}\textbf{(Cardinality of $\Delta^\alpha_n$, case $n\geq 1$)} Denote by $\#\Delta^\alpha_n$ the cardinality of $\Delta^\alpha_n$. Then
    \begin{equation*}
        \#\Delta^\alpha_n< 2^{3d} n^{d(1+\alpha)}.
    \end{equation*}
\end{lem}
\begin{proof}\,
 Let $R$ be a cube of $\Delta_n^\alpha$ with $n\geq 1$. Then
 \begin{equation*}
     \#\Delta^\alpha_n=\frac{|[-n-1,n+1)^d|-|[-n,n)^d|}{|R|}=\frac{(2(n+1))^d-(2n)^d}{2^{-d\delta^\alpha_n}}=\frac{(2n)^d\Big[(1+\frac{1}{n})^d-1\Big]}{2^{-d\delta^\alpha_n}}.
 \end{equation*}
 Since $(1+\frac{1}{n})^d\leq 2^d$ (because $n\geq 1$) and $n^{-\alpha}<2^{-\delta_n^\alpha +1}$, the above gives
 \begin{equation*}
\#\Delta^\alpha_n=\frac{(2n)^d\Big[(1+\frac{1}{n})^d-1\Big]}{2^{-d\delta^\alpha_n}}\leq\frac{(2n)^d(2^d-1)}{2^{-d\delta^\alpha_n}}<2^dn^{\alpha d}(2n)^d(2^d-1)\leq 2^{3d}n^{d+\alpha d}
 \end{equation*}
\end{proof}
\begin{lem}\label{exApa}
  Let $\alpha\in [0,1)$, $0<\gamma\leq \alpha+1$, and let $Q\in\Delta^\alpha_n$ with $n\geq 6$. Then for all $x,y \in N_\alpha(Q)$,
\begin{equation*}
    ||x||^\gamma_\infty\leq ||y||^\gamma_\infty+c
\end{equation*}
where $c$ is a constant independent of $x$, $y$, and $Q$.
\end{lem}
\begin{proof} Since $N_\alpha(Q)$ is a cube of side length $4l(Q)\leq 4$ containing $Q$, and since every $z\in Q$ satisfies $\|z\|_\infty\in[n,n+1]$, every $y\in N_\alpha(Q)$ satisfies
\begin{equation}\label{equ-exApa-loc}
    n-4\leq\|y\|_\infty\leq n+5 .
\end{equation}
As $n\geq 6$, this gives $\|y\|_\infty\geq 2\geq 1$ and $\frac{n}{3}\leq\|y\|_\infty\leq 3n$, whence $l(Q)=2^{-\delta^\alpha_n}\simeq n^{-\alpha}\simeq\|y\|_\infty^{-\alpha}$ with absolute implicit constants. For $x$ and $y \in N_\alpha(Q)$, we thus know that $\|x-y\|_\infty \leq 4l(Q) \simeq 4\|y\|_\infty^{-\alpha}$. Thus, by the following triangle inequality
 \begin{equation*}
     ||(x-y)+y||_\infty\leq ||x-y||_\infty+||y||_\infty,
 \end{equation*} 
we arrive at
\begin{align}\label{fnfhfbdhdbdhdblld}
    \|x\|_\infty^\gamma &\leq (\|y\|_\infty + C\|y\|_\infty^{-\alpha})^\gamma\\
    &= \|y\|_\infty^\gamma (1 + C\|y\|_\infty^{-\alpha-1})^\gamma.
\end{align}
Now we observe that
    \begin{equation}\label{GDNIDHNJDNJNDBJDNDNDON}
        (1+C||y||_\infty^{-\alpha-1})^\gamma\leq 1+C(1+C)\gamma||y||_\infty^{-\alpha-1}.
    \end{equation}
    Indeed, by applying the mean value theorem to the function $t \mapsto f(t) = (1+t)^\gamma$ on $[0, C\|y\|_\infty^{-\alpha-1}]$, there exists $\theta$ in $(0,C\|y\|_\infty^{-\alpha-1})$ such that
    \begin{equation*}
        f(C||y||_\infty^{-\alpha-1})-f(0)=C\gamma(1+\theta)^{\gamma-1}||y||_\infty^{-\alpha-1}.
    \end{equation*}
    Thus,
\begin{equation*}
  (1+C\|y\|_\infty^{-\alpha-1})^\gamma - 1 = C\gamma (1+\theta)^{\gamma-1}\|y\|_\infty^{-\alpha-1}.
\end{equation*}
However, since $\theta < C\|y\|_\infty^{-\alpha-1}$ and $\|y\|_\infty \geq 1$, we have
    \begin{equation*}
        (1+\theta)^{\gamma-1}\leq (1+C)^{\gamma-1}\leq 1+C, \;\;\;\text{if}\;\;\;1\leq \gamma<2.
    \end{equation*}
    and 
     \begin{equation*}
        (1+\theta)^{\gamma-1}\leq 1\leq 1+C, \;\;\;\text{if}\;\;\;0\leq \gamma\leq 1.
    \end{equation*}
    This justifies equation (\ref{GDNIDHNJDNJNDBJDNDNDON}). Now, replacing in equation (\ref{fnfhfbdhdbdhdblld}) the upper bound given in equation (\ref{GDNIDHNJDNJNDBJDNDNDON}), we obtain

       \begin{align*}
        ||x||_\infty^\gamma
        &\leq ||y||_\infty^\gamma(1+C(1+C)\gamma||y||_\infty^{-\alpha-1})\\
        &=||y||_\infty^\gamma+C(1+C)\gamma||y||_\infty^{\gamma-\alpha-1}\\
       &\leq ||y||_\infty^\gamma+ c,
    \end{align*}
    since $\gamma-\alpha-1\leq 0$.
\end{proof}
\begin{rem}\label{jfbdbkbdhbdkbddbd}
   Note that under the same hypotheses of Lemma \ref{exApa}, exchanging the roles of $x$ and $y$ in the proof, we obtain in the same way that
\begin{equation*}
    ||y||^\gamma_\infty\leq ||x||^\gamma_\infty+c'
\end{equation*}
where $c'$ is a constant independent of $x$, $y$, and $Q$.
\end{rem}

\subsection{Definition and properties of the  weight class $A_p^\alpha$}
\label{subsec-4-2}

 \begin{defi}\label{defi-Ap-alpha}\textbf{($A_p^\alpha$ class)}
For any fixed $1<p<\infty$ and $\omega$ a weight on $\mathbb{R}^d$, we say that $\omega$ belongs to $A_p^\alpha$ if

\begin{equation}
    [\omega]_{A_p^\alpha}:=\sup_{R\in \Delta^\alpha}\sup_{Q\subseteq N_\alpha(R)} \Big(\frac{1}{|Q|}\int_{Q}\omega(x) dx\Big)^\frac{1}{p} \Big(\frac{1}{|Q|}\int_{Q}\omega^{-\frac{1}{p-1}}(x)dx\Big)^\frac{p-1}{p}<\infty,
\end{equation}
 where the supremum is taken over all $R\in \Delta^\alpha$ and over all cubes $Q\subseteq N_\alpha(R)$ with sides parallel to the coordinate axes.
\end{defi}
We will begin with some remarks on the weight class $A_p^\alpha$, and then we will present its properties.

\begin{rem}\label{JabuUU} We have a version of Lemma \ref{Jabu} for weights in the class $A_p^\alpha$. That is, for $1<p,p'<\infty$ conjugate exponents, if $\w\in A_p^{\alpha}$, then $\w^{-\frac{p'}{p}}\in A_{p'}^{\alpha}$. Indeed, by setting $\sigma=\w^{-\frac{p'}{p}}$, for $Q$ a cube contained in $N_\alpha(R)$, we observe that
\begin{align*}
    \sigma(Q)^\frac{1}{p'}\sigma^{-\frac{p}{p'}}(Q)^\frac{1}{p}&=\w^{-\frac{p'}{p}}(Q)^\frac{1}{p'}\w(Q)^\frac{1}{p}\\
    &\leq [\w]_{A_p^{\alpha}}|Q|.
\end{align*}
We also deduce that $[\w]_{A_p^\alpha}=[\w^{-\frac{p'}{p}}]_{A_{p'}^{\alpha}}$.
\end{rem}

In the following remark we give a version of Proposition \ref{lol} for weights in the class $A_p^\alpha$.

\begin{rem}\label{SCDHDICOCGISCDRO} By replacing $N(R)$ with $N_\alpha(R)$ in the proof of Proposition \ref{lol}, we show that for all $1<p<\infty$ and  $\w\in A_p^{\alpha}$
\begin{equation*}
\w(N_\alpha(R))\leq[\w]^p_{A_p^\alpha}\Big(\frac{|N_\alpha(R)|}{|Q|}\Big)^p\w(Q),
\end{equation*} 
for any cube $Q\subseteq N_\alpha(R)$.
\end{rem}

\begin{rem}\label{emm} For $Q$ and $R$ two cubes of $\Delta^\alpha$ and $\w\in A_p^\alpha$, then
   \begin{equation}\label{ulr1}
        \w(R)\leq (4^{pd}[\w]^p_{A_p^\alpha})^{r(Q,R)}\w(Q),
    \end{equation}
    and
 \begin{equation}\label{ulr2}
        \w^{-\frac{p'}{p}}(R)^\frac{p}{p'}\leq (4^{pd}[\w]^p_{A_p^\alpha})^{r(Q,R)}\w^{-\frac{p'}{p}}(Q)^\frac{p}{p'}.
    \end{equation}

Before justifying equations $(\ref{ulr1})$ and $(\ref{ulr2})$, let us first observe that, by Remarks \ref{JabuUU} and \ref{SCDHDICOCGISCDRO} applied with the cube $R\subseteq N_\alpha(R)$ and with $|N_\alpha(R)| = 4^d|R|$, we have that
  \begin{equation}\label{CKK1}
         \w(N_\alpha(R))\leq (4^{pd}[\w]^p_{A_p^{\alpha}})\w(R),
    \end{equation}
   and  
    \begin{equation}\label{CKK2}
      \w^{-\frac{p'}{p}}(N_\alpha(R))^{\frac{p}{p'}}\leq  (4^{pd}[\w]^p_{A_p^{\alpha}})\w^{-\frac{p'}{p}}(R)^{\frac{p}{p'}}.
\end{equation}

Now by equations (\ref{CKK1}) and (\ref{CKK2}), for $R$ and $S$ two cubes of $\Delta^\alpha$ that are neighbors (i.e. such that $\rho(R,S)=0$), we have $\w(S)\leq \w(N_\alpha(R))\leq 4^{pd}[\w]_{A_p^\alpha}^p \w(R)$.
Then let $(R_k)_{k=0}^{r(Q,R)}$ be a chain of elements of $\Delta^\alpha$ such that $R_k$ and $R_{k+1}$ are neighbors and such that $R_0=Q$ and $R_{r(Q,R)}=R$.
Then we observe by equation (\ref{CKK1}) applied step by step that
\begin{align*}
   \w(R)=\w(R_{r(Q,R)})&\leq 4^{pd}[\w]_{A_p^\alpha}^p \w(R_{r(Q,R)-1}) \leq \ldots \leq \w(R_0)\prod_{k=1}^{r(Q,R)}(4^{pd}[\w]_{A_p^\alpha}^p)\\
   &=(4^{pd}[\w]_{A_p^\alpha}^p)^{r(Q,R)}\w(Q)
\end{align*}
Consequently, we arrive at equation (\ref{ulr1}).
Similarly, equation (\ref{ulr2}) is justified in the same way due to equation (\ref{CKK2}).
\end{rem}

\begin{lem}\label{exep}
Let $\alpha\in [0,1)$ and $\gamma\leq \alpha+1$. Then $\omega_0(x)=\exp({||x||_\infty^\gamma})$ belongs to the $A_p^\alpha$ class.
\end{lem}
\begin{proof}
Let $B\subseteq N_\alpha(Q)$ for some $Q\in \Delta^\alpha_n$. Consider first the case $n\leq 5$. Then $Q\subseteq[-6,6)^d$ and $N_\alpha(Q)$ is a cube of side length $4l(Q)\leq 8$ containing $Q$, so that $N_\alpha(Q)\subseteq[-14,14)^d$ and hence $1\leq\omega_0\leq e^{14^\gamma}$ on $B$. Since $\omega_0\geq 1$ gives $\omega_0^{-\frac{1}{p-1}}\leq 1$, we obtain
    \begin{equation*}
        \omega_0(B)\omega_0^{-\frac{1}{p-1}}(B)^{p-1}\leq e^{14^\gamma}|B|^p .
    \end{equation*}
Now, for $n\geq 6$, by Lemma \ref{exApa} and Remark \ref{jfbdbkbdhbdkbddbd}, we can observe for $x$ and $x_0$ in $B$ that $\omega_0(x)\leq c\omega_0(x_0)$ and that $\omega_0(x)\geq C\omega_0(x_0)$. Thus, we then have that
    \begin{align*}
        \omega_0(B)\omega_0^{-\frac{1}{p-1}}(B)^{p-1}&\leq c\omega_0(x_0)|B|(C^{-\frac{1}{p-1}}\omega_0(x_0)^{-\frac{1}{p-1}})^{p-1}|B|^{p-1}\\
        &\leq c C^{-1}|B|^p.
    \end{align*}
    This concludes the proof.
\end{proof}

\begin{prop}\label{Proposorijfnfjj} Let $\alpha\in[0,1)$ and let $\gamma>\alpha+1$. Then the function $x\mapsto \w_0(x)=\exp(||x||_\infty^\gamma)$ does not belong to $A_p^\alpha$. In particular $\exp(||x||_\infty^2)\notin A_p^\alpha$ for every $\alpha\in[0,1)$.
\end{prop}
\begin{proof} Note first that $\gamma>\alpha+1\geq 1$. For $n\geq 1$ put $h:=2^{-\delta^\alpha_n}$ and
\begin{equation*}
    R_n:=[n,n+h)\times[0,h)^{d-1}.
\end{equation*}
Then $R_n\in\Delta^\alpha_n$. Indeed $2^{\delta^\alpha_n}(n,0,\dots,0)\in\Z^d$, and since $h\leq 1\leq n$ every $x\in R_n$ satisfies
\begin{equation}\label{equ-Rn-sup}
    \|x\|_\infty=x_1\in[n,n+1),
\end{equation}
so that $R_n\subseteq[-(n+1),n+1)^d\setminus[-n,n)^d$. Moreover $R_n\subseteq N_\alpha(R_n)$, so that $R_n$ is admissible in Definition \ref{defi-Ap-alpha}. By \eqref{equ-Rn-sup} and Tonelli's theorem,
\begin{equation*}
    \int_{R_n}e^{\|x\|^\gamma_\infty}dx\Big(\int_{R_n}e^{-\frac{p'}{p}\|x\|^\gamma_\infty}dx\Big)^\frac{p}{p'}
    = h^{(d-1)p}\,I_n\,J_n^{\frac{p}{p'}},
\end{equation*}
where we used $1+\frac{p}{p'}=p$ and where
\begin{equation*}
    I_n:=\int_n^{n+h}e^{t^\gamma}dt,\qquad J_n:=\int_n^{n+h}e^{-\frac{p'}{p}t^\gamma}dt .
\end{equation*}
As $|R_n|^p=h^{dp}$, it therefore suffices to prove that
\begin{equation}\label{equ-Rn-goal}
    \frac{I_nJ_n^{\frac{p}{p'}}}{h^p}\xrightarrow[\;n\to\infty\;]{}\infty .
\end{equation}
Put $A_n:=\gamma(n+h)^{\gamma-1}$ and $D_n:=(n+h)^\gamma-n^\gamma$. Since $\gamma\geq 1$, we have $\gamma t^{\gamma-1}\leq A_n$ for $t\in[n,n+h]$, whence
\begin{equation*}
    I_n\geq\frac{1}{A_n}\int_n^{n+h}\gamma t^{\gamma-1}e^{t^\gamma}dt=\frac{e^{n^\gamma}\left(e^{D_n}-1\right)}{A_n}
\end{equation*}
and, in the same way,
\begin{equation*}
    J_n\geq\frac{p}{p'A_n}\int_n^{n+h}\frac{p'}{p}\gamma t^{\gamma-1}e^{-\frac{p'}{p}t^\gamma}dt=\frac{p}{p'}\,\frac{e^{-\frac{p'}{p}n^\gamma}\left(1-e^{-\frac{p'}{p}D_n}\right)}{A_n}.
\end{equation*}
Since $e^{n^\gamma}\big(e^{-\frac{p'}{p}n^\gamma}\big)^{\frac{p}{p'}}=1$ and $A_n^{-1}A_n^{-\frac{p}{p'}}=A_n^{-p}$, we obtain
\begin{equation}\label{equ-Rn-ratio}
    \frac{I_nJ_n^{\frac{p}{p'}}}{h^p}\geq\Big(\frac{p}{p'}\Big)^{\frac{p}{p'}}\frac{\left(e^{D_n}-1\right)\left(1-e^{-\frac{p'}{p}D_n}\right)^{\frac{p}{p'}}}{\left(A_nh\right)^p}.
\end{equation}
It remains to compare $D_n$ and $A_nh$ with $u_n:=n^{\gamma-1-\alpha}$. From $\frac12 n^{-\alpha}<h\leq n^{-\alpha}\leq 1$, from $n+h\leq 2n$ and from $\gamma\geq 1$ we get
\begin{equation*}
    D_n\geq\gamma n^{\gamma-1}h>\frac{\gamma}{2}u_n,
    \qquad
    A_nh\leq\gamma(2n)^{\gamma-1}h\leq\gamma\,2^{\gamma-1}u_n .
\end{equation*}
Both $\sigma\mapsto e^\sigma-1$ and $\sigma\mapsto 1-e^{-\frac{p'}{p}\sigma}$ being increasing, \eqref{equ-Rn-ratio} yields
\begin{equation*}
    \frac{I_nJ_n^{\frac{p}{p'}}}{h^p}\geq\Big(\frac{p}{p'}\Big)^{\frac{p}{p'}}\frac{\left(e^{\frac{\gamma}{2}u_n}-1\right)\left(1-e^{-\frac{p'\gamma}{2p}u_n}\right)^{\frac{p}{p'}}}{\left(\gamma\,2^{\gamma-1}u_n\right)^p}.
\end{equation*}
Since $\gamma-1-\alpha>0$ we have $u_n\to\infty$, and the right hand side tends to $\infty$ because the exponential dominates any power of $u_n$. This proves \eqref{equ-Rn-goal} and hence the proposition.
\end{proof}

\begin{lem}\label{ahto}Let $1 < p < \infty$ and $\alpha \in [0,1)$. If $\w \in A_p^\alpha$, then $\w \in A_p(\zeta[-1,1)^d)$ as given in Definition \ref{local muck}. Moreover,
\begin{equation*}
    [\w]_{ A_p(\zeta[-1,1)^d)}\leq C([\w]_{A^\alpha_p}),
\end{equation*}
where $C([\w]_{A^\alpha_p})$ depends on the characteristic $[\w]_{A^\alpha_p}$ and on the real constant $\zeta$.
\end{lem}

\begin{proof}\,
  Let $\zeta[-1,1)^d$ be a cube centered at $0$ and with side $l(\zeta[-1,1)^d)\leq 2(\lfloor \zeta \rfloor+1 )$.
  Let $Q$ be a cube contained in $\zeta[-1,1)^d$. Then, there exists a finite family $\mathcal{F}:=(Q_k)_{k=1}^m$ of elements of $\Delta^\alpha$, such that $Q\subseteq \bigsqcup_{k= 1}^mQ_k$ and $\rho(Q_k,0)\leq \zeta$, for all $k$.\

 Let $Q_\beta$ be a cube of $\mathcal{F}$ such that $l(Q_k)\geq l(Q_\beta)$ for all $k\in\{1,\ldots,m\}$, and let $\Delta^\alpha_{n'}$ be the layer containing $Q_\beta$.\\

   \textit{Case $l(Q)\geq l(Q_\beta)=2^{-\delta^\alpha_{n'}}$:}\\

  By Remark \ref{emm}, for $k$ fixed
 \begin{equation}\label{fd1}
    \w(Q_k)\leq([\w]^p_{A_p^\alpha}4^{pd})^{r(Q_k,Q_\beta)}\w(Q_\beta),
\end{equation}
and
\begin{equation}\label{fd2}
    \w^{-\frac{p'}{p}}(Q_k)^\frac{p}{p'}\leq ([\w]^p_{A_p^\alpha}4^{pd})^{r(Q_k,Q_\beta)}\w^{-\frac{p'}{p}}(Q_\beta)^\frac{p}{p'},
\end{equation}
with $r(Q_k,Q_\beta)$ representing the number of steps to go from $Q_k$ to $Q_\beta$, for all $k\in\{1,...,m\}$. Furthermore, by Remark \ref{HDHDKDKDKLDJJDLK}, we know that
\begin{align*}
 r(Q_k,Q_\beta)&\leq\frac{2\rho(Q_k,Q_\beta)}{\min(l(Q_k),l(Q_\beta))}+8\\
 &=2l(Q_\beta)^{-1}\rho(Q_k,Q_\beta)+8\\
 &= 2^{\delta^\alpha_{n'}+1} \rho(Q_k,Q_\beta)+8\\
 &\leq2^{\delta^\alpha_{n'}} 4(\lfloor \zeta\rfloor+1)+8.
\end{align*}
Thus, $ \w(Q_k)\leq C\w(Q_\beta)$ and $ \w^{-\frac{p'}{p}}(Q_k)^\frac{p}{p'}\leq C \w^{-\frac{p'}{p}}(Q_\beta)^\frac{p}{p'}$, with $C$ a non-negative constant depending on $\zeta$. We then observe that
\begin{align*}
    \w(Q)\w^{-\frac{p'}{p}}(Q)^\frac{p}{p'}&=\Big(\int_{Q}\w(x)dx\Big)\Big(\int_{Q}\w^{-\frac{p'}{p}}(x)dx\Big)^{\frac{p}{p'}}\\
&\leq\Big(\sum_{k=1}^m\int_{Q_k}\w(x)dx\Big)\Big(\sum_{k=1}^m\int_{Q_k}\w^{-\frac{p'}{p}}(x)dx\Big)^{\frac{p}{p'}}\\
&\leq C^2\Big(\sum_{k=1}^m\w(Q_\beta)\Big)\Big(\sum_{k=1}^m\w^{-\frac{p'}{p}}(Q_\beta)\Big)^{\frac{p}{p'}}\\
&\leq C^2[\w]_{A_p^\alpha}^{p}m^p|Q_\beta|^p\\
&\leq K|Q|^p,
\end{align*}
with $K$ a non-negative constant depending on $[\w]_{A_p^\alpha}$, since $m$ and $C$ are bounded above once $\zeta$ is fixed.\\

\textit{case $l(Q)< l(Q_\beta)=2^{-\delta^\alpha_{n'}}$:}\\

Since $\mathcal{F}$ covers $Q$ and $Q_\beta$ is such that $l(Q_\beta)\leq l(Q_k)$, for all $k\in\{1,... , m\}$, then $l(Q)<l(Q_k)$ for all $k\in\{1,...,m\}$.
Moreover, there always exists a cube $\tilde{Q}$, element of $\mathcal{F}$, such that $\rho(Q,\tilde{Q})=0$. Every $x\in Q$ then satisfies $\rho(x,\tilde{Q})\leq l(Q)<l(\tilde{Q})$, so that $Q\subseteq N_\alpha(\tilde{Q})$ by Lemma \ref{lem-N-alpha-exists}(4). Since the supremum in Definition \ref{defi-Ap-alpha} is taken over all cubes contained in $N_\alpha(\tilde{Q})$, we obtain $\omega(Q)\omega^{-\frac{p'}{p}}(Q)^{\frac{p}{p'}} \leq [\omega]^p_{A_p^\alpha} |Q|^p$.
\end{proof}

\begin{lem}\label{lem-inclusions}
Let $1<p<\infty$.
\begin{enumerate}
\item If $0\leq\alpha_1\leq\alpha_2<1$, then $A_p^{\alpha_1}\subseteq A_p^{\alpha_2}$.
\item If $\alpha\in[0,1)$, then $A_p^{\alpha}\subseteq A_p^{\loc}$.
\end{enumerate}
In both cases, the characteristic of $\w$ in the larger class is bounded by a constant depending only on its characteristic in the smaller class, on $d$, on $p$ and on the parameters $\alpha_1,\alpha_2$ resp.\ $\alpha$.
\end{lem}

\begin{proof}
We shall use twice the following observation. Let $\beta\in[0,1)$, let $S\subseteq\R^d$ be a cube of side length $\lambda$ and let $z\in S$. Denote by $Q_z^\beta$ the cube of $\Delta^\beta$ containing $z$ and assume that
\begin{equation}\label{equ-incl-hyp}
    \lambda\leq\tfrac12\,l(Q_z^\beta).
\end{equation}
Then every $x\in S$ satisfies $\rho(x,Q_z^\beta)\leq\|x-z\|_\infty\leq\lambda<l(Q_z^\beta)$, so that $S\subseteq N_\beta(Q_z^\beta)$ by Lemma \ref{lem-N-alpha-exists}(4). Since the supremum in Definition \ref{defi-Ap-alpha} is taken over all cubes contained in $N_\beta(Q_z^\beta)$, this gives
\begin{equation}\label{equ-incl-transfer}
    [\w]_{A_p(S)}\leq[\w]_{A_p^{\beta}} .
\end{equation}

(1) If $\alpha_1=\alpha_2$ there is nothing to prove, so assume $\alpha_1<\alpha_2$ and let $\w\in A_p^{\alpha_1}$. Put
\begin{equation*}
    n_1:=\max\left(7,\left\lceil 32^{\frac{1}{\alpha_2-\alpha_1}}\right\rceil\right)
\end{equation*}
and let $R\in\Delta^{\alpha_2}_n$.

Assume first $n\geq n_1$. The cube $S:=N_{\alpha_2}(R)$ has side length $\lambda=4l(R)\leq 4n^{-\alpha_2}\leq 4$. Fix $z\in S$ and let $m\geq 0$ be the index of the $\Delta^{\alpha_1}$-layer containing $z$. Since $R\subseteq L^{\alpha_2}_n$ and $S$ is a cube of side length $\lambda\leq 4$ containing $R$, we have $n-4\leq\|z\|_\infty \leq n+1+4\leq 2n$, whence $2\leq m\leq 2n$ because $n\geq 7$. Therefore
\begin{equation*}
    \tfrac12\,l(Q_z^{\alpha_1})=\tfrac12\,2^{-\delta^{\alpha_1}_m}>\tfrac14\,m^{-\alpha_1}\geq\tfrac14\,2^{-\alpha_1}n^{-\alpha_1}\geq 4n^{-\alpha_2}\geq\lambda,
\end{equation*}
where the second inequality from the right holds because $n^{\alpha_2-\alpha_1}\geq 32\geq 16\cdot 2^{\alpha_1}$ by the choice of $n_1$. So \eqref{equ-incl-hyp} holds and \eqref{equ-incl-transfer} gives $[\w]_{A_p(N_{\alpha_2}(R))}\leq[\w]_{A_p^{\alpha_1}}$.

Assume now $n<n_1$. Then $N_{\alpha_2}(R)\subseteq[-(n_1+5),n_1+5)^d\subseteq\zeta[-1,1)^d$ with $\zeta:=n_1+5$, and Lemma \ref{ahto} yields $[\w]_{A_p(N_{\alpha_2}(R))}\leq[\w]_{A_p(\zeta[-1,1)^d)}\leq C([\w]_{A_p^{\alpha_1}})$. Taking the supremum over $R\in\Delta^{\alpha_2}$ gives $\w\in A_p^{\alpha_2}$.

(2) Let $\w\in A_p^{\alpha}$, put $\ell_0:=\left\lceil\frac{5}{1-\alpha}\right\rceil$ and let $R\in\Delta^\gamma_0$ with $\ell:=j(R)$, so that $l(R)=2^{-\ell}$.

Assume first $\ell\geq\ell_0$; note that $\ell_0\geq 5$. The cube $S:=N(R)$ has side length $\lambda=2^{2-\ell}$. Fix $z\in S$ and let $m\geq 0$ be the index of the $\Delta^{\alpha}$-layer containing $z$. Since $R\subseteq L_\ell$ we have $2^{\ell-1}\leq\|u\|_\infty \leq 2^\ell$ for $u\in R$, hence
\begin{equation*}
    2^{\ell-1}-2^{2-\ell}\leq\|z\|_\infty \leq 2^{\ell}+2^{2-\ell}\leq 2^{\ell+1},
\end{equation*}
and therefore $1\leq m\leq 2^{\ell+1}$, the lower bound because $\ell\geq 5$ gives $\|z\|_\infty\geq 2^4-1>2$. Consequently
\begin{equation*}
    \tfrac12\,l(Q_z^\alpha)=\tfrac12\,2^{-\delta^{\alpha}_m}>\tfrac14\,m^{-\alpha}\geq\tfrac14\,2^{-\alpha(\ell+1)}\geq 2^{2-\ell}=\lambda,
\end{equation*}
the last inequality being equivalent to $\ell(1-\alpha)\geq 4+\alpha$, which holds since $\ell\geq\frac{5}{1-\alpha}$. So \eqref{equ-incl-hyp} holds and \eqref{equ-incl-transfer} gives $[\w]_{A_p(N(R))}\leq[\w]_{A_p^{\alpha}}$.

Assume now $\ell<\ell_0$. Then $N(R)\subseteq[-(2^{\ell_0}+4),2^{\ell_0}+4)^d\subseteq\zeta[-1,1)^d$ with $\zeta:=2^{\ell_0}+4$, and Lemma \ref{ahto} yields $[\w]_{A_p(N(R))}\leq C([\w]_{A_p^{\alpha}})$. Taking the supremum over $R\in\Delta^\gamma_0$ gives $\w\in A_p^{\loc}$.
\end{proof}

\begin{cor}\label{DEDEDEff}
 For all $1<p<\infty$ and $\alpha \in [0,1)$, we have
\begin{equation*}
    A_p \subsetneq A_p^\alpha \subsetneq A_p^{\loc}.
\end{equation*}
Moreover, $A^{\alpha_1}_p\subsetneq A^{\alpha_2}_p$ for $0\leq\alpha_1\leq\alpha_2<1$.
\end{cor}

\begin{proof}
The inclusion $A_p\subseteq A_p^\alpha$ is immediate from the definitions, since the $A_p$ condition is required for all cubes of $\R^d$ and in particular for those contained in some $N_\alpha(R)$. The inclusions $A_p^\alpha\subseteq A_p^{\loc}$ and $A^{\alpha_1}_p\subseteq A^{\alpha_2}_p$ are Lemma \ref{lem-inclusions}.

It remains to prove that the two inclusions of the displayed chain and $A^{\alpha_1}_p \subsetneq A^{\alpha_2}_p$ are strict.

Consider first $\w_1(x)=e^{\|x\|_\infty}$. By Lemma \ref{exep}, applied with $\gamma=1\leq\alpha+1$, we have $\w_1\in A_p^\alpha$. On the other hand $\w_1\notin A_p$. Indeed, for $L\geq 1$ let $Q_L:=[0,L)^d$. Then, on the one hand,
\begin{equation*}
    \int_{Q_L}\w_1(x)dx\geq\int_{[\frac L2,L)^d}e^{\|x\|_\infty}dx\geq\Big(\frac{L}{2}\Big)^de^{\frac L2},
\end{equation*}
since $\|x\|_\infty\geq\frac L2$ on $[\frac L2,L)^d$; on the other hand,
\begin{equation*}
    \int_{Q_L}\w_1^{-\frac{p'}{p}}(x)dx\geq\int_{[0,1)^d}e^{-\frac{p'}{p}\|x\|_\infty}dx\geq e^{-\frac{p'}{p}} .
\end{equation*}
Consequently
\begin{equation*}
    \frac{1}{|Q_L|^p}\int_{Q_L}\w_1(x)dx\Big(\int_{Q_L}\w_1^{-\frac{p'}{p}}(x)dx\Big)^\frac{p}{p'}\geq\frac{2^{-d}e^{-1}}{L^{d(p-1)}}\,e^{\frac L2}\xrightarrow[\;L\to\infty\;]{}\infty,
\end{equation*}
so that $\w_1\notin A_p$.

Consider next $\w_2(x)=e^{\|x\|^2_\infty}$. By Proposition \ref{Proposorijfnfjj}, applied with $\gamma=2>\alpha+1$, we have $\w_2\notin A_p^\alpha$, whereas $\w_2\in A_p^{\loc}$ by Remark \ref{rem-sup-norm-loc}.
Finally, $\w_3(x) = e^{\|x\|^{\alpha_2+1}}$ belongs to $A^{\alpha_2}_p$, but not to $A^{\alpha_1}_p$, according to Lemma \ref{exep} and Proposition \ref{Proposorijfnfjj}.
\end{proof}

In the final section, we will study the estimates of the Ornstein--Uhlenbeck maximal operator $\mathcal{H}^*$ on weighted $L^p$ spaces with weights in $A_p^\alpha$.

%In the next two sections, we will study the weighted $L^p$ estimates (with weights in $A_p^\alpha$) of the Ornstein–Uhlenbeck maximal operator restricted to $G_-(x)$ and $G_+(x)$, defined below.

\section{The $A_p^\alpha$ class and the Ornstein-Uhlenbeck maximal operator}
\label{sec-5}

The goal of this section is to determine whether, for $1<p<\infty$ and $\w\in A_p^\alpha$, the Ornstein–Uhlenbeck semigroup maximal operator $\mathcal{H}^*$ is such that
\begin{equation*}
    ||\mathcal{H}^*f||_{L^p(\w(x)e^{-\frac{p|x|^2}{2}}dx)}\leq C([\w]_{A_p^\alpha}) ||f||_{L^p(\w(x)e^{-\frac{p|x|^2}{2}}dx)},
\end{equation*}
with $C([\w]_{A_p^\alpha})$ a constant depending on the characteristic of the weight $\w$. The question is positively answered in Theorem \ref{result2}.\

The proof of the main Theorem \ref{result2} is subdivided into Propositions \ref{susuki} and \ref{suski}. For a given $x \in \R^d$, we will divide the Euclidean space $\R^d$ into two parts:
\begin{equation*}
    \R^d = G_-(x) \cup G_+(x),
\end{equation*}
where
\begin{equation*}
    G_{-}(x):=\{y\in \mathbb{R}^d\;:\;\langle x,y\rangle< 0\}\;\;\;\;\text{and}\;\;\; G_{+}(x):=\{y\in \mathbb{R}^d\;:\;\langle x,y\rangle\geq 0\}.
\end{equation*}
Now, since
\begin{equation*}
  \mathcal{H}^*(f)(x) \leq \mathcal{H}^*(f \chi_{G_-(x)})(x) + \mathcal{H}^*(f \chi_{G_+(x)})(x),
\end{equation*}
it suffices study the maximal operators $\mathcal{H}^*(f \chi_{G_-(x)})$ and $\mathcal{H}^*(f \chi_{G_+(x)})$ separately.\

The reason for this splitting is that the decay of the Mehler kernel defining the Ornstein-Uhlenbeck semigroup operator, which we recall in the form \eqref{eqs},
\begin{equation}\label{fhkfhfkhfslslsls}
    N_s(x,y) = \frac{(1+s)^d}{(4s)^\frac{d}{2}}\exp\Big(\frac{|x|^2+|y|^2}{2}-\frac{s}{4}|x+y|^2-\frac{1}{4s}|x-y|^2\Big)
\end{equation}
is different for $y$ in $G_-(x)$ and $y$ in $G_+(x)$. Namely, for $y$ in $G_-(x)$, 
\begin{equation*}
 |x-y|^2 = |x|^2 + |y|^2 - 2\langle x,y\rangle \geq |x|^2 +
|y|^2    
\end{equation*}
creates decay for $|y| \to \infty$, whereas $|x+y|^2$ can stay close to $0$ (if $y$ is close to $-x$).
This decay is then used to control possible explosion of the ratio $\w(Q)/\w(R)$ of the weight for two distant cubes $Q$ and $R$ belonging to the grid $\Delta^\alpha$ (defined in Definition \ref{DJDDJLZUEYYEYEYEYYE}). The rest of the argument consists more or less simply in writing $f(y) = f(y) w(y)^{\frac1p} w(y)^{-\frac1p}$ and applying Hölder inequality.
On the other hand, for $y$ in $G_+(x)$, 
\begin{equation*}
  |x+y|^2 = |x|^2 + |y|^2 + 2\langle x,y\rangle  \geq |x|^2 + |y|^2  
\end{equation*}
creates decay for $|y| \to \infty$ and $|x-y|^2$ can stay close to $0$ (if $y$ is close to $x$).
Moreover, we will split the $G_+(x)$ part into two separate ones:

\begin{equation*}
\mathcal{H}^*(f\chi_{G_+(x)})(x) \leq \mathcal{H}^*(f\chi_{G_+(x) \cap N_\alpha(Q_x)^C})(x) + \mathcal{H}^*(f\chi_{G_+(x) \cap N_\alpha(Q_x)})(x)
\end{equation*}
The first component has a decay in $y$ which, as for the $G_-(x)$ part, is used to counterbalance the ratio $\w(Q)/\w(R)$ for two distant cubes $Q$ and $R$. The second component is controlled pointwise against the heat maximal operator. Here the variables $x$ and $y$ in the Mehler kernel (see Equation (\ref{fhkfhfkhfslslsls}) above) cannot move far away from each other. This is the very point where we need the weight to belong to the $A^\alpha_p$ class, that is, a precise balancing of $\w(Q)$ and $\w^{-p'/p}(Q)^{p/p'}$, for the other parts, a much cruder upper bound of $\w(Q) \times \w^{-p'/p}(R)^{p/p'}$ for far distant cubes $Q$ and $R$ would suffice.

\subsection{Study of the maximal operator on $G_-(x)$}
\label{subsec-5-1}

The objective of this subsection is to prove the following proposition.

\begin{prop}\label{susuki}   Let $1<p<\infty$, $\alpha\in [0,1)$, and $\w\in A^\alpha_p$. Then there exists a constant $C([\w]_{A_p^\alpha})>0$, depending on  $[\w]_{A_p^\alpha}$ and $d$, such that
    \begin{equation*}
       || \sup_{t>0}\mathcal{H}_t(|f|\chi_{ G_-(x)})||_{L^p(\mathbb{R}^d,e^{-\frac{p}{2}|x|^2}\w(x)dx)}\leq C([\w]_{A_p^\alpha})||f||_{L^p(\mathbb{R}^d,e^{-\frac{p}{2}|x|^2}\w(x)dx)}.
    \end{equation*}
\end{prop}

\begin{rem}\label{susuki1}

Let $D_d=[-\sqrt{2d},\sqrt{2d}]^d$ be a cube centered at $0$. We recall from Lemma \ref{chal} that, for all $f\in L^1_{\loc}$ and $x\in \mathbb{R}^d$,
\begin{equation*}
\sup_{t>0}\mathcal{H}_t(|f|\chi_{G_-(x)})(x)\lesssim e^\frac{|x|^2}{2}\sup_{0<s<1}T_s^\Delta(|f|\chi_{G_-(x)}e^{-\frac{|\;.\;|^2}{2}})(x).
\end{equation*}

From there we see that

\begin{align*}
    \begin{split}
 &||\sup_{t>0}\mathcal{H}_t(|f|\chi_{G_-(x)})||_{L^p(\mathbb{R}^d,\;e^{-p\frac{|x|^2}{2}}\w(x)dx)}\\
 &\lesssim ||e^\frac{|x|^2}{2}\sup_{0<s<1}T_s^\Delta(|f|\chi_{G_-(x)}e^{-\frac{|\;.\;|^2}{2}})||_{L^p(\mathbb{R}^d,\;e^{-p\frac{|x|^2}{2}}\w(x)dx)}\\
   &\leq\Big|\Big|e^\frac{|x|^2}{2}\Big(\sup_{0<s<1}T_s^\Delta(|f|\chi_{G_-(x)\cap\mathbb{R}^d\setminus{D_d}}e^{-\frac{|\;.\;|^2}{2}})\\
   &\quad+\sup_{0<s<1}T_s^\Delta(|f|\chi_{G_-(x)\cap D_d}e^{-\frac{|\;.\;|^2}{2}})\Big)\Big|\Big|_{L^p(\mathbb{R}^d,\;e^{-p\frac{|x|^2}{2}}\w(x)dx)}\\
   &\leq ||e^\frac{|x|^2}{2}\sup_{0<s<1}T_s^\Delta(|f|\chi_{G_-(x)\cap\mathbb{R}^d\setminus{D_d}}e^{-\frac{|\;.\;|^2}{2}})||_{L^p(\mathbb{R}^d,\;e^{-p\frac{|x|^2}{2}}\w(x)dx)}\\
   &\quad + ||e^\frac{|x|^2}{2}\sup_{0<s<1}T_s^\Delta(|f|\chi_{G_-(x)\cap D_d}e^{-\frac{|\;.\;|^2}{2}})||_{L^p(\mathbb{R}^d,\;e^{-p\frac{|x|^2}{2}}\w(x)dx)},
    \end{split}
\end{align*}
 by Minkowski's inequality.
    
\end{rem}
By this remark, to prove Proposition \ref{susuki} it suffices to show the two propositions below:

\begin{prop}\label{shibuya2}  Let $1<p<\infty$, $\alpha\in [0,1)$, and $\w\in A^\alpha_p$. Then there exists a constant $C([\w]_{A_p^\alpha})>0$, depending on  $[\w]_{A_p^\alpha}$ and $d$, such that
    \begin{equation*}
     ||e^\frac{|x|^2}{2}\sup_{0<s<1}T_s^\Delta(|f|\chi_{G_-(x)\cap\mathbb{R}^d \setminus{D_d}}e^{-\frac{|\;.\;|^2}{2}})||_{L^p(\mathbb{R}^d,\;e^{-p \frac{|x|^2}{2}}\w(x)dx)} \leq C([\w]_{A_p^\alpha})||f||_{L^p(\mathbb{R}^d,\w(x)e^{-\frac{p}{2}|x|^2}dx)}.
    \end{equation*}
 
\end{prop}

\begin{prop}\label{shibuya1}   Let $1<p<\infty$, $\alpha\in [0,1)$, and $\w\in A^\alpha_p$. Then there exists a constant $C([\w]_{A_p^\alpha})>0$, depending on  $[\w]_{A_p^\alpha}$ and $d$, such that
    \begin{equation*}
       ||e^\frac{|x|^2}{2}\sup_{0<s<1}T_s^\Delta(|f|\chi_{G_-(x)\cap D_d}e^{-\frac{|\;.\;|^2}{2}})||_{L^p(\mathbb{R}^d,\;e^{-p\frac{|x|^2}{ 2}}\w(x)dx)}\leq C([\w]_{A_p^\alpha})||f||_{L^p(\mathbb{R}^d,\w(x)e^{-\frac{p}{2}|x|^2} dx)},
    \end{equation*}
\end{prop}
\begin{proof}\,\textbf{(Proposition \ref{shibuya2})}     For all $0<s<1$, let $s\mapsto g_s(x,y)=s^{-\frac{d}{2}}\exp\Big(-\frac{1}{4s}|x-y |^2\Big) $, with $x$ in $\mathbb{R}^d$ and $y\in G_-(x)\cap\mathbb{R}^d\setminus{D_d}=G_- (x)\cap D_d^C$ (with $D_d^C$ to designate the complementary part of $D_d$). We have that

\begin{align*}
   \partial_sg(x,y)&=-\frac{d}{2}\Big(\frac{1}{s}\Big)^{\frac{d}{2}+1}e^{-\frac{1}{4s}|x-y|^2}+\Big(\frac{1}{s}\Big)^{\frac{d}{2}+1}\Big(\frac{|x-y|^2}{4s}\Big)e^{-\frac{1}{4s}|x-y|^2}\\
    &=\Big(-\frac{d}{2}+\frac{|x-y|^2}{4s}\Big) \Big(\frac{1}{s}\Big)^{\frac{d}{2}+1}e^{-\frac{1}{4s}|x-y|^2},
\end{align*}
for all fixed $x$ and $y$. Thus, the function $s\mapsto g_s(x,y)$ reaches its maximum at the point $s=\frac{|x-y|^2}{2d}$. Since $y\in G_-(x)\cap D_d^C$, we observe that
\begin{align}
    |x-y|^2&=|x|^2+|y|^2+2|\langle x, y \rangle|\\\label{eiir}
    &\geq |x|^2+|y|^2\\
    &>2d,
\end{align}
 then, $\frac{|x-y|^2}{2d}>1$. Therefore, in this case
\begin{equation*}
    \sup_{0<s<1}g_s(x,y)=g_1(x,y)=\exp\Big(-\frac{1}{4}|x-y|^2\Big).
\end{equation*}
 This observation implies that

\begin{align*}
\begin{split}
&||e^\frac{|x|^2}{2}\sup_{0<s<1}T_s^\Delta(|f|\chi_{G_-(x)\cap D_d^C}e^{-\frac{|\;.\;|^2}{2}})||^p_{L^p(\mathbb{R}^d,\;e^{-p\frac{|x|^2}{2}}\w(x)dx)}\\
&\cong\int_{\mathbb{R}^d}\Big(\sup_{0<s<1}\int_{G_-(x)\cap D_d^C}g_s(x,y)|f(y)|e^{-\frac{|y|^2}{2}}dy\Big)^p\w(x)dx\\
    &\lesssim\int_{\mathbb{R}^d}\Big(\int_{G_-(x)\cap D_d^C}\sup_{0<s<1}g_s(x,y)|f(y)|e^{-\frac{|y|^2}{2}}dy\Big)^p\w(x)dx \\
    &=\int_{\mathbb{R}^d}\Big(\int_{G_-(x)\cap D_d^C}\exp\Big(-\frac{1}{4}|x-y|^2\Big)|f(y)|e^{-\frac{|y|^2}{2}}dy\Big)^p\w(x)dx,
\end{split}
\end{align*}
       
Then, by the fact that $y\in G_-(x)$, we have by equation (\ref{eiir}), 
$-|x-y|^2\leq -|x|^2-|y|^2$. Moreover $y\in D^C_d$ means $\|y\|_\infty>\sqrt{2d}$, so that $|y|>\sqrt{2d}\geq\sqrt{2}>1$ and hence $y$ is outside the cube $[-1,1)^d$. Then, we deduce by the previous equations that  
\begin{align}
   &||e^\frac{|x|^2}{2}\sup_{0<s<1}T_s^\Delta(|f|\chi_{G_-(x)\cap D_d^C}e^{-\frac{|\;.\;|^2}{2}})||^p_{L^p(\mathbb{R}^d,\;e^{-p\frac{|x|^2}{2}}\w(x)dx)}\\&\leq\int_{\mathbb{R}^d}\Big(\int_{G_-(x)\cap D_d^C}\exp\Big(-\frac{1}{4}(|x|^2+|y|^2)\Big)|f(y)|e^{-\frac{|y|^2}{2}}dy\Big)^p\w(x)dx\\
&\leq\int_{\mathbb{R}^d}\Big(\int_{\mathbb{R}^d\setminus [-1,1)^d}\exp\Big(-\frac{1}{4}(|x|^2+|y|^2)\Big)|f(y)|e^{-\frac{|y|^2}{2}}dy\Big)^p\w(x)dx\label{wA}.
\end{align}

Then by decomposing $\mathbb{R}^d$ into cubes of the grid $\Delta^\alpha$, equation (\ref{wA}) becomes

\begin{align*}
    \begin{split}
       &||e^\frac{|x|^2}{2}\sup_{0<s<1}T_s^\Delta(|f|\chi_{G_-(x)\cap D_d^C}e^{-\frac{|\;.\;|^2}{2}})||^p_{L^p(\mathbb{R}^d,\;e^{-p\frac{|x|^2}{2}}\w(x)dx)}\\
       &\leq \int_{\mathbb{R}^d}\Big(\int_{\mathbb{R}^d\setminus [-1,1)^d}e^{-\frac{1}{4}(|x|^2+|y|^2)}e^{-\frac{|y|^2}{2}}|f(y)|dy\Big)^p\w(x)dx\\
        &=\sum_{Q\in \Delta^\alpha}\int_{Q}\Big(\sum_{R\in \Delta^\alpha\setminus \Delta_0^\alpha}\int_{R}e^{-\frac{1}{4}(|x|^2+|y|^2)}|f(y)|e^{-\frac{|y|^2}{2}}dy\Big)^p\w(x)dx\\
        &=\sum_{Q\in \Delta^\alpha}\int_{Q}e^{-p\frac{1}{4}|x|^2}\Big(\sum_{R\in \Delta^\alpha\setminus \Delta_0^\alpha}\int_{R}e^{-\frac{1}{4}|y|^2}|f(y)|e^{-\frac{|y|^2}{2}}dy\Big)^p\w(x)dx\\
        &\leq\sum_{n=0}^\infty e^{-\frac{p}{4}n^2}\sum_{Q\in \Delta_n^\alpha}\int_{Q}\Big(\sum_{l=1}^\infty e^{-\frac{1}{4}l^2}\sum_{R\in \Delta_l^\alpha}\int_{R}|f(y)|e^{-\frac{|y|^2}{2}}dy\Big)^p\w(x)dx\\
        &\leq\sum_{n=0}^\infty e^{-\frac{p}{4}n^2}\sum_{Q\in \Delta_n^\alpha}\int_{Q}\Big(\sum_{l=1}^\infty \sum_{R\in \Delta_l^\alpha}e^{-\frac{p'}{8}l^2}\Big)^\frac{p}{p'}\sum_{l=1}^\infty \sum_{R\in \Delta_l^\alpha}e^{-\frac{p}{8}l^2}\Big(\int_{R}|f(y)|e^{-\frac{|y|^2}{2}}dy\Big)^p\w(x)dx,
    \end{split}
\end{align*}
where we used the fact that  $|x|\geq n$ and $|y|\geq l$ when respectively $x\in Q\in \Delta^\alpha_n$ and $y\in R\in \Delta^\alpha_l$ and Hölder's inequality applied under the sum in $R$, then in $l$.  Moreover, by Lemma \ref{cardi}, since $\#\Delta_l^\alpha\lesssim 
    l^{d+\alpha d}$, then

\begin{equation*}
    \Big(\sum_{l=1}^\infty \sum_{R\in \Delta_l^\alpha}e^{-\frac{p'}{8}l^2}\Big)^\frac{p}{p'}\lesssim \Big(\sum_{l=1}^\infty e^{-\frac{p'}{8}l^2}l^{d+\alpha d}\Big)^\frac{p}{p'}= \Big(\sum_{l=1}^\infty e^{-\frac{p'}{8}l^2+(d+\alpha d)\ln l}\Big)^\frac{p}{p'}<\infty.
\end{equation*}

From what precedes, we obtain
\begin{align}
   &||e^\frac{|x|^2}{2}\sup_{0<s<1}T_s^\Delta(|f|\chi_{G_-(x)\cap D_d^C}e^{-\frac{|\;.\;|^2}{2}})||^p_{L^p(\mathbb{R}^d,\;e^{-p\frac{|x|^2}{2}}\w(x)dx)}\\
        &\lesssim\sum_{n=0}^\infty e^{-\frac{p}{4}n^2}\sum_{Q\in \Delta_n^\alpha}\int_{Q}\sum_{l=1}^\infty e^{-\frac{p}{8}l^2} \sum_{R\in \Delta_l^\alpha}\Big(\int_{R}|f(y)|e^{-\frac{|y|^2}{2}}\w^{-\frac{1}{p}}(y)\w^\frac{1}{p}(y)dy\Big)^p \w(x)dx\\
        &\leq \sum_{n=0}^\infty e^{-\frac{p}{4}n^2}\sum_{Q\in \Delta_n^\alpha}\int_{Q}\sum_{l=1}^\infty e^{-\frac{p}{8}l^2}\sum_{R\in \Delta^\alpha_l}   \int_{R} |f(y)|^p\w(y)e^{-\frac{p}{2}|y|^2}dy \Big(\int_{R}\w^{-\frac{p'}{p}}(y)\Big)^\frac{p}{p'}\w(x)dx\\
     &\leq \sum_{n=0}^\infty e^{-\frac{p}{4}n^2}\sum_{Q\in \Delta_n^\alpha}\w(Q)\sum_{l=1}^\infty e^{-\frac{p}{8}l^2}\sum_{R\in \Delta^\alpha_l} \w^{-\frac{p'}{p}}(R)^\frac{p}{p'} \int_{R} |f(y)|^p\w(y)e^{-\frac{p}{2}|y|^2}dy\\
     &=:\mathcal{M}\label{M}
\end{align}
 where we applied Hölder's inequality under the integral in $R$.\

However, for $R\in \Delta^\alpha_l$ and $Q\in \Delta^\alpha_n$, by Remarks \ref{HDHDKDKDKLDJJDLK} and \ref{emm} and Lemma \ref{oh}, we have that
 \begin{align*}
\w^{-\frac{p'}{p}}(R)^\frac{p}{p'} &\leq ([\w]^p_{A_p^\alpha}4^{pd})^{r(R,Q)}\w^{-\frac{p'}{p}}(Q)^\frac{p}{p'}\\ 
  &\leq ([\w]_{A_p^\alpha}4^{d})^{\frac{2\rho(Q,R)}{\min(l(Q),l(R))}+8}\w^{-\frac{p'}{p}}(Q)^\frac{p}{p'}\\
&\lesssim([\w]_{A_p^\alpha}4^{d})^{\frac{2\rho(Q,R)}{\min(l(Q),l(R))}}\w^{-\frac{p'}{p}}(Q)^\frac{p}{p'}\\
&\lesssim ([\w]^p_{A_p^\alpha}4^{pd})^{16\max(n,l)^{\alpha+1}}\w^{-\frac{p'}{p}}(Q)^\frac{p}{p'},
 \end{align*}
because  $\rho(Q,R)\leq 2\max(n,l)+2$, 
\begin{equation*}
\min(l(Q),l(R))^{-1}=\max(2^{\delta_n^\alpha},2^{\delta_l^\alpha})<2\max(1,n)^\alpha\,\max(1,l)^\alpha\leq 2\max(n,l)^\alpha ,
\end{equation*}
and $\max(n,l)\geq 1$ given that $l\geq 1$. Here we have written $\max(1,n)$ rather than $n$ because the summation over $n$ starts at $n=0$; for $n=0$ one has $2^{\delta^\alpha_0}=\frac12<2^{\delta^\alpha_l}$, so that the displayed inequality holds as well.

Thus, by equation $(\ref{M})$ and the above, we arrive at

\begin{align*}
\begin{split}
     \mathcal{M}&\lesssim \sum_{n=0}^\infty e^{-\frac{p}{4}n^2}\sum_{Q\in \Delta_n^\alpha}\w(Q)\sum_{l=1}^\infty e^{-\frac{p}{8}l^2}\sum_{R\in \Delta^\alpha_l}([\w]^p_{A_p^\alpha}4^{pd})^{16\max(n,l)^{\alpha+1}}\w^{-\frac{p'}{p}}(Q)^\frac{p}{p'} \int_{R} |f(y)|^p\w(y)e^{-\frac{p}{2}|y|^2}dy\\
     &\leq \sum_{n=0}^\infty e^{-\frac{p}{4}n^2}\sum_{Q\in \Delta_n^\alpha}\sum_{l=1}^\infty e^{-\frac{p}{8}l^2}\sum_{R\in \Delta^\alpha_l}([\w]^p_{A_p^\alpha}4^{pd})^{16n^{\alpha+1}+16l^{\alpha+1}}\w(Q)\w^{-\frac{p'}{p}}(Q)^\frac{p}{p'} \int_{R} |f(y)|^p\w(y)e^{-\frac{p}{2}|y|^2}dy\\
&\lesssim\sum_{n=0}^\infty e^{-\frac{p}{4}n^2}([\w]_{A_p^\alpha}4^{d})^{p16n^{\alpha+1}}\sum_{Q\in \Delta_n^\alpha}\sum_{l=1}^\infty[\w]^p_{A_p^\alpha}|Q|^p e^{-\frac{p}{8}l^2}([\w]_{A_p^\alpha}4^{d})^{p16l^{\alpha+1}}\sum_{R\in \Delta^\alpha_l}\int_{R} |f(y)|^p\w(y)e^{-\frac{p}{2}|y|^2}dy\\
&\lesssim\sum_{n=0}^\infty e^{-\frac{p}{4}n^2}([\w]_{A_p^\alpha}4^{d})^{p16n^{\alpha+1}}(n^{\alpha d+d}+1)\sum_{l=1}^\infty [\w]^p_{A_p^\alpha}e^{-\frac{p}{8}l^2}([\w]_{A_p^\alpha}4^{d})^{p16l^{\alpha+1}}\sum_{R\in \Delta^\alpha_l}\int_{R} |f(y)|^p\w(y)e^{-\frac{p}{2}|y|^2}dy\\
&\lesssim[\w]^p_{A_p^\alpha}\sup_{l\geq 1}\Big(e^{-\frac{p}{8}l^2}([\w]_{A_p^\alpha}4^{d})^{p16l^{\alpha+1}}\Big)\sum_{l=0}^\infty \sum_{R\in \Delta^\alpha_l}\int_{R} |f(y)|^p\w(y)e^{-\frac{p}{2}|y|^2}dy\\
&\leq C([\w]_{A_p^\alpha})||f||^p_{L^p(\mathbb{R}^d,\w(x)e^{-\frac{p}{2}|x|^2}dx)},
\end{split} 
\end{align*}
 where we used $\w(Q)\w^{-\frac{p'}{p}}(Q)^\frac{p}{p'}\leq[\w]^p_{A_p^\alpha}|Q|^p$, which is legitimate since $Q\subseteq N_\alpha(Q)$, and then $\sum_{Q\in\Delta^\alpha_n}|Q|^p\leq 2^{dp}\,\#\Delta^\alpha_n\lesssim n^{\alpha d+d}+1$ by Lemma \ref{cardi}, the first inequality because $|Q|=2^{-d\delta^\alpha_n}\leq 2^d$. The conclusion follows, given that for $\alpha\in [0,1)$,
\begin{equation*}
    \sum_{n=0}^\infty e^{-\frac{p}{4}n^2}([\w]_{A_p^\alpha}4^{d})^{p16n^{\alpha+1}}(n^{\alpha d+d}+1)<\infty,
\end{equation*}
and 
\begin{equation*}
   \sup_{l\geq 1}\Big(e^{-\frac{p}{8}l^2}([\w]_{A_p^\alpha}4^{d})^{p16l^{\alpha+1}}\Big)<\infty. 
\end{equation*}
\end{proof}

\begin{proof}\,\textbf{(Proposition \ref{shibuya1})}
%Nous commençons par remarquer que  
Let us begin by observing that
\begin{align*}
\begin{split}
     &||e^\frac{|x|^2}{2}\sup_{0<s<1}T_s^\Delta(|f|\chi_{G_-(x)\cap D_d}e^{-\frac{|\;.\;|^2}{2}})||^p_{L^p(\mathbb{R}^d,\;e^{-p\frac{|x|^2}{2}}\w(x)dx)}\\
     &\leq||e^\frac{|x|^2}{2}\sup_{0<s<1}T_s^\Delta(|f|\chi_{G_-(x)\cap D_d}e^{-\frac{|\;.\;|^2}{2}})||^p_{L^p(D_d^C,\;e^{-p\frac{|x|^2}{2}}\w(x)dx)}\\
     &\quad +  ||e^\frac{|x|^2}{2}\sup_{0<s<1}T_s^\Delta(|f|\chi_{G_-(x)\cap D_d}e^{-\frac{|\;.\;|^2}{2}})||^p_{L^p(D_d,\;e^{-p\frac{|x|^2}{2}}\w(x)dx)}\\
     &\leq||e^\frac{|x|^2}{2}\sup_{0<s<1}T_s^\Delta(|f|\chi_{G_-(x)}e^{-\frac{|\;.\;|^2}{2}})||^p_{L^p(D_d^C,\;e^{-p\frac{|x|^2}{2}}\w(x)dx)}\\
     &\quad +  ||e^\frac{|x|^2}{2}\sup_{0<s<1}T_s^\Delta(|f|\chi_{D_d}e^{-\frac{|\;.\;|^2}{2}})||^p_{L^p(D_d,\;e^{-p\frac{|x|^2}{2}}\w(x)dx)}.
\end{split}  
\end{align*}

However, the control of

\begin{equation*}
    ||e^\frac{|x|^2}{2}\sup_{0<s<1}T_s^\Delta(|f|\chi_{G_-(x)}e^ {- \frac{|\;.\;|^2}{2}})||^p_{L^p(D_d^C,\;e^{-p\frac{|x|^2}{ 2} }\w(x)dx)} \;\;\;\text{by}\;\;\;\; ||f||^p_{L^p(\mathbb{R}^d,\w(x)e^{-\frac{p}{2}|x|^2 }dx)}
\end{equation*}
is treated in the same way as $ ||e^\frac{|x|^2}{2}\sup_{0<s<1}T_s^\Delta(|f|\chi_{G_- (x )\cap D_d^C}e^{-\frac{|\;.\;|^2}{2}})||_{L^p(\mathbb{R}^d,\;e ^{ -p \frac{|x|^2}{2}}\w(x)dx)}$. Indeed, in this case $x\in D^C_d$ and $y\in G_-(x)$, so that $|x|^2+|y|^2\geq|x|^2>2d$; this was the key point for the proof of Proposition \ref{shibuya2}. Two differences should be noted. First, $y$ is no longer confined to $D_d^C$ and may therefore lie in $[-1,1)^d$, so that the summation over $R\in\Delta^\alpha_l$ now runs over all $l\geq 0$; this affects neither the application of Hölder's inequality nor the convergence of $\sum_{l\geq 0}e^{-\frac{p'}{8}l^2}\#\Delta^\alpha_l$. Second, the estimate $\frac{2\rho(Q,R)}{\min(l(Q),l(R))}\leq 16\max(n,l)^{\alpha+1}$ requires $\max(n,l)\geq 1$, which in Proposition \ref{shibuya2} followed from $l\geq 1$ and which here follows from $x\in D_d^C$, whence $|x|>\sqrt{2d}>1$ and therefore $n\geq 1$. 
We will therefore only study the control of
\begin{equation*}
    ||e^\frac{|x|^2}{2}\sup_{0<s<1}T_s^\Delta(|f|\chi_{D_d} e^ {- \frac{|\;.\;|^2}{2}})||^p_{L^p(D_d,\;e^{-p\frac{|x|^2}{ 2} }\w(x) dx) }\;\;\;\;\text{by}\;\;\;\; ||f||^p_{L^p(\mathbb{R}^d,e^{-\frac{p}{2}|x|^2 }\w(x)dx) }
\end{equation*}
Since $D_d\subseteq\zeta[-1,1)^d$ for $\zeta:=\sqrt{2d}+1$, Lemma~\ref{ahto} gives $\omega \in A_p(D_d)$ together with $[\w]_{ A_p(D_d)}\leq C([\w]_{A^\alpha_p})$. Now, by Lemma~\ref{extension}, there exists a weight $v$ on $\mathbb{R}^d$ such that $v|_{D_d} = \omega$ on $D_d$  and $[v]_{A_p} \leq \max(2^{dp},(4\zeta + 1)^{dp}) [\omega]_{A_p(D_d)}$. Hence, we observe
    \begin{align}
     ||e^\frac{|x|^2}{2}\sup_{0<s<1}T_s^\Delta(|f|\chi_{D_d}e^{-\frac{|\;.\;|^2}{2}})||^p_{L^p(D_d,\;e^{-p\frac{|x|^2}{2}}\w(x)dx)}
     & =||\sup_{0<s<1}T_s^\Delta(|f|\chi_{D_d}e^{-\frac{|\;.\;|^2}{2}})||^p_{L^p(D_d,v(x)dx)}\label{++-+}\\
     &\leq ||\sup_{0<s<1}T_s^\Delta(|f|\chi_{D_d}e^{-\frac{|\;.\;|^2}{2}})||^p_{L^p(\mathbb{R}^d,v(x)dx)}\label{--+-}\\
        &\leq C\left([v]_{A_p}\right)||f||^p_{L^p(D_d,e^{-\frac{p|x|^2}{2}}v(x)dx)}\\
        &\leq C'\left([\w]_{A_p(D_d)}\right)||f||^p_{L^p(D_d,e^{-\frac{p|x|^2}{2}}\w(x)dx)}\\
        %&\lesssim[w]^p_{A_p^\alpha}||f||^p_{L^p(D_d,e^{-\frac{p|x|^2}{2}}w(x))}\\
    &\leq C''\left([\w]_{A_p^\alpha}\right)||f||^p_{L^p(\mathbb{R}^d,e^{-\frac{p|x|^2}{2}}\w(x)dx)},
    \end{align}
    where the transition from equation (\ref{++-+}) to (\ref{--+-}) is the enlargement of the domain of integration from $D_d$ to $\R^d$, and where the following step is justified by Theorem \ref{Stein} together with $\sup_{0<s<1}T^\Delta_s\leq T^*$. 
   
\end{proof}

By Remark~\ref{susuki1} and Propositions~\ref{shibuya2} and \ref{shibuya1}, Proposition~\ref{susuki} is proved. Thus, we have achieved the objective of this subsection. We can now proceed to the study of $G_+(x)$.

\subsection{Study of the maximal operator on $G_+(x)$}
\label{subsec-5-2}

In this subsection, the objective will be to show the following proposition.

\begin{prop}\label{suski} Let $1<p<\infty$, $\alpha\in [0,1)$, and $\w\in A^\alpha_p$. Then there exists a constant $C([\w]_{A_p^\alpha})>0$, depending on  $[\w]_{A_p^\alpha}$ and  $d$, such that
    \begin{equation*}
       || \sup_{t>0}\mathcal{H}_t(|f|\chi_{ G_+(x)})||_{L^p(\mathbb{R}^d,\w(x)e^{ -\frac{p}{2}|x|^2}dx)}\leq C([\w]_{A_p^\alpha})||f||_{L^p(\mathbb{R}^d,\w(x)e^{-\frac{p}{2}|x|^2}dx)}
    \end{equation*}
\end{prop}

\begin{rem}

Since, for every fixed $x \in \mathbb{R}^d$, we have
\begin{equation*}
    \mathcal{H}^*f(x)=\pi^{-\frac{d}{2}}\sup_{0<s<1} \frac{(1+s)^d}{(4s)^\frac{d}{2}}\int_{\mathbb{R}^d}e^{\frac{|x|^2+|y|^2}{2}-\frac{s}{4}|x+y|^2-\frac{1}{4s}|x-y|^2}|f(y)|d\gamma(y),
\end{equation*}
according to equation (\ref{eqs}), we observe on $G_+(x)$ that, for all $x \in \mathbb{R}^d$ and $f \in L^1_{\loc}$,
\begin{align*}
     \sup_{t>0}\mathcal{H}_t(|f|\chi_{ G_+(x)})(x)&\lesssim e^{\frac{|x|^2}{2}}\sup_{0<s<1} \frac{1}{(4s)^\frac{d}{2}}\int_{G_+(x)}\exp\Big(-\frac{s}{4}|x+y|^2-\frac{1}{4s}|x-y|^2\Big)|f(y)|e^{-\frac{|y|^2}{2}}dy\\
     &\lesssim e^{\frac{|x|^2}{2}}\sup_{0<s<1} \frac{1}{s^\frac{d}{2}}\int_{G_+(x)}\exp\Big(-\frac{s}{4}(|x|^2+|y|^2)-\frac{1}{4s}|x-y|^2\Big)|f(y)|e^{-\frac{|y|^2}{2}}dy,
\end{align*}
since in this case $|x+y|^2=|x|^2+|y|^2+2|\langle x,y\rangle|\geq |x|^2+|y|^2$. Now we set $F_s(x,y)$ the function defined below
\begin{equation}\label{F}
    F_s(x,y):=\frac{1}{s^\frac{d}{2}}\exp\Big(-\frac{s}{4}(|x|^2+|y|^2)-\frac{1}{4s}|x-y|^2\Big).
\end{equation}
We obtain that  
\begin{equation}\label{968}
    \sup_{t>0}\mathcal{H}_t(|f|\chi_{ G_+(x)})(x)\lesssim  e^{\frac{|x|^2}{2}}\sup_{0<s<1}\int_{\mathbb{R}^d}|f(y)|F_s(x,y)e^{-\frac{|y|^2}{2}}dy.
\end{equation}
\end{rem}

\begin{lem}\label{G+}
    Let $Q\in\Delta_n^\alpha$ and $R\in\Delta_l^\alpha$, and let $x\in Q$ and $y\in R$. Then, if $R\cap N_\alpha(Q)= \emptyset$, the function $s\mapsto F_s(x,y)$ admits a maximum at a positive point, which we will note $s_{\max}(x,y)= : s_ {\max}$, verifying that
        \begin{equation}\label{lab0}
            \frac{k'|x-y|}{\sqrt{|x|^2+|y|^2}}\leq s_{\max}\leq \frac{k|x-y|}{\sqrt{|x|^2+|y|^2}},
        \end{equation}
       with $k$ and $k'$ two strictly positive constants independent of $x$, $y$, $R$ and $Q$.   In this case, we deduce that
        \begin{equation}\label{lab00}
            F_{s_{\max}}(x,y)\lesssim (|x|^2+|y|^2)^\frac{d}{4}|x-y|^{-\frac{d}{2}}\exp\Big(-k''|x-y|\sqrt{|x|^2+|y|^2}\Big)
        \end{equation}
        with $k''=\frac{k'}{4}+\frac{1}{4k}$.

\end{lem}
\begin{proof}\,
    Calculating the derivative of the function $s\mapsto F_s(x,y)$ gives
        \begin{align*}
            \partial_sF_s(x,y)&=\partial_s\Big(\frac{1}{s^\frac{d}{2}}e^{-\frac{s}{4}(|x|^2+|y|^2)}e^{-\frac{1}{4s}|x-y|^2}\Big)\\
            &=\Big(-\frac{d}{2}s^{-\frac{d}{2}-1}+s^{-\frac{d}{2}}\Big(-\frac{1}{4}(|x|^2+|y|^2)\Big)+s^{-\frac{d}{2}}\frac{1}{4s^2}|x-y|^2\Big)e^{-\frac{s}{4}(|x|^2+|y|^2)}e^{-\frac{1}{4s}|x-y|^2}\\
            &=\Big(-\frac{d}{2}s^{-1}+\Big(-\frac{1}{4}(|x|^2+|y|^2)\Big)+\frac{1}{4s^2}|x-y|^2\Big)s^{-\frac{d}{2}}e^{-\frac{s}{4}(|x|^2+|y|^2)}e^{-\frac{1}{4s}|x-y|^2}\\
            &=\Big(-\frac{1}{4}(|x|^2+|y|^2)s^2-\frac{d}{2}s+\frac{1}{4}|x-y|^2\Big)s^{-\frac{d}{2}-2}e^{-\frac{s}{4}(|x|^2+|y|^2)}e^{-\frac{1}{4s}|x-y|^2}\\
            &=-\frac{|x|^2+|y|^2}{4}(s-S_+)(s-S_-)s^{-\frac{d}{2}-2}e^{-\frac{s}{4}(|x|^2+|y|^2)}e^{-\frac{1}{4s}|x-y|^2},
        \end{align*}
      
        with 
  \begin{equation*}
      S_+:=\frac{-\frac{d}{2}+\sqrt{\frac{d^2}{4}+\frac{1}{4}(|x|^2+|y|^2)|x-y|^2}}{\frac{1}{2}(|x|^2+|y|^2)}\;\;\;\;\text{and}\;\;\;\;\;S_-:=\frac{-\frac{d}{2}-\sqrt{\frac{d^2}{4}+\frac{1}{4}(|x|^2+|y|^2)|x-y|^2}}{\frac{1}{2}(|x|^2+|y|^2)}.
  \end{equation*}

   $S_+$ and $S_-$ are well defined since $R\cap N_\alpha(Q)= \emptyset$. Indeed, in this case $x = 0$ forces $|y|\geq 1$, and vice versa: if $x=0$, then $Q=[-1,1)^d$ and $|0-y|^2\geq \rho(Q , R)^2\geq l(Q)^2=4\geq 1$; the reasoning is the same if $y=0$. Furthermore, by a table of variations, we can notice that $S_+$ is the maximum point $s_{\max}$ of the function $s\mapsto F_s(x,y)$. Now, for $x\in Q$ and $y\in R$, denote
\begin{equation*}
    P(x,y):=(|x|^2+|y|^2)|x-y|^2.
\end{equation*}
 We observe that $P(x,y)\geq \frac{1}{4}$.  Indeed, on the one hand for $x\in Q\in \Delta^\alpha_0$,
\begin{align*}
      P(x,y)&\geq (|x|^2+|y|^2)l(Q)^2\\
      &\geq |x|^2+1\\
      &\geq \frac{1}{4}.
\end{align*}
On the other hand, for $x\in Q\notin \Delta^\alpha_0$, since $\alpha\in [0,1)$, we have that
\begin{align*}
    P(x,y)&\geq (|x|^2+|y|^2)2^{-2\delta^\alpha_n}\\
    &> (|x|^2+|y|^2)2^{-2}n^{-2\alpha}\\
    &\geq (|x|^2+|y|^2)2^{-2}|x|^{-2\alpha}\\
    &=2^{-2}(|x|^{2(1-\alpha)}+|y|^2|x|^{-2\alpha})\\
    &\geq 2^{-2}(1+|y|^2|x|^{-2\alpha})\\
    &\geq \frac{1}{4}.
\end{align*}
We can therefore write that
\begin{equation}\label{d^2}
    \frac{d^2}{4}\leq c\frac{1}{4}P(x,y),
\end{equation}
with $c=4d^2$. However, by equation (\ref{d^2}), we then have

\begin{align*}
       s_{\max}&\leq \frac{\sqrt{\frac{d^2}{4}+\frac{1}{4}P(x,y)}}{\frac{1}{2}(|x|^2+|y|^2)}\\
       &\leq\frac{\sqrt{(c+1)\times\frac{1}{4}(|x|^2+|y|^2)|x-y|^2}}{\frac{1}{2}(|x|^2+|y|^2)}\\
       &=k\frac{|x-y|}{\sqrt{|x|^2+|y|^2}}.
   \end{align*}
with $k=\sqrt{c+1}$ a positive real constant. Now let us look for a lower bound of $s_{\max}$. We know that

\begin{align}\label{opmlm}
     s_{\max}&=\frac{-\frac{d}{2}+\sqrt{\frac{d^2}{4}+\frac{1}{4}P(x,y)}}{\frac{1}{2}(|x|^2+|y|^2)}\\\label{1.52}
       &=\frac{d}{|x|^2+|y|^2}\Big(-1+\sqrt{1+P_d}\Big),
\end{align}
with $P_d:=\frac{P(x,y)}{d^2}$. However $\sqrt{1+P_d}\geq 1+k'\sqrt{P_d}$, with $k'=\frac{\sqrt{1+u'}-1}{\sqrt{u'}}$, where $u'=\frac{1}{4d^2}$. Indeed, as $P_d\geq \frac{1}{4d^2}=:u'$ and $P_d=\sqrt{P_d}\sqrt{P_d}$, consequently
  $P_d\geq \sqrt{u'}\sqrt{P_d}$ and then
\begin{align*}
    (1+k'\sqrt{P_d})^2&=1+2k'\sqrt{P_d}+k'^2P_d\\
    &\leq 1+2k'\frac{P_d}{\sqrt{u'}}+k'^2P_d\\
    &=1+(\frac{2k'}{\sqrt{u'}}+k'^2)P_d\\
    &\leq 1+P_d,
\end{align*}
if $\frac{2k'}{\sqrt{u'}}+k'^2\leq 1$, which is the case for $k'=\frac{\sqrt{1+u'}-1}{\sqrt {u'}}$. Thus, from equations (\ref{opmlm}) and $(\ref{1.52})$, we have that
\begin{align*}
     s_{\max}&=\frac{d}{|x|^2+|y|^2}\Big(-1+\sqrt{1+P_d}\Big)\\
     &\geq \frac{d}{|x|^2+|y|^2}\Big(-1+(1+k'\sqrt{P_d})\Big)\\
     &= \frac{k'|x-y|}{\sqrt{|x|^2+|y|^2}}.
\end{align*}
 By the previous estimates on $s_{\max}$, we deduce that
\begin{align*}
    F_{s_{\max}}(x,y)&=\frac{1}{(s_{\max})^\frac{d}{2}}\exp\Big(-\frac{s_{\max}}{4}(|x|^2+|y|^2)-\frac{1}{4s_{\max}}|x-y|^2\Big)\\
    &\leq \Big(\frac{1}{k'}\Big)^\frac{d}{2}\Big(\frac{\sqrt{|x|^2+|y|^2}}{|x-y|}\Big)^\frac{d}{2}\exp\Big(-\frac{k'}{4}|x-y|\sqrt{|x|^2+|y|^2}-\frac{1}{4k}\frac{\sqrt{|x|^2+|y|^2}}{|x-y|}|x-y|^2\Big)\\
    &\lesssim (|x|^2+|y|^2)^\frac{d}{4}|x-y|^{-\frac{d}{2}}\exp\Big(-\Big(\frac{k'}{4}+\frac{1}{4k}\Big)|x-y|\sqrt{|x|^2+|y|^2}\Big).
\end{align*}

\end{proof}
\begin{prop}\label{p1} Let $1<p<\infty$, $\alpha\in [0,1)$, and $\w\in A^\alpha_p$. Then there exists a constant $C([\w]_{A_p^\alpha})>0$, depending on $[\w]_{A_p^{\alpha}}$  and $d$, such that
  \begin{align*}
  \begin{split}
&\sum_{n=0}^\infty\sum_{Q\in\Delta^\alpha_n}\int_{Q}\Bigg(\sup_{0<s<1}\sum_{l=0}^\infty\sum_{R\in\Delta_l^\alpha,\,R\cap N_\alpha(Q)=\emptyset}\int_{R}|f(y)|F_s(x,y)e^{-\frac{|y|^2}{2}}dy\Bigg)^p\w(x)dx\\
&\leq C([\w]_{A_p^\alpha})||f||^p_{L^p(\mathbb{R}^d,\w(x)e^{-\frac{p}{2}|x|^2}dx)}.
  \end{split}
\end{align*} 

\end{prop} 

We denote 
\begin{equation*}
     \Sigma_1:=\sum_{n=0}^\infty\sum_{Q\in\Delta^\alpha_n}\int_{Q}\Bigg(\sup_{0<s<1}\sum_{l=0}^\infty\sum_{R\in\Delta_l^\alpha,\,R\cap N_\alpha(Q)=\emptyset}\int_{R}|f(y)|F_s(x,y)e^{-\frac{|y|^2}{2}}dy\Bigg)^p\w(x)dx.
\end{equation*}
\begin{proof}\,  Since $\sup_{0<s<1}F_s(x,y)\leq \sup_{s>0}F_s(x,y)\leq F_{s_{\max}}(x,y)$, the sum $\Sigma_1$ is dominated as follows
    \begin{align*}
    \begin{split}
\Sigma_1%=\sum_{n=0}^\infty\sum_{Q\in\Delta^\alpha_n}\int_{Q}\Bigg(\sup_{0<s<1}\sum_{l=0}^\infty\sum_{R\in\Delta_l^\alpha,\,R\cap N_\alpha(Q)=\emptyset}\int_{R}|f(y)|F_s(x,y)e^{-\frac{|y|^2}{2}}dy\Bigg)^pw(x)dx\\
       &\leq \sum_{n=0}^\infty\sum_{Q\in\Delta^\alpha_n}\int_{Q}\Bigg(\sum_{l=0}^\infty\sum_{R\in\Delta_l^\alpha,\,R\cap N_\alpha(Q)=\emptyset}\int_{R}|f(y)|F_{s_{\max}}(x,y)e^{-\frac{|y|^2}{2}}dy\Bigg)^p\w(x)dx\\
        &= \sum_{Q\in\Delta^\alpha_0}\int_{Q}\Bigg(\sum_{l=1}^\infty\sum_{R\in\Delta_l^\alpha,\,R\cap N_\alpha(Q)=\emptyset}\int_{R}|f(y)|F_{s_{\max}}(x,y)e^{-\frac{|y|^2}{2}}dy\Bigg)^p\w(x)dx\\
        &\quad +\sum_{n=1}^\infty\sum_{Q\in\Delta^\alpha_n}\int_{Q}\Bigg(\sum_{l=0}^\infty\sum_{R\in\Delta_l^\alpha,\,R\cap N_\alpha(Q)=\emptyset}\int_{R}|f(y)|F_{s_{\max}}(x,y)e^{-\frac{|y|^2}{2}}dy\Bigg)^p\w(x)dx\\
        &=:(\Sigma_1)_0+(\Sigma_1)_{n\geq 1},
    \end{split}
    \end{align*}
   because if $R\cap N_\alpha(Q)= \emptyset$, then $Q\in\Delta^\alpha_0$ implies $R\notin \Delta^\alpha_0$, and $R\in \Delta_0^\alpha$ implies $Q\notin \Delta^\alpha_0$.

    \textit{Case $n=0$:}\\

As $\rho(Q,R)\geq l(Q)=2\geq 1$ and $l> 0$, then by Lemma \ref{G+}, using $|x|\leq\sqrt{d}$ for $x\in Q=[-1,1)^d$ and $|y|\leq\sqrt{d}(l+1)$ for $y\in R\in\Delta^\alpha_l$, so that $|x|^2+|y|^2\lesssim 1+l^2$, and using $|x-y|\geq\rho(Q,R)\geq 2\geq 1$, so that $|x-y|^{-\frac{d}{2}}\leq 1$, we obtain
\begin{align*}
    \begin{split}
(\Sigma_1)_0&=\sum_{Q\in\Delta^\alpha_0}\int_{Q}\Bigg(\sum_{l=1}^\infty\sum_{R\in\Delta_l^\alpha,\,R\cap N_\alpha(Q)=\emptyset}\int_{R}|f(y)|F_{s_{\max}}(x,y)e^{-\frac{|y|^2}{2}}dy\Bigg)^p\w(x)dx\\
&\lesssim\sum_{Q\in\Delta_0^\alpha}\int_{Q}\Bigg(\sum_{l=1}^\infty\sum_{R\in\Delta_l^\alpha,\,R\cap N_\alpha(Q)=\emptyset}\int_{R} (|x|^2+|y|^2)^\frac{d}{4}|x-y|^{-\frac{d}{2}}e^{-k''|x-y|\sqrt{|x|^2+|y|^2}}|f(y)|e^{-\frac{|y|^2}{2}}dy\Bigg)^p\\
&\quad \times \w(x)dx\\
&\lesssim \sum_{Q\in\Delta_0^\alpha}\int_{Q}\Bigg(\sum_{l=1}^\infty\sum_{R\in\Delta_l^\alpha,\,R\cap N_\alpha(Q)=\emptyset}\int_{R} (1+l^2)^\frac{d}{4}e^{-k''\rho(Q,R)l}|f(y)|e^{-\frac{|y|^2}{2}}dy\Bigg)^p\w(x)dx\\
&\lesssim\sum_{Q\in\Delta_0^\alpha}\int_{Q}\Bigg(\sum_{l=1}^\infty\sum_{R\in\Delta_l^\alpha,\,R\cap N_\alpha(Q)=\emptyset}l^\frac{d}{2}e^{-\frac{k''}{2}\rho(Q,R)l} \int_{R} e^{-\frac{k''}{2}\rho(Q,R)l}|f(y)|e^{-\frac{|y|^2}{2}}dy\Bigg)^p\w(x)dx.
    \end{split}
\end{align*}
Moreover, by Hölder's inequality under the sum in $R$ and in $l$, we have that
\begin{align*}
\begin{split}
(\Sigma_1)_0&\lesssim\sum_{Q\in\Delta_0^\alpha}\int_{Q}\Bigg(\sum_{l=1}^\infty\sum_{R\in\Delta_l^\alpha,\,R\cap N_\alpha(Q)=\emptyset}l^\frac{p'd}{2}e^{-p'\frac{k''}{2}\rho(Q,R)l}\Bigg)^\frac{p}{p'}\sum_{l=1}^\infty\sum_{R\in\Delta_l^\alpha,\,R\cap N_\alpha(Q)=\emptyset}e^{-p\frac{k''}{2}\rho(Q,R)l}\\
&\quad \times\Bigg(\int_{R} |f(y)|e^{-\frac{|y|^2}{2}}dy\Bigg)^p\w(x)dx\\
&\lesssim\sum_{Q\in\Delta_0^\alpha}\int_{Q}\Bigg(\sum_{l=1}^\infty l^{\alpha d+d}l^\frac{p'd}{2}e^{-p'\frac{k''}{2}l}\Bigg)^\frac{p}{p'}\sum_{l=1}^\infty\sum_{R\in\Delta_l^\alpha,\,R\cap N_\alpha(Q)=\emptyset}e^{-p\frac{k''}{2}\rho(Q,R)l}\\
&\quad\times \Bigg(\int_{R} |f(y)|e^{-\frac{|y|^2}{2}}dy\Bigg)^p\w(x)dx\\
&\lesssim\sum_{Q\in\Delta_0^\alpha}\w(Q)\sum_{l=1}^\infty\sum_{R\in\Delta_l^\alpha,\,R\cap N_\alpha(Q)=\emptyset}e^{-p\frac{k''}{2}\rho(Q,R)l}\Bigg(\int_{R} |f(y)|e^{-\frac{|y|^2}{2}}\w^\frac{1}{p}(y)\w^{-\frac{1}{p}}(y)dy\Bigg)^p\\
&\leq\sum_{Q\in\Delta_0^\alpha}\sum_{l=1}^\infty\sum_{R\in\Delta_l^\alpha,\,R\cap N_\alpha(Q)=\emptyset}e^{-p\frac{k''}{2}\rho(Q,R)l}\w(Q)\w^{-\frac{p'}{p}}(R)^\frac{p}{p'}\int_{R}|f(y)|^p\w(y)e^{-p\frac{|y|^2}{2}}dy,
\end{split}
\end{align*}
with Hölder's inequality applied again to the last equation under the integral in $R$. Then by Lemma \ref{oh} and Remarks \ref{HDHDKDKDKLDJJDLK} and \ref{emm} , $\w^{-\frac{p'}{p}}(R)^\frac{p}{p'}\lesssim ([ \w]_{A_p^\alpha}4^{d})^{4p\rho(Q,R)l^{\alpha}}\w^{-\frac{p'}{p}}(Q)^ \frac{p}{p'}$ and therefore by the above, we have 
\begin{align*}
\begin{split}
 (\Sigma_1)_0&\leq   \sum_{Q\in\Delta_0^\alpha}\sum_{l=1}^\infty\sum_{R\in\Delta_l^\alpha,\,R\cap N_\alpha(Q)=\emptyset}e^{-p\frac{k''}{2}\rho(Q,R)l}([\w]_{A_p^\alpha}4^{d})^{4p\rho(Q,R)l^{\alpha}}\w(Q)\w^{-\frac{p'}{p}}(Q)^\frac{p}{p'}\\
 &\quad\times \int_{R}|f(y)|^p\w(y)e^{-p\frac{|y|^2}{2}}dy\\
&\lesssim\sum_{Q\in\Delta_0^\alpha}\sum_{l=1}^\infty\sum_{R\in\Delta_l^\alpha,\,R\cap N_\alpha(Q)=\emptyset}[\w]^p_{A_p^\alpha}|Q|^pe^{-p\frac{k''}{2}\rho(Q,R)l}([\w]_{A_p^\alpha}4^{d})^{4p\rho(Q,R)l^{\alpha}}\int_{R}|f(y)|^p\w(y)e^{-p\frac{|y|^2}{2}}dy\\
&=\sum_{Q\in\Delta_0^\alpha}\sum_{l=1}^\infty\sum_{R\in\Delta_l^\alpha,\,R\cap N_\alpha(Q)=\emptyset}[\w]^p_{A_p^\alpha}e^{p\rho(Q,R)l^\alpha \beta\Big(1-\frac{k''}{2\beta}l^{1-\alpha}\Big)}\int_{R}|f(y)|^p\w(y)e^{-p\frac{|y|^2}{2}}dy
\end{split}
 \end{align*}

 with $\beta:=4\ln([\w]_{A_p^\alpha}4^{d})$. However, for $l\geq l_0:=\Big\lceil\Big(\frac{4\beta}{k''}\Big)^\frac{1}{1-\alpha}\Big\rceil$,
 \begin{equation}\label{ompo}
     \exp\Big(p\rho(Q,R) l^\alpha\beta(1-\frac{k''}{2\beta}l^{1-\alpha})\Big)
    \leq \exp\Big(-p\beta l^\alpha\rho(Q,R)\Big).
 \end{equation}
  Thus, by posing
 \begin{equation*}
     (\Sigma_1')_{0,\;l\geq l_0}:=\sum_{Q\in\Delta_0^\alpha}\sum_{l=l_0}^\infty\sum_{R\in\Delta_l^\alpha,\,R\cap N_\alpha(Q)=\emptyset}[\w]^p_{A_p^\alpha}e^{p\rho(Q,R)l^\alpha \beta\Big(1-\frac{k''}{2\beta}l^{1-\alpha}\Big)}\int_{R}|f(y)|^p\w(y)e^{-p\frac{|y|^2}{2}}dy,
 \end{equation*}

 and
\begin{equation*}
     (\Sigma_1')_{0,\;l\leq  l_0-1}:=\sum_{Q\in\Delta_0^\alpha}\sum_{l=1}^{l_0-1}\sum_{R\in\Delta_l^\alpha,\,R\cap N_\alpha(Q)=\emptyset}[\w]^p_{A_p^\alpha}e^{p\rho(Q,R)l^\alpha \beta\Big(1-\frac{k''}{2\beta}l^{1-\alpha}\Big)}\int_{R}|f(y)|^p\w(y)e^{-p\frac{|y|^2}{2}}dy.
 \end{equation*}

We have by equation (\ref{ompo}) that

 \begin{align*}
 \begin{split}
  (\Sigma_1')_{0,\;l\geq l_0} &\leq [\w]^p_{A_p^\alpha}\sum_{Q\in\Delta_0^\alpha}\sum_{l=l_0}^\infty\sum_{R\in\Delta_l^\alpha,\,R\cap N_\alpha(Q)=\emptyset}e^{-p\beta l^\alpha\rho(Q,R)}\int_{R}|f(y)|^p\w(y)e^{-p\frac{|y|^2}{2}}dy\\
&\leq [\w]^p_{A_p^\alpha}2^d \sum_{l=l_0}^\infty\sum_{R\in\Delta_l^\alpha}e^{-pl^\alpha \beta}\int_{R}|f(y)|^p\w(y)e^{-p\frac{|y|^2}{2}}dy\\
&\leq \mu||f||^p_{L^p(\mathbb{R}^d,\w(x)e^{-\frac{p}{2}|x|^2}dx)}.
\end{split}
 \end{align*}
 with $\mu$ a constant depending on $[\w]_{A_p^\alpha}$. Also, we observe that
\begin{align*}
    (\Sigma_1')_{0,\;l\leq  l_0-1}&\leq\sum_{Q\in\Delta_0^\alpha}\sum_{l=1}^{l_0-1}\sum_{R\in\Delta_l^\alpha}[\w]^p_{A_p^\alpha}e^{p\rho(Q,R)l^\alpha \beta\Big(1-\frac{k''}{2\beta}l^{1-\alpha}\Big)}\int_{R}|f(y)|^p\w(y)e^{-p\frac{|y|^2}{2}}dy \\
    &\leq \sum_{Q\in\Delta_0^\alpha}\sum_{l=1}^{l_0-1}\sum_{R\in\Delta_l^\alpha}[\w]^p_{A_p^\alpha}e^{p(l_0+1)l_0^{\alpha} \beta}\int_{R}|f(y)|^p\w(y)e^{-p\frac{|y|^2}{2}}dy \\
    &\leq 2^d[\w]^p_{A_p^\alpha}e^{p(l_0+1)l_0^{\alpha} \beta}\sum_{l=1}^{l_0-1}\sum_{R\in\Delta_l^\alpha}\int_{R}|f(y)|^p\w(y)e^{-p\frac{|y|^2}{2}}dy\\
    &\leq \mu'||f||^p_{L^p(\mathbb{R}^d,\w(x)e^{-\frac{p}{2}|x|^2}dx)},
 \end{align*}
with $\mu'$ a constant depending on $[\w]_{A_p^\alpha}$. Therefore, the sum $(\Sigma_1)_0$ is controlled by 

\begin{equation*}
    \mu{''}||f||^p_{L^p(\mathbb{R}^d,\w(x)e^{ -\frac{p}{2}|x|^2}dx)}
\end{equation*}
 with $\mu{''}$ a positive constant depending on $[\w]_{A_p^\alpha}$, because
 \begin{equation*}
    (\Sigma_1)_0\lesssim (\Sigma_1')_{0,\;l\leq  l_0-1}+(\Sigma_1')_{0,\;l\geq l_0}.
 \end{equation*}
 \textit{Case $n\geq 1$ :}\\

In this case we obtain, by Minkowski's inequality,
 \begin{align*}
     \begin{split}
        (\Sigma_1)_{n\geq1}^\frac{1}{p}&\leq  \Bigg[\sum_{n=1}^\infty\sum_{Q\in\Delta^\alpha_n}\int_{Q}\Bigg(\sum_{l=0}^{\infty}\sum_{R\in\Delta_l^\alpha,\,R\cap N_\alpha(Q)=\emptyset}\int_{R}|f(y)|F_{s_{\max}}(x,y)e^{-\frac{|y|^2}{2}}dy\Bigg)^p\w(x)dx\Bigg]^\frac{1}{p}\\
&=\Bigg[\sum_{n=1}^\infty\sum_{Q\in\Delta^\alpha_n}\int_{Q}\Bigg(\sum_{l=0}^{n-1}\sum_{R\in\Delta_l^\alpha,\,R\cap N_\alpha(Q)=\emptyset}\int_{R}|f(y)|F_{s_{\max}}(x,y)e^{-\frac{|y|^2}{2}}dy\\
&\quad+ \sum_{l=n}^\infty\sum_{R\in\Delta_l^\alpha,\,R\cap N_\alpha(Q)=\emptyset}\int_{R}|f(y)|F_{s_{\max}}(x,y)e^{-\frac{|y|^2}{2}}dy\Bigg)^p\w(x)dx\Bigg]^\frac{1}{p}\\
&\leq\Bigg[\sum_{n=1}^\infty\sum_{Q\in\Delta^\alpha_n}\int_{Q}\Bigg(\sum_{l=0}^{n-1}\sum_{R\in\Delta_l^\alpha,\,R\cap N_\alpha(Q)=\emptyset}\int_{R}|f(y)|F_{s_{\max}}(x,y)e^{-\frac{|y|^2}{2}}dy\Bigg)^p\w(x)dx\Bigg]^\frac{1}{p}\\
&\quad+\Bigg[\sum_{n=1}^\infty\sum_{Q\in\Delta^\alpha_n}\int_{Q}\Bigg(\sum_{l=n}^{\infty}\sum_{R\in\Delta_l^\alpha,\,R\cap N_\alpha(Q)=\emptyset}\int_{R}|f(y)|F_{s_{\max}}(x,y)e^{-\frac{|y|^2}{2}}dy\Bigg)^p\w(x)dx\Bigg]^\frac{1}{p}\\
&=:S_1^\frac{1}{p}+S_2^\frac{1}{p},
     \end{split}
 \end{align*}
with 
 \begin{equation*}
S_1=\sum_{n=1}^\infty\sum_{Q\in\Delta^\alpha_n}\int_{Q}\Bigg(\sum_{l=0}^{n-1}\sum_{R\in\Delta_l^\alpha,\,R\cap N_\alpha(Q)=\emptyset}\int_{R}|f(y)|F_{s_{\max}}(x,y)e^{-\frac{|y|^2}{2}}dy\Bigg)^p\w(x)dx,
 \end{equation*}
 and 
 \begin{equation*}
   S_2=  \sum_{n=1}^\infty\sum_{Q\in\Delta^\alpha_n}\int_{Q}\Bigg(\sum_{l=n}^{\infty}\sum_{R\in\Delta_l^\alpha,\,R\cap N_\alpha(Q)=\emptyset}\int_{R}|f(y)|F_{s_{\max}}(x,y)e^{-\frac{|y|^2}{2}}dy\Bigg)^p\w(x)dx.
 \end{equation*}
 Thus the sum $(\Sigma_1)_{n\geq1}$ will be treated in two subcases.\\

\textit{Subcase $S_1$}:\\

As $n\geq 1$ then $|x-y|\geq \rho(Q,R)\geq l(Q)>2^{-1}n^{-\alpha}$. Also since $l\leq n$, then 
\begin{equation*}
 |x|^2+|y|^2\leq2d(n+1)^2=2dn^2(1+\frac{1}{n})^2\leq 8dn^2\lesssim n^2. 
\end{equation*}
Thus, we can dominate $S_1$ as follows
\begin{align*}
\begin{split}
S_1&=\sum_{n=1}^\infty\sum_{Q\in\Delta^\alpha_n}\int_{Q}\Bigg(\sum_{l=0}^{n-1}\sum_{R\in\Delta_l^\alpha,\,R\cap N_\alpha(Q)=\emptyset}\int_{R}|f(y)|(|x|^2+|y|^2)^\frac{d}{4}|x-y|^{-\frac{d}{2}}e^{-k''|x-y|\sqrt{|x|^2+|y|^2}}e^{-\frac{|y|^2}{2}}dy\Bigg)^p\\
&\quad \times\w(x)dx\\
&\lesssim \sum_{n=1}^\infty\sum_{Q\in\Delta^\alpha_n}\int_{Q}\Bigg(\sum_{l=0}^{n-1}\sum_{R\in\Delta_l^\alpha,\,R\cap N_\alpha(Q)=\emptyset}\int_{R}|f(y)|n^\frac{d}{2}n^{\alpha\frac{d}{2}}e^{-k''\rho(Q,R)n}e^{-\frac{|y|^2}{2}}dy\Bigg)^p\w(x)dx.
\end{split}
\end{align*}
Now, by Hölder's inequality under the sum in $R$ and in $l$, and under the integral in $R$ by introducing the fact that $1=\w^\frac{1}{p}(y) \w^{-\frac{1}{p}}(y)$, we then dominate $S_1$ as follows
\begin{align*}
S_1&\lesssim\sum_{n=1}^\infty\sum_{Q\in\Delta_n^\alpha}\Bigg(\sum_{l=0}^{n-1}\sum_{R\in\Delta_l^\alpha,\,R\cap N_\alpha(Q)=\emptyset}n^{p'\frac{d}{2}}n^{p'\alpha\frac{d}{2}}e^{-\frac{p'k''}{2}\rho(Q,R)n}\Bigg)^\frac{p}{p'}\\
&\quad\times\sum_{l=0}^{n-1}\sum_{R\in\Delta_l^\alpha,\,R\cap N_\alpha(Q)=\emptyset}e^{-p\frac{k''}{2}\rho(Q,R)n}\w(Q)\w^{-\frac{p'}{p}}(R)^\frac{p}{p'}\int_{R}|f(y)|^p\w(y)e^{-p\frac{|y|^2}{2}}dy\\
&\lesssim\sum_{n=1}^\infty\sum_{Q\in\Delta_n^\alpha}\Bigg(\sum_{l=0}^{n-1}\max(1,l)^{\alpha d+d}n^{p'\frac{d}{2}}n^{p'\alpha\frac{d}{2}}e^{-\frac{p'k''}{2}2^{-1}n^{1-\alpha}}\Bigg)^\frac{p}{p'}\\
&\quad\times\sum_{l=0}^{n-1}\sum_{R\in\Delta_l^\alpha,\,R\cap N_\alpha(Q)=\emptyset}e^{-p\frac{k''}{2}\rho(Q,R)n}\w(Q)\w^{-\frac{p'}{p}}(R)^\frac{p}{p'}\int_{R}|f(y)|^p\w(y)e^{-p\frac{|y|^2}{2}}dy\\
&\leq \sum_{n=1}^\infty\sum_{Q\in\Delta_n^\alpha}\Big(n^{\alpha d+d+1+p'\frac{d}{2}+p'\alpha\frac{d}{2}}e^{-\frac{p'k''}{4}n^{1-\alpha}}\Big)^\frac{p}{p'}\\
&\quad\times\sum_{l=0}^{n-1}\sum_{R\in\Delta_l^\alpha,\,R\cap N_\alpha(Q)=\emptyset}e^{-p\frac{k''}{2}\rho(Q,R)n}\w(Q)\w^{-\frac{p'}{p}}(R)^\frac{p}{p'}\int_{R}|f(y)|^p\w(y)e^{-p\frac{|y|^2}{2}}dy\\
&=\sum_{n=1}^\infty\sum_{Q\in\Delta_n^\alpha}\sum_{l=0}^{n-1}\sum_{R\in\Delta_l^\alpha,\,R\cap N_\alpha(Q)=\emptyset}N(n)^\frac{p}{p'}e^{-p\frac{k''}{2}\rho(Q,R)n}\w(Q)\w^{-\frac{p'}{p}}(R)^\frac{p}{p'}\int_{R}|f(y)|^p\w(y)e^{-p\frac{|y|^2}{2}}dy,
\end{align*}
with $N(n)=n^{\alpha d+d+1+p'\frac{d}{2}+p'\alpha\frac{d}{2}}e^{-\frac{p 'k''}{4}n^{1-\alpha}}$. Then by Lemma \ref{oh} and Remarks \ref{HDHDKDKDKLDJJDLK} and \ref{emm} , we have that
\begin{align*}
\begin{split}
    S_1&\lesssim \sum_{n=1}^\infty\sum_{Q\in\Delta_n^\alpha}\sum_{l=0}^{n-1}\sum_{R\in\Delta_l^\alpha, R\cap N_\alpha(Q)=\emptyset}N(n)^\frac{p}{p'}e^{-p\frac{k''}{2}\rho(Q,R)n}([\w]^p_{A_p^\alpha}4^{dp})^{4\rho(Q,R)n^{\alpha}}\w(R)\w^{-\frac{p'}{p}}(R)^\frac{p}{p'}\\
    &\quad \times \int_{R}|f(y)|^p\w(y)e^{-p\frac{|y|^2}{2}}dy\\
    &\lesssim \sum_{n=1}^\infty\sum_{Q\in\Delta_n^\alpha}\sum_{l=0}^{n-1}\sum_{R\in\Delta_l^\alpha,\,R\cap N_\alpha(Q)=\emptyset}[\w]^p_{A_p^\alpha}|R|^pN(n)^\frac{p}{p'}e^{-p\frac{k''}{2}\rho(Q,R)n}([\w]^p_{A_p^\alpha}4^{dp})^{4\rho(Q,R)n^{\alpha}}\\
    &\quad\times \int_{R}|f(y)|^p\w(y)e^{-p\frac{|y|^2}{2}}dy\\
    &\leq \sum_{n=1}^\infty\sum_{Q\in\Delta_n^\alpha}\sum_{l=0}^{n-1}\sum_{R\in\Delta_l^\alpha,\,R\cap N_\alpha(Q)=\emptyset}[\w]^p_{A_p^\alpha}N(n)^\frac{p}{p'}e^{-p\frac{k''}{2}\rho(Q,R)n}([\w]^p_{A_p^\alpha}4^{dp})^{4\rho(Q,R)n^{\alpha}}\\
    &\quad \times \int_{R}|f(y)|^p\w(y)e^{-p\frac{|y|^2}{2}}dy.
    \end{split}
\end{align*}
%en inversant l'ordre des sommes en $R$ et en $Q$.\
Furthermore, by setting $\beta=4\ln([\w]_{A_p^\alpha}4^{d})$, we observe that for all $n\geq n_0:=\Big\lceil\Big(\frac{4\beta }{k''}\Big)^\frac{1}{1-\alpha}\Big\rceil$,
 \begin{equation*}
     \frac{k''}{2\beta}n^{1-\alpha}\geq 2,
 \end{equation*}
 and so
 \begin{equation*}
     -\frac{k''}{4\beta}n^{1-\alpha}\geq 1- \frac{k''}{2\beta}n^{1-\alpha}.
 \end{equation*}
 Therefore for all $n\geq n_0$,
 \begin{equation}\label{n_0}
([\w]_{A_p^\alpha}4^{d})^{4p\rho(Q,R)n^\alpha}e^{-p\frac{k''}{2}\rho(Q,R)n}\leq e^{p\rho(Q,R)n^\alpha \beta\Big(1-\frac{k''}{2\beta}n^{1-\alpha}\Big)}\leq e^{-p\rho(Q,R)n^\alpha\frac{k''}{4}n^{1-\alpha}}.
 \end{equation}
 Let then 
 \begin{align*}
 \begin{split}
          (S_1')_{n\geq n_0}&:=\sum_{n=n_0}^\infty\sum_{Q\in\Delta_n^\alpha}\sum_{l=0}^{n-1}\sum_{R\in\Delta_l^\alpha,\,R\cap N_\alpha(Q)=\emptyset}[\w]^p_{A_p^\alpha}N(n)^\frac{p}{p'}e^{-p\frac{k''}{2}\rho(Q,R)n}([\w]^p_{A_p^\alpha}4^{dp})^{4\rho(Q,R)n^{\alpha}}\\
          &\quad \times\int_{R}|f(y)|^p\w(y)e^{-p\frac{|y|^2}{2}}dy
 \end{split}
 \end{align*}
 and
 \begin{align*}
 \begin{split}
     (S_1')_{n\leq n_0-1}&:=\sum_{n=1}^{n_0-1}\sum_{Q\in\Delta_n^\alpha}\sum_{l=0}^{n-1}\sum_{R\in\Delta_l^\alpha,\,R\cap N_\alpha(Q)=\emptyset}[\w]^p_{A_p^\alpha}N(n)^\frac{p}{p'}e^{-p\frac{k''}{2}\rho(Q,R)n}([\w]^p_{A_p^\alpha}4^{dp})^{4\rho(Q,R)n^{\alpha}}\\
     &\quad\times\int_{R}|f(y)|^p\w(y)e^{-p\frac{|y|^2}{2}}dy.
 \end{split}
 \end{align*}

Thus, by equation (\ref{n_0}), we have that

 \begin{align*}
     &(S_1')_{n\geq n_0}\\
     &\leq\sum_{n=n_0}^\infty\sum_{Q\in\Delta_n^\alpha}\sum_{l=0}^{n-1}\sum_{R\in\Delta_l^\alpha,\,R\cap N_\alpha(Q)=\emptyset}[\w]^p_{A_p^\alpha} N(n)^\frac{p}{p'}e^{-p\rho(Q,R)n^\alpha\frac{k''}{4}n^{1-\alpha}}\int_{R}|f(y)|^p\w(y)e^{-p\frac{|y|^2}{2}}dy\\
     &\leq \sum_{n=n_0}^\infty\sum_{Q\in\Delta_n^\alpha}\sum_{l=0}^{n-1}\sum_{R\in\Delta_l^\alpha,\,R\cap N_\alpha(Q)=\emptyset}[\w]^p_{A_p^\alpha}N(n)^\frac{p}{p'}e^{-\frac{pk''}{8}n^{1-\alpha}}\int_{R}|f(y)|^p\w(y)e^{-p\frac{|y|^2}{2}}dy\\
     &\leq \sum_{n=n_0}^\infty\sum_{Q\in\Delta_n^\alpha}[\w]^p_{A_p^\alpha}(n^{\alpha d+d+1+p'\frac{d}{2}+p'\alpha\frac{d}{2}}e^{-\frac{p'k''}{4}n^{1-\alpha}})^\frac{p}{p'}e^{-\frac{pk''}{8}n^{1-\alpha}}\sum_{l=0}^{n-1}\sum_{R\in\Delta_l^\alpha}\int_{R}|f(y)|^p\w(y)e^{-p\frac{|y|^2}{2}}dy\\
&\lesssim\sum_{n=n_0}^\infty[\w]^p_{A_p^\alpha}n^{\alpha d+d}(n^{\alpha d+d+1+p'\frac{d}{2}+p'\alpha\frac{d}{2}}e^{-\frac{p'k''}{4}n^{1-\alpha}})^\frac{p}{p'}e^{-\frac{pk''}{8}n^{1-\alpha}}||f||^p_{L^p(\mathbb{R}^d,\w(x)e^{-\frac{p}{2}|x|^2})}\\
&=[\w]^p_{A_p^\alpha}||f||^p_{L^p(\mathbb{R}^d,\w(x)e^{-\frac{p}{2}|x|^2})}\sum_{n=n_0}^\infty e^{[(\alpha d+d)+(\alpha d+d+1)\frac{p}{p'}+p\frac{d}{2}+p\alpha\frac{d}{2}]\ln n-3\frac{pk''}{8}n^{1-\alpha}}\\
     &\leq  K_1||f||^p_{L^p(\mathbb{R}^d,\w(x)e^{-\frac{p}{2}|x|^2} dx)}.
 \end{align*}
  with $K_1$ a constant depending on $[\w]_{A_p^\alpha}$. Then, we also observe that

  \begin{align*}
     &(S_1')_{n\leq n_0-1}\\
     &\leq\sum_{n=1}^{n_0-1}\sum_{Q\in\Delta_n^\alpha}\sum_{l=0}^{n-1}\sum_{R\in\Delta_l^\alpha,\,R\cap N_\alpha(Q)=\emptyset}[\w]^p_{A_p^\alpha}N(n)^\frac{p}{p'}([\w]^p_{A_p^\alpha}4^{dp})^{4\rho(Q,R)n^{\alpha}}\int_{R}|f(y)|^p\w(y)e^{-p\frac{|y|^2}{2}}dy\\
     &\leq \sum_{n=1}^{n_0-1}\sum_{Q\in\Delta_n^\alpha}\sum_{l=0}^{n-1}\sum_{R\in\Delta_l^\alpha,\,R\cap N_\alpha(Q)=\emptyset}[\w]^p_{A_p^\alpha}N(n)^\frac{p}{p'}([\w]^p_{A_p^\alpha}4^{dp})^{4\times 2(n+1)n^{\alpha}}\int_{R}|f(y)|^p\w(y)e^{-p\frac{|y|^2}{2}}dy\\
     &\leq \sum_{n=1}^{n_0-1}\sum_{Q\in\Delta_n^\alpha}\sum_{l=0}^{n-1}\sum_{R\in\Delta_l^\alpha}[\w]^p_{A_p^\alpha}N'(n_0)^\frac{p}{p'}([\w]^p_{A_p^\alpha}4^{dp})^{8n_0^{\alpha+1}}\int_{R}|f(y)|^p\w(y)e^{-p\frac{|y|^2}{2}}dy\\
     &\lesssim[\w]^p_{A_p^\alpha}N'(n_0)^\frac{p}{p'}([\w]^p_{A_p^\alpha}4^{dp})^{8n_0^{\alpha+1}} \sum_{n=1}^{n_0-1}n^{\alpha d+d}\sum_{l=0}^{n-1}\sum_{R\in\Delta_l^\alpha}\int_{R}|f(y)|^p\w(y)e^{-p\frac{|y|^2}{2}}dy\\
     &\leq n_0^{\alpha d+d+1} [\w]^p_{A_p^\alpha}N'(n_0)^\frac{p}{p'}([\w]^p_{A_p^\alpha}4^{dp})^{8n_0^{\alpha+1}}||f||^p_{L^p(\mathbb{R}^d,\w(x)e^{-\frac{p}{2}|x|^2}dx)}\\
     &\leq K_2 ||f||^p_{L^p(\mathbb{R}^d,\w(x)e^{-\frac{p}{2}|x|^2}dx)}
  \end{align*}
with $N'(n_0)=n_0^{\alpha d+d+1+p'\frac{d}{2}+p'\alpha\frac{d}{2}}$ a constant and $K_2$ a constant depending on $[\w]_{A_p^\alpha}$. Then we also deduce that $S_1$ is controlled by $K||f||^p_{L^p(\mathbb{R}^d,\w(x)e^{-\frac{p}{2} |x|^2}dx)}$ with $K$ a constant depending on $[\w]_{A_p^\alpha}$ because

  \begin{equation*}
    S_1\lesssim  (S_1')_{n\leq n_0-1}+(S_1')_{n\geq n_0}.
  \end{equation*}

 \textit{ Subcase $S_2$}:\\

In this case $n\leq l$, so that $(|x|^2+|y|^2)^\frac{d}{4}\lesssim l^\frac{d}{2}$. Moreover,

\begin{equation*}
    |x-y|\geq \rho(Q,R)\geq l(Q)>2^{-1}n^{-\alpha}
\end{equation*}
(since $R\cap N_\alpha(Q)=\emptyset$), so we have
\begin{align*}
\begin{split}
S_2&=\sum_{n=1}^\infty\sum_{Q\in\Delta^\alpha_n}\int_{Q}\Bigg(\sum_{l=n}^{\infty}\sum_{R\in\Delta_l^\alpha,\,R\cap N_\alpha(Q)=\emptyset}\int_{R}|f(y)|(|x|^2+|y|^2)^\frac{d}{4}|x-y|^{-\frac{d}{2}}e^{-k''|x-y|\sqrt{|x|^2+|y|^2}}e^{-\frac{|y|^2}{2}}dy\Bigg)^p\\
&\quad \times \w(x)dx\\
&\lesssim \sum_{n=1}^\infty\sum_{Q\in\Delta^\alpha_n}\int_{Q}\Bigg(\sum_{l=n}^{\infty}\sum_{R\in\Delta_l^\alpha,\,R\cap N_\alpha(Q)=\emptyset}\int_{R}|f(y)|l^\frac{d}{2}n^{\alpha\frac{d}{2}}e^{-k''\rho(Q,R)l}e^{-\frac{|y|^2}{2}}dy\Bigg)^p\w(x)dx\\
&=\sum_{n=1}^\infty\sum_{Q\in\Delta^\alpha_n}\int_{Q}\Bigg(\sum_{l=n}^{\infty}\sum_{R\in\Delta_l^\alpha,\,R\cap N_\alpha(Q)=\emptyset}l^\frac{d}{2}n^{\alpha\frac{d}{2}}e^{-\frac{k''}{2}\rho(Q,R)l}\int_{R}e^{-\frac{k''}{2}\rho(Q,R)l}|f(y)|e^{-\frac{|y|^2}{2}}dy\Bigg)^p\w(x)dx,
\end{split}
\end{align*}
Thus by Hölder's inequality under the sum in $l$ and in $R$, and under the integral in $R$ by adding the weight $\w$, we have that

\begin{align}\label{as}
    \begin{split}
S_2&\lesssim\sum_{n=1}^\infty\sum_{Q\in\Delta_n^\alpha}\Bigg(\sum_{l=n}^{\infty}\sum_{R\in\Delta_l^\alpha,\,R\cap N_\alpha(Q)=\emptyset}l^{p'\frac{d}{2}}n^{p'\alpha\frac{d}{2}}e^{-\frac{p'k''}{2}\rho(Q,R)l}\Bigg)^\frac{p}{p'}\\
&\quad\times\sum_{l=n}^{\infty}\sum_{R\in\Delta_l^\alpha,\,R\cap N_\alpha(Q)=\emptyset}e^{-p\frac{k''}{2}\rho(Q,R)l}\w(Q)\w^{-\frac{p'}{p}}(R)^\frac{p}{p'}\int_{R}|f(y)|^p\w(y)e^{-p\frac{|y|^2}{2}}dy.
    \end{split}
\end{align}
However,
\begin{align*}
\sum_{l=n}^{\infty}\sum_{R\in\Delta_l^\alpha,\,R\cap N_\alpha(Q)=\emptyset}l^{p'\frac{d}{2}}n^{p'\alpha\frac{d}{2}}e^{-\frac{p'k''}{2}\rho(Q,R)l}&\leq\sum_{l=n}^{\infty}\sum_{R\in\Delta_l^\alpha}l^{p'\frac{d}{2}}n^{p'\alpha\frac{d}{2}}e^{-\frac{p'k''}{4}n^{-\alpha}l}\\
&\lesssim\sum_{l=n}^{\infty}l^{\alpha d+ d}l^{p'\frac{d}{2}}n^{p'\alpha\frac{d}{2}}e^{-\frac{p'k''}{4}n^{-\alpha}l}\\
&\leq \sum_{l=n}^{\infty}l^{\alpha d+ d}l^{p'\frac{d}{2}}l^{p'\alpha\frac{d}{2}}e^{-\frac{p'k''}{4}l^{1-\alpha}}\\
&\leq M,
\end{align*}
with $M$ a constant, since $\alpha\in [0,1)$. Thus, from equation (\ref{as}), we end up with
\begin{align*}
\begin{split}
&S_2\\
&\lesssim\sum_{n=1}^\infty\sum_{Q\in\Delta_n^\alpha}\sum_{l=n}^{\infty}\sum_{R\in\Delta_l^\alpha,\,R\cap N_\alpha(Q)=\emptyset}e^{-p\frac{k''}{2}\rho(Q,R)l}\w(Q)\w^{-\frac{p'}{p}}(R)^\frac{p}{p'}\int_{R}|f(y)|^p\w(y)e^{-p\frac{|y|^2}{2}}dy\\
&\lesssim \sum_{n=1}^\infty\sum_{Q\in\Delta_n^\alpha}\sum_{l=n}^{\infty}\sum_{R\in\Delta_l^\alpha,\,R\cap N_\alpha(Q)=\emptyset}e^{-p\frac{k''}{2}\rho(Q,R)l}([\w]^p_{A_p^\alpha}4^{dp})^{4\rho(Q,R)l^{\alpha}}\w(Q)\w^{-\frac{p'}{p}}(Q)^\frac{p}{p'}\\
&\quad \times \int_{R}|f(y)|^p\w(y)e^{-p\frac{|y|^2}{2}}dy\\
&\lesssim \sum_{n=1}^\infty\sum_{Q\in\Delta_n^\alpha}\sum_{l=n}^{\infty}\sum_{R\in\Delta_l^\alpha,\,R\cap N_\alpha(Q)=\emptyset}[\w]^p_{A_p^\alpha}|Q|^pe^{-p\frac{k''}{2}\rho(Q,R)l}([\w]^p_{A_p^\alpha}4^{dp})^{4\rho(Q,R)l^{\alpha}}\int_{R}|f(y)|^p\w(y)e^{-p\frac{|y|^2}{2}}dy\\
&=\sum_{l=1}^\infty\sum_{R\in\Delta_l^\alpha}\sum_{n=1}^{l}\sum_{Q\in\Delta_n^\alpha,\,R\cap N_\alpha(Q)=\emptyset}[\w]^p_{A_p^\alpha}e^{-p\frac{k''}{2}\rho(Q,R)l}([\w]^p_{A_p^\alpha}4^{dp})^{4\rho(Q,R)l^{\alpha}}\int_{R}|f(y)|^p\w(y)e^{-p\frac{|y|^2}{2}}dy
\end{split}
\end{align*}
since $\w^{-\frac{p'}{p}}(R)^\frac{p}{p'}\lesssim([\w]^p_{A_p^\alpha}4^{dp}) ^{4\rho(Q,R)l^{\alpha}} \w^{-\frac{p'}{p}}(Q)^\frac{p}{p'}$ by Lemma \ref{oh} and Remarks \ref{emm} and \ref{HDHDKDKDKLDJJDLK}. Then by taking the constant $\beta=4\ln([\w]_{A_p^\alpha}4^{d})$, we observe that

\begin{align*}
S_2&\lesssim[\w]^p_{A_p^\alpha}\sum_{l=1}^\infty\sum_{R\in\Delta_l^\alpha}\sum_{n=1}^{l}\sum_{Q\in\Delta_n^\alpha,\,R\cap N_\alpha(Q)=\emptyset}\exp\Big({\beta p\rho(Q,R)l^{\alpha}-p\frac{k''}{2}\rho(Q,R)l}\Big)\int_{R}|f(y)|^p\w(y)e^{-p\frac{|y|^2}{2}}dy\\
&=[\w]^p_{A_p^\alpha}\sum_{l=1}^\infty\sum_{R\in\Delta_l^\alpha}\sum_{n=1}^{l}\sum_{Q\in\Delta_n^\alpha,\,R\cap N_\alpha(Q)=\emptyset}\exp\Big(\beta p\rho(Q,R)l^\alpha \Big(1-\frac{k''}{2\beta}l^{1-\alpha}\Big)\Big)\int_{R}|f(y)|^p\w(y)e^{-p\frac{|y|^2}{2}}dy.
\end{align*}

We observe that for $l\geq l_0:=\Big\lceil\Big(\frac{4\beta}{k''}\Big)^\frac{1}{1-\alpha}\Big\rceil$
\begin{equation}\label{e_0}
  \exp\Big(\beta p\rho(Q,R)l^\alpha \Big(1-\frac{k''}{2\beta}l^{1-\alpha}\Big)\Big)\leq \exp\Big(-p\rho(Q,R)l^\alpha\frac{k''}{4}l^{1-\alpha}\Big).
\end{equation}

Let us then pose
\begin{align*}
\begin{split}
    (S_2')_{l\geq l_0}&:=[\w]^p_{A_p^\alpha}\sum_{l=l_0}^\infty\sum_{R\in\Delta_l^\alpha}\sum_{n=1}^{l}\sum_{Q\in\Delta_n^\alpha,\,R\cap N_\alpha(Q)=\emptyset}\exp\Big(\beta p\rho(Q,R)l^\alpha \Big(1-\frac{k''}{2\beta}l^{1-\alpha}\Big)\Big)\\
    &\quad \times \int_{R}|f(y)|^p\w(y)e^{-p\frac{|y|^2}{2}}dy
    \end{split}
\end{align*}
and
\begin{align*}
\begin{split}
    (S_2')_{l\leq l_0-1}&:=[\w]^p_{A_p^\alpha}\sum_{l=1}^{l_0-1}\sum_{R\in\Delta_l^\alpha}\sum_{n=1}^{l}\sum_{Q\in\Delta_n^\alpha,\,R\cap N_\alpha(Q)=\emptyset}\exp\Big(\beta p\rho(Q,R)l^\alpha \Big(1-\frac{k''}{2\beta}l^{1-\alpha}\Big)\Big)\\
    &\quad\times \int_{R}|f(y)|^p\w(y)e^{-p\frac{|y|^2}{2}}dy .
 \end{split}
\end{align*}
 By equation (\ref{e_0}), we have that

\begin{align*}
    (S_2')_{l\geq l_0}%&:=[w]^p_{A_p^\alpha}\sum_{l=l_0}^\infty\sum_{R\in\Delta_l^\alpha,\,R\cap N_\alpha(Q)=\emptyset}\sum_{n=1}^{l}\sum_{Q\in\Delta_n^\alpha}\exp\Big[Zp\rho(Q,R)l^\alpha \Big(1-\frac{k''}{2Z}l^{1-\alpha}\Big)\Big]\int_{R}|f(y)|^pw(y)e^{-p\frac{|y|^2}{2}}dy\\
    &\leq [\w]^p_{A_p^\alpha}\sum_{l=l_0}^\infty\sum_{R\in\Delta_l^\alpha}\sum_{n=1}^{l}\sum_{Q\in\Delta_n^\alpha,\,R\cap N_\alpha(Q)=\emptyset}\exp\Big(-p\rho(Q,R)l^\alpha\frac{k''}{4}l^{1-\alpha}\Big)\int_{R}|f(y)|^p\w(y)e^{-p\frac{|y|^2}{2}}dy\\
   % &\leq [w]^p_{A_p^\alpha}\sum_{l=l_0}^\infty\sum_{R\in\Delta_l^\alpha}\sum_{n=1}^{l}\sum_{Q\in\Delta_n^\alpha}\exp\Big(-n^{-\alpha}l^\alpha\frac{pk''}{8}l^{1-\alpha}\Big)\int_{R}|f(y)|^pw(y)e^{-p\frac{|y|^2}{2}}dy\\
&\lesssim[\w]^p_{A_p^\alpha}\sum_{l=l_0}^\infty\sum_{R\in\Delta_l^\alpha}\sum_{n=1}^{l}n^{\alpha d+d}\exp\Big(-n^{-\alpha}l^\alpha\frac{pk''}{8}l^{1-\alpha}\Big)\int_{R}|f(y)|^p\w(y)e^{-p\frac{|y|^2}{2}}dy\\
%&\leq[w]^p_{A_p^\alpha}\sum_{l=l_0}^\infty\sum_{R\in\Delta_l^\alpha}\sum_{n=1}^{l}n^{\alpha d+d}\exp\Big(-\frac{pk''}{8}l^{1-\alpha}\Big)\int_{R}|f(y)|^pw(y)e^{-p\frac{|y|^2}{2}}dy\\
&\leq[\w]^p_{A_p^\alpha}\sum_{l=l_0}^\infty\sum_{R\in\Delta_l^\alpha}l^{\alpha d+d+1}\exp\Big(-\frac{pk''}{8}l^{1-\alpha}\Big)\int_{R}|f(y)|^p\w(y)e^{-p\frac{|y|^2}{2}}dy\\
&\leq \sup_{l\geq 1}\exp\Big(-\frac{pk''}{8}l^{1-\alpha}+(\alpha d+d+1)\ln l\Big)[\w]^p_{A_p^\alpha}\sum_{l=l_0}^\infty\sum_{R\in\Delta_l^\alpha}\int_{R}|f(y)|^p\w(y)e^{-p\frac{|y|^2}{2}}dy\\
&\leq \mu||f||^p_{L^p(\mathbb{R}^d,\w(x)e^{-\frac{p}{2}|x|^2}dx)}.
\end{align*}
with $\mu$ a constant depending on $[\w]_{A_p^\alpha}$. Furthermore, we observe that

\begin{align*}
    (S_2')_{l\leq l_0-1}&\leq[\w]^p_{A_p^\alpha}\sum_{l=1}^{l_0-1}\sum_{R\in\Delta_l^\alpha}\sum_{n=1}^{l}\sum_{Q\in\Delta_n^\alpha,\,R\cap N_\alpha(Q)=\emptyset}\exp\Big(\beta p\rho(Q,R)l^\alpha \Big)\int_{R}|f(y)|^p\w(y)e^{-p\frac{|y|^2}{2}}dy  \\
    &\leq[\w]^p_{A_p^\alpha}e^{2\beta p\,l_0^{\alpha+1}}\sum_{l=1}^{l_0-1}\sum_{R\in\Delta_l^\alpha}\sum_{n=1}^{l}\sum_{Q\in\Delta_n^\alpha}\int_{R}|f(y)|^p\w(y)e^{-p\frac{|y|^2}{2}}dy  \\
    &\lesssim[\w]^p_{A_p^\alpha}e^{2\beta p\,l_0^{\alpha+1}}\,l_0^{\alpha d+d+1}\sum_{l=1}^{l_0-1}\sum_{R\in\Delta_l^\alpha}\int_{R}|f(y)|^p\w(y)e^{-p\frac{|y|^2}{2}}dy  \\
    &\leq \mu'||f||^p_{L^p(\mathbb{R}^d,\w(x)e^{-\frac{p}{2}|x|^2}dx)},
\end{align*}
where the second line uses $1-\frac{k''}{2\beta}l^{1-\alpha}\leq 1$ together with $\rho(Q,R)\leq (n+1)+(l+1)\leq 2l_0$ for $n\leq l\leq l_0-1$, and the third line uses $\sum_{n=1}^{l}\#\Delta^\alpha_n\lesssim l^{\alpha d+d+1}\leq l_0^{\alpha d+d+1}$ by Lemma \ref{cardi}.
 Here, $\mu'$ is a constant depending on $[\w]_{A_p^\alpha}$. However
\begin{equation*}
    S_2\lesssim (S_2')_{l\leq l_0-1}+(S_2')_{l\geq l_0}.
\end{equation*}
Thus $S_2$ is controlled by $\mu''||f||^p_{L^p(\mathbb{R}^d,\w(x)e^{-\frac{p}{2}|x|^2 }dx)}$ with $\mu''$ a constant depending on $[\w]_{A_p^\alpha}$.

\end{proof}

  \begin{prop}\label{p2}   Let $1<p<\infty$, $\alpha\in [0,1)$, and $\w\in A^\alpha_p$. Then there exists a constant $C([\w]_{A_p^\alpha})>0$, depending on  $[\w]_{A_p^\alpha}$ and $d$, such that
  \begin{equation*}
\sum_{Q\in\Delta^\alpha}\int_{Q}\Bigg(\sup_{0<s<1}\sum_{R\in\Delta^\alpha,\; R\subseteq N_\alpha(Q)}\int_{R}|f(y)|F_s(x,y)e^{-\frac{|y|^2}{2}}dy\Bigg)^p\w(x)dx\leq C([\w]_{A_p^\alpha},d)||f||^p_{L^p(\mathbb{R}^d,\w(x)e^{-\frac{p}{2}|x|^2}dx)}
\end{equation*} 
\end{prop}  
\begin{proof}\, By direct calculations, we observe
\begin{align}
\begin{split}
\Sigma_2&:=\sum_{Q\in\Delta^\alpha}\int_{Q}\Bigg(\sup_{0<s<1}\frac{1}{s^\frac{d}{2}}\sum_{R\in\Delta^\alpha,\; R\subseteq N_\alpha(Q)}\int_{R}|f(y)|\exp\Big(-\frac{s}{4}(|x|^2+|y|^2)-\frac{1}{4s}|x-y|^2\Big)e^{-\frac{|y|^2}{2}}dy\Bigg)^p\\
&\quad \times \w(x)dx\\
&\leq \sum_{Q\in\Delta^\alpha}\int_{Q}\Bigg(\sup_{0<s<1}\frac{1}{s^\frac{d}{2}}\sum_{R\in\Delta^\alpha,\; R\subseteq N_\alpha(Q)}\int_{R}|f(y)|\exp\Big(-\frac{1}{4s}|x-y|^2\Big)e^{-\frac{|y|^2}{2}}dy\Bigg)^p\w(x)dx\\
&\lesssim\sum_{Q\in\Delta^\alpha}\int_{Q}\Bigg(\sum_{R\in\Delta^\alpha,\; R\subseteq N_\alpha(Q)}\sup_{0<s<1}T^{\Delta}_s(|f|\chi_{R}e^{-\frac{|\;.\;|^2}{2}})(x)\Bigg)^p\w(x)dx,\label{Neymar1}
\end{split}
\end{align}
where the last line uses $s^{-\frac{d}{2}}=(4\pi)^{\frac{d}{2}}(4\pi s)^{-\frac{d}{2}}$. Before applying Hölder's inequality on the sum in $R$, we bound the number of terms. If $R \in \Delta^\alpha_{n'}$ satisfies $R \subseteq N_\alpha(Q)$ with $Q \in \Delta^\alpha_n$, then $\rho(Q,R) < 4l(Q) \leq 8$, so that $|n-n'| \leq 8$ as in the proof of Lemma \ref{cardinalité2}, and therefore $l(R) \geq 2^{-8}l(Q)$ by Lemma \ref{lem-neighbour-sizes}(1). Since these cubes are pairwise disjoint and contained in the cube $N_\alpha(Q)$ of side length $4l(Q)$, we obtain
\begin{equation}\label{equ-count-inside}
\left|\left\{R \in \Delta^\alpha \;:\; R \subseteq N_\alpha(Q)\right\}\right| \leq \left(\frac{4l(Q)}{2^{-8}l(Q)}\right)^d = 2^{10d} .
\end{equation}
Then, by Hölder's inequality on the sum in $R$, we arrive at
\begin{align*}
\Sigma_2&\leq \sum_{Q\in\Delta^\alpha}\int_{Q}\Big(\sum_{R\in\Delta^\alpha,\; R\subseteq N_\alpha(Q)}1\Big)^\frac{p}{p'}\sum_{R\in\Delta^\alpha,\; R\subseteq N_\alpha(Q)}\sup_{0<s<1}T^{\Delta}_s(|f|\chi_{R}e^{-\frac{|\;.\;|^2}{2}})(x)^p\w(x)dx\\
&\leq \left(2^{10d}\right)^\frac{p}{p'}\sum_{Q\in\Delta^\alpha}\int_{Q}\sum_{R\in\Delta^\alpha,\; R\subseteq N_\alpha(Q)}\sup_{0<s<1}T^{\Delta}_s(|f|\chi_{R}e^{-\frac{|\;.\;|^2}{2}})(x)^p\w(x)dx\\
&\lesssim\sum_{Q\in\Delta^\alpha}\sum_{R\in\Delta^\alpha,\; R\subseteq N_\alpha(Q)}\Big|\Big|\sup_{0<s<1}T^{\Delta}_s(|f|\chi_{R}e^{-\frac{|\;.\;|^2}{2}})\Big|\Big|_{L^p(Q,\; \w(x)dx)}^p.
\end{align*}
However, by Lemma \ref{extension}, there exists a weight $v_Q$ of $A_p(\mathbb{R}^d)$, such that for $Q\in \Delta^\alpha$, $v_Q|_{N_\alpha(Q)}=\w$ and $[v_Q]_{A_p(\R^d)} \leq \max(2^{dp},17^{dp})[\w]_{A_p(N_\alpha(Q))}\leq 17^{dp}[\w]_{A_p^\alpha} $, the bound being uniform in $Q$. Moreover, $R\subseteq N_\alpha(Q)$ implies $N_\alpha(Q)\cap R\ne\varnothing$, so that
\begin{equation}\label{equ-count-outside}
\left|\left\{Q\in\Delta^\alpha \;:\; R\subseteq N_\alpha(Q)\right\}\right|\leq 2^{20d}
\end{equation}
by Lemma \ref{cardinalité2}. Thus, $\Sigma_2$ becomes
\begin{align*}
\Sigma_2&\lesssim\sum_{Q\in\Delta^\alpha}\sum_{R\in\Delta^\alpha,\; R\subseteq N_\alpha(Q)}\Big|\Big|\sup_{0<s<1}T^{\Delta}_s(|f|\chi_{R}e^{-\frac{|\;.\;|^2}{2}})\Big|\Big|_{L^p(\mathbb{R}^d,\; v_Q(x)dx)}^p\\
&\lesssim\sum_{Q\in\Delta^\alpha}\sum_{R\in\Delta^\alpha,\; R\subseteq N_\alpha(Q)}C\left([v_Q]_{A_p}\right)||f||_{L^p(R,\; \w(x)e^{-\frac{p|x|^2}{2}}dx)}^p\\
&\leq C\left([\w]_{A_p^\alpha}\right)\sum_{Q\in\Delta^\alpha}\sum_{R\in\Delta^\alpha,\; R\subseteq N_\alpha(Q)}||f||_{L^p(R,\; \w(x)e^{-\frac{p|x|^2}{2}}dx)}^p\\
&=C\left([\w]_{A_p^\alpha}\right)\sum_{R\in\Delta^\alpha}\sum_{Q\in\Delta^\alpha,\; R\subseteq N_\alpha(Q)}||f||_{L^p(R,\; \w(x)e^{-\frac{p|x|^2}{2}}dx)}^p\\
&\leq 2^{20d}\,C\left([\w]_{A_p^\alpha}\right)\sum_{R\in\Delta^\alpha}||f||_{L^p(R,\; \w(x)e^{-\frac{p|x|^2}{2}}dx)}^p\\
&\leq C([\w]_{A_p^\alpha},d) ||f||_{L^p(\mathbb{R}^d,\; \w(x)e^{-\frac{p|x|^2}{2}}dx)}^p.
\end{align*}

\end{proof}

We can now prove Proposition~\ref{suski}.

\begin{proof}(\textbf{Proposition~\ref{suski}})\, By equation~(\ref{968}), we have
    \begin{align*}
    \begin{split}
&||\sup_{t>0}\mathcal{H}_t(|f|\chi_{ G_+(x)})||_{L^p(\mathbb{R}^d,\w(x)e^{-\frac{p}{2}|x|^2}dx)}\\
&\lesssim  \Bigg[\int_{\mathbb{R}^d}\Bigg(\sup_{0<s<1}\int_{\mathbb{R}^d}|f(y)|F_s(x,y)e^{-\frac{|y|^2}{2}}dy\Bigg)^p\w(x)dx\Bigg]^\frac{1}{p}\\
&=\Bigg[\sum_{Q\in\Delta^\alpha}\int_{Q}\Bigg(\sup_{0<s<1}\sum_{R\in\Delta^\alpha}\int_{R}|f(y)|F_s(x,y)e^{-\frac{|y|^2}{2}}dy\Bigg)^p\w(x)dx\Bigg]^\frac{1}{p}\\
&\leq \Bigg[\sum_{Q\in\Delta^\alpha}\int_{Q}\Bigg(\sup_{0<s<1}\sum_{R\in\Delta^\alpha,\;R\cap N_\alpha(Q)=\emptyset}\int_{R}|f(y)|F_s(x,y)e^{-\frac{|y|^2}{2}}dy\\
&\quad +\sup_{0<s<1}\sum_{R\in\Delta^\alpha,\;R\subseteq N_\alpha(Q)}\int_{R}|f(y)|F_s(x,y)e^{-\frac{|y|^2}{2}}dy\Bigg)^p\w(x)dx\Bigg]^\frac{1}{p}\\
&\leq  \Bigg[\sum_{Q\in\Delta^\alpha}\int_{Q}\Bigg(\sup_{0<s<1}\sum_{R\in\Delta^\alpha,\;R\cap N_\alpha(Q)=\emptyset}\int_{R}|f(y)|F_s(x,y)e^{-\frac{|y|^2}{2}}dy\Bigg)^p\w(x)dx\Bigg]^\frac{1}{p}\\
&\quad +\Bigg[\sum_{Q\in\Delta^\alpha}\int_{Q}\Bigg(\sup_{0<s<1}\sum_{R\in\Delta^\alpha,\;R\subseteq N_\alpha(Q)}\int_{R}|f(y)|F_s(x,y)e^{-\frac{|y|^2}{2}}dy\Bigg)^p\w(x)dx\Bigg]^\frac{1}{p},
 \end{split}
    \end{align*}
by Minkowski's inequality. We then observe that
\begin{align*}
 \begin{split}
 &||\sup_{t>0}\mathcal{H}_t(|f|\chi_{ G_+(x)})||_{L^p(\mathbb{R}^d,\w(x)e^{-\frac{p}{2}|x|^2}dx)}\\
 &\leq  \Bigg[\sum_{n=0}^\infty\sum_{Q\in\Delta_n^\alpha}\int_{Q}\Bigg(\sup_{0<s<1}\sum_{l=0}^\infty\sum_{R\in\Delta_l^\alpha,\;R\cap N_\alpha(Q)=\emptyset}\int_{R}|f(y)|F_s(x,y)e^{-\frac{|y|^2}{2}}dy\Bigg)^p\w(x)dx\Bigg]^\frac{1}{p}\\
&\quad+\Bigg[\sum_{Q\in\Delta^\alpha}\int_{Q}\Bigg(\sup_{0<s<1}\sum_{R\in\Delta^\alpha,\;R\subseteq N_\alpha(Q)}\int_{R}|f(y)|F_s(x,y)e^{-\frac{|y|^2}{2}}dy\Bigg)^p\w(x)dx\Bigg]^\frac{1}{p} \\
 &=\Sigma_1^\frac{1}{p}+\Sigma_2^\frac{1}{p}.
 \end{split}   
\end{align*}
Hence the result follows by Propositions \ref{p1} and \ref{p2}.
\end{proof}

By Proposition $\ref{susuki}$ and Proposition $\ref{suski}$, we arrive at the second main result of this article:
   \begin{thm}\label{result2}
      Let $1<p<\infty$, $\alpha\in [0,1)$, and $\w\in A^\alpha_p$. Then there exists a constant $C([\w]_{A_p^\alpha})>0$, depending on  $[\w]_{A_p^\alpha}$ and $d$, such that
       \begin{equation*}
           ||\mathcal{H}^*f||^p_{L^p(\w(x)e^{-\frac{p|x|^2}{2}}dx)}\leq C({[\w]_{A_p^\alpha})}||f||^p_{ L^p(\w(x)e^{-\frac{p|x|^2}{2}}dx)}.
       \end{equation*}     
    \end{thm}

\begin{rem}
 At the end of the numerous parts of the proof of Theorem \ref{result2}, it can be instructive to revisit the strategies of these parts of proof.
 Namely, recalling the Mehler kernel from \eqref{equ-intro-Mehler} in the introduction,  associated with the Ornstein-Uhlenbeck semigroup, we have
 \begin{align*}
     \mathcal{H}^*|f|(x) & = \sup_{t > 0} \int_{\R^d} M_t(x,y) |f(y)| dy \\
     & = \sup_{t > 0} \int_{\R^d} M_t(x,y) |f(y)| \times \left( \chi_{G_-(x) \cap D_d^C}(y) + \chi_{G_-(x) \cap D_d}(y) \chi_{D_d^C}(x) \right. \\
     & \left. + \chi_{G_-(x) \cap D_d}(y) \chi_{D_d}(x) + \chi_{G_+(x) \cap N_\alpha(R_x)}(y) + \chi_{G_+(x) \cap N_\alpha(R_x)^C}(y) \right) dy
 \end{align*}
 Now for the first two summands and the last one, we have estimated $\sup_{t > 0} \int_{\R^d} \ldots dy$ crudely by $\int_{\R^d} \sup_{t > 0} \ldots dy$, and for the third and forth summands, we have used the pointwise control of these parts against the heat maximal operator together with the Extension Lemma \ref{extension}.
\end{rem}

%Now, we will discuss in the following subsection some immediate consequences of Theorem \ref{result2}.

\paragraph{Acknowledgment}
The second author acknowledges financial support by Campus France, Agence nationale des bourses du Gabon, during his PhD thesis which this article is part from.
Both authors acknowledge the use of the IA Claude Opus 5 at the final stages of editorial polishing of the article, in particular for the correct formulation of the part Remark \ref{rem-grid-basics} --- Lemma \ref{oh}.

\paragraph{Declaration of interest} None.

\paragraph{Competing interests} The authors declare that they have no competing interests.

\paragraph{Data availability} No data sets were generated during this study.

\footnotesize{
\noindent J\'er\'emie Moukambi\\
\noindent
Universit\'e Clermont Auvergne,\\
CNRS,\\
LMBP,\\
F-63000 CLERMONT-FERRAND,\\
FRANCE \\
mukambijeremie96@gmail.com\hskip.3cm
}

\vspace{0.2cm}
\footnotesize{
\noindent Christoph Kriegler\\
\noindent
Universit\'e Clermont Auvergne,\\
CNRS,\\
LMBP,\\
F-63000 CLERMONT-FERRAND,\\
FRANCE \\
URL: \href{http://lmbp.uca.fr/~kriegler/indexenglish.html}{http://lmbp.uca.fr/{\raise.17ex\hbox{$\scriptstyle\sim$}}\hspace{-0.1cm} kriegler/indexenglish.html}\\
christoph.kriegler@uca.fr\hskip.3cm \\
ORCID: 0000-0001-8120-6251
}
\end{document}